\documentclass[preprint,11pt]{elsarticle}
\usepackage[a4paper,margin=0.6in]{geometry}
\usepackage{amssymb}
\usepackage{amsmath,amsthm}
\usepackage{graphicx}
\usepackage{tikz}
\usepackage{tikz-cd}
\usepackage{pgfplots}
\pgfplotsset{compat=1.18}
\usepackage{array}
\usepackage{algorithm}
\usepackage{algpseudocode}
\usepackage[most]{tcolorbox}
\usepackage[colorlinks=true,allcolors=blue]{hyperref}
\graphicspath{{figures/}}

\numberwithin{equation}{section}
\newtheorem{theorem}{Theorem}[section]
\newtheorem{proposition}[theorem]{Proposition}
\newtheorem{corollary}[theorem]{Corollary}

\theoremstyle{remark}

\newtcolorbox{statebox}{enhanced,colback=blue!4,colframe=blue!40!black,arc=3pt,boxrule=0.6pt,left=5pt,right=5pt,top=3pt,bottom=3pt,before skip=6pt,after skip=6pt}

\newcommand{\zenodoDOI}{10.5281/zenodo.21045015}
\newcommand{\tr}{\mathrm{tr}}
\newcommand{\proj}{\mathrm{proj}}

\journal{Computer Physics Communications}

\begin{document}

\begin{frontmatter}

\title{Vector alignment in matrix Lie groups}

\author[a]{Congzhou M. Sha\corref{author}}
\cortext[author]{Corresponding author.\\\textit{E-mail address:} consha@sas.upenn.edu}
\address[a]{Penn Medicine Doylestown Health, 595 W State St, Doylestown, PA 18901, USA}

\begin{abstract}
Many measurements in physics determine a linear transformation from pairs of input and output vectors, where the transformation is known in advance to belong to a matrix Lie group, for example the Lorentz group for the four-momenta measured by two inertial observers, or the unitary group for the states produced by a quantum circuit. Among the classical groups, only the rotation and unitary groups have a well-known closed form solution, given by the Kabsch and Horn algorithms and by the polar decomposition. We extend our previously proposed Lie algebra method from the image of the matrix exponential to the entire Lie group. We derive the projection in closed form for the classical groups over the real and complex numbers, including the indefinite orthogonal and unitary, symplectic, spin, special Euclidean and similarity groups, and we give conditions under which the method is exact for noiseless data, with examples in which it fails otherwise and an extension by a data dependent involution which is exact on the whole group. For noisy data, we bound the error of the method to first order in the noise, and we show that, for the orthogonal and unitary groups, it agrees with the least squares solution over the group to first order when the input vectors are isotropic, but not otherwise. We apply the method to an eight-qubit quantum circuit and to the rigid registration of point clouds, and compare it with Gauss--Newton refinement and direct minimization. A Julia implementation is provided.
\end{abstract}


\end{frontmatter}

\noindent\textbf{PROGRAM SUMMARY}\par\smallskip
\begin{small}
\noindent
{\em Program Title:} FastLieAlign\\
{\em CPC Library link to program files:} (to be added by Technical Editor)\\
{\em Code Ocean capsule:} (to be added by Technical Editor)\\
{\em Developer's repository link:} \href{https://github.com/mikesha2/FastLieAlign.jl}{https://github.com/mikesha2/FastLieAlign.jl}\\
{\em Licensing provisions:} MIT\\
{\em Programming language:} Julia\\
{\em Supplementary material:} Proofs and error histograms for eighteen groups (PDF)\\
{\em External routines/libraries:} Julia standard library LinearAlgebra; ForwardDiff.jl and ReverseDiff.jl for the comparison with direct minimization; Plots.jl, StatsPlots.jl, LaTeXStrings.jl, ProgressBars.jl and HypothesisTests.jl for the benchmarks\\
{\em Nature of problem:} Many measurements in physics determine a linear transformation from pairs of input and output vectors, where the transformation is known in advance to belong to a matrix Lie group, e.g. the Lorentz transformation between two inertial frames or the unitary implemented by a quantum circuit. The unconstrained least squares fit of noisy data does not belong to the group, so it violates the physical law which defines the group, and for noncompact groups such as the Lorentz and symplectic groups there is no closed form solution analogous to the Kabsch algorithm for rotations.\\
{\em Solution method:} The program computes the unconstrained least squares matrix with the Moore--Penrose pseudoinverse, factors out an involution of the group built from its eigenvalues with negative real part (or, for the complex special groups, a reflection of an even number of eigenvalues, which makes the logarithm traceless), takes the principal matrix logarithm, projects the logarithm orthogonally onto the Lie algebra of the group with a closed form formula, exponentiates the result and multiplies back, which returns an element of the group and recovers every element exactly from noiseless data. An optional Gauss--Newton iteration on the group refines the estimate toward the least squares solution over the group. The supported groups are $GL$, $SL$, $SO$, $U$, $SU$, $SO(p,q)$, $Sp$, $\mathrm{Spin}$ and $SE$ over the real and complex numbers.\\
{\em Additional comments including restrictions and unusual features:} The method recovers the true transformation exactly from noiseless data whenever the principal logarithm of the true element lies in the Lie algebra (\S\ref{sec:exact}), and the branch-cut aware method, which factors out an involution computed from the eigenvalues of the data, does so on the whole group (Theorem~\ref{thm:global}). The archive contains scripts which reproduce every figure and table.
\end{small}

\section{Introduction}
\label{sec:intro}
Two inertial observers who measure the four-momenta of the same particles are related by a Lorentz transformation $\Lambda$, with $v_{B,i}=\Lambda v_{A,i}$ for noiseless data. In the setting of noise, one instead asks for the Lorentz transformation which best explains the measured pairs. In \cite{sha2026}, we showed that the orthogonal Procrustes solutions of Sch\"onemann, Kabsch and Horn \cite{schonemann1966,kabsch,horn} for rotations do not extend to this problem and proposed a Lie algebra method: (1) compute the unconstrained least squares matrix, (2) take its matrix logarithm, (3) project the logarithm onto the Lie algebra $\mathfrak{so}(3,1)$, and (4) exponentiate, so that the estimate is a Lorentz transformation by construction (Algorithm~\ref{alg:naive}).
\par The same problem structure as in the Lorentz example exists for reconciling measurements in other areas of physics (Table~\ref{tab:applications}). Generally, for pairs of vectors $(v_{A,i},v_{B,i})$ measured in the fundamental representation of some Lie group $G$, the goal is to find $g\in G$ such that $v_{B,i}\approx g\,v_{A,i}$. In a quantum computer, the output states of a circuit on $n$ qubits for the computational basis states are the columns of a unitary $U\in U(2^n)$ \citep{nielsenChuang}, and the circuit and its generating Hamiltonian are recovered from them. In lattice field theory, the links $U_\mu(x)\in SU(3)$ transport color vectors between neighboring sites \citep{gattringerLang}, two gauge equivalent configurations are related by a local gauge transformation $g(x)\in SU(3)$, and recovering $g(x)$ from color vectors at a site is alignment in $SU(3)$.
\par The original method proposed in \cite{sha2026} is restricted to the identity component of the Lie group, with restrictions on reachable matrix exponentials for real and complex Lie groups with negative eigenvalues, and further restrictions for unit determinant complex Lie groups. In this work, we carefully treat the eigenspectrum and derive a method whose output reaches every element of the Lie groups considered.
\par First, we derive the projection step from the standard construction of the Lie algebra (\S\ref{sec:projections}). Second, we determine the branch of the logarithm, which the earlier work left to the principal branch, by prefactoring an involution out of the unconstrained least squares matrix of step (1). We refer to the original method as the \textbf{naive Lie algebra method} (Algorithm~\ref{alg:naive}), and the refined algorithm as the \textbf{branch-cut aware method} (Algorithm~\ref{alg:global}). We also write out the explicit derivatives for Gauss--Newton refinement of either estimate. Exactness of the algorithm on noiseless data is shown in \S\ref{sec:exact}, and its surjectivity to the entire Lie group is addressed in \S\ref{sec:components}. We derive the first order error of the method under noise, and the subsequent Gauss--Newton refinement (\S\ref{sec:noise}), and concrete benchmarks in eleven Lie groups are performed (\S\ref{sec:benchmarks}).

\begin{table}[t]
\centering
\caption{Examples of vector alignment in physics. In each case, the measured or computed vectors $v_{A,i}$ and $v_{B,i}$ are related by an unknown element of the listed group, and the task is to recover that element. The last column gives the section of this paper in which the example is studied, where applicable.}
\label{tab:applications}
\footnotesize
\setlength{\tabcolsep}{4pt}
\begin{tabular}{>{\raggedright\arraybackslash}p{0.13\linewidth}>{\raggedright\arraybackslash}p{0.23\linewidth}>{\raggedright\arraybackslash}p{0.14\linewidth}>{\raggedright\arraybackslash}p{0.27\linewidth}l}
\hline
Field & Aligned vectors & Group & Task & Section\\
\hline
Relativity & Four-momenta of the same particles in two inertial frames, or tangent vectors in neighboring freely falling frames & $SO(3,1)^0$ & Lorentz transformation between the frames \citep{sha2026}; connection between the cells of a smooth lattice \citep{brewin2012} & \ref{sec:benchmarks}\\
Particle physics & Spinors of the same particles in two frames & $SL(2,\mathbb C)$ & Spinor transformation covering a Lorentz transformation \citep{Peskin:1995ev} & \\
Lattice gauge theory & Color vectors at a lattice site, or a smeared link with $X=I$ & $SU(3)$ & Local gauge transformation, projection of a smeared link \citep{gattringerLang,albanese1987,morningstarPeardon2004} & \ref{sec:benchmarks}\\
Quantum computing & Output states of a circuit for a set of basis states & $U(2^n)$, $SU(2^n)$ & Recover the unitary and decompose it into Pauli terms \citep{nielsenChuang} & \ref{sec:qubits}\\
Ray optics & Positions and angles of paraxial rays & $Sp(2,\mathbb R)$, $Sp(4,\mathbb R)$ & Ray transfer matrix of an optical system \citep{siegman1986} & \\
Structural biology, rigid body mechanics & Atomic coordinates of two copies of a molecule, or points on a rigid body before and after a motion & $SO(3)$, $SE(3)$, $\mathrm{Sim}(3)$ & Superposition of the two structures \citep{kabsch}; registration of point clouds \citep{beslMcKay1992} & \ref{sec:results-se}\\
\hline
\end{tabular}
\end{table}

\section{The Lie algebra method}
\label{sec:method}
\noindent The definitions of matrix Lie groups and algebras, of the groups preserving a form, of the matrix exponential and of the principal logarithm, with the notation used below, are collected in \ref{sec:background}. Derivations and proofs are in the Supplementary Material.
\subsection{Problem statement and algorithm}
Collect the observations as the columns of the $n\times m$ matrices $X=[v_{A,1}\cdots v_{A,m}]$ and $Y=[v_{B,1}\cdots v_{B,m}]$, so that noiseless data satisfy $Y=gX$.
\begin{statebox}
\emph{Vector alignment in a matrix Lie group.} Given $X$, $Y$ and $G$, find $g\in G$ such that $\|gX-Y\|_F$ is small, and such that $g=g_\ast$ exactly when $Y=g_\ast X$.
\end{statebox}
We assume throughout that $X$ has full row rank $n$, which holds with probability one for $m\ge n$ measured vectors with continuously distributed noise. Then the least squares matrix over all matrices is $g_0=YX^+$, so that:
\begin{equation}
g_0=YX^+,\qquad L_0=\log g_0,\qquad L=\proj_{\mathfrak g}(L_0),\qquad g=\exp(L).
\label{eq:method}
\end{equation}
Eq.~\eqref{eq:method} stated as Algorithm~\ref{alg:naive} is the \textbf{naive Lie algebra method} \cite{sha2026}. It provides the correct answer when the principal logarithm of $g_0$ lies in $\mathfrak g$, which \S\ref{sec:exact} characterizes.
\par The \textbf{branch-cut aware method}, Algorithm~\ref{alg:global}, first factors out of $g_0$ an involution $r$ of the group. For the real groups, $r$ is built from the eigenvalues of $g_0$ with negative real part, and for $SL(n,\mathbb C)$ and $SU(n)$ it is a reflection of an even number of eigenvalues chosen so that the logarithm becomes traceless. The method then applies Eq.~\eqref{eq:method} to $r^{-1}g_0$ and multiplies the result by $r$. This makes it exact on the whole group (\S\ref{sec:components}), and it reduces to the naive method when $r=I$, as for the near-identity samples of \S\ref{sec:accuracy}. \S\ref{sec:global} tests it on elements of the other components and other Riemann sheets.
\par The next sections derive the projection $\proj_{\mathfrak g}$ and show that Algorithm~\ref{alg:global} is correct for every $g_\ast\in G$.

\begin{algorithm}[t]
\caption{The naive Lie algebra method, Eq.~\eqref{eq:method} (\texttt{naive\_align})}
\label{alg:naive}
\begin{algorithmic}[1]
\Require data $X,Y$ with $X$ of full row rank; the group $G$ with its projection $\proj_{\mathfrak g}$ (Table~\ref{tab:projections})
\Procedure{NaiveAlign}{$X,Y,G$}
\State $g_0\gets YX^+$ \Comment{least squares over all matrices}
\State $L_0\gets\log g_0$ \Comment{principal logarithm; undefined if $g_0$ has a negative real eigenvalue}
\State $L\gets\proj_{\mathfrak g}(L_0)$
\State \Return $\exp(L)$
\EndProcedure
\end{algorithmic}
\end{algorithm}
\begin{algorithm}[t]
\caption{The branch-cut aware Lie algebra method (\texttt{branch\_cut\_align}), with the equations of \S\ref{sec:components}}
\label{alg:global}
\begin{algorithmic}[1]
\Require data $X,Y$ with $X$ of full row rank; the group $G$ with its projection $\proj_{\mathfrak g}$
\Procedure{BranchCutAlign}{$X,Y,G$}
\State $g_0\gets YX^+$
\If{$G$ is $\mathrm{Spin}(n)$ or $\mathrm{Pin}(n)$} \Return \Call{SpinAlign}{$X,Y$} \Comment{Algorithm~\ref{alg:spin}}
\EndIf
\State $r\gets$ \Call{SchurReflection}{$g_0,G$} \Comment{Algorithm~\ref{alg:schur}}
\State \Return $r\exp\big(\proj_{\mathfrak g}\log(r^{-1}g_0)\big)$ \Comment{Eq.~\eqref{eq:globalreal}}
\EndProcedure
\end{algorithmic}
\end{algorithm}
\begin{algorithm}[t]
\caption{The involution $r$ from the Schur decomposition of $g_0$ (\texttt{negative\_involution}, \texttt{even\_reflection}), called by Algorithms~\ref{alg:global} and~\ref{alg:spin}}
\label{alg:schur}
\begin{algorithmic}[1]
\Require $g_0$; the group $G$ with the form $M$ and the subspace $W$ of Table~\ref{tab:W}, with $M=I$ when $G$ preserves no form
\Procedure{SchurReflection}{$g_0,G$}
\If{$G$ is $E(n)$, $SE(n)$, $\mathrm{Sim}(n)$, $E(p,q)$ or $SE(p,q)$} \Return $\mathrm{diag}(r_R,1)$ with $r_R\gets$ \Call{SchurReflection}{rotation or Lorentz block of $g_0$, that group}
\EndIf
\State $g_0=ZTZ^\dagger$ \Comment{Schur decomposition, the real Schur form for a real $G$}
\If{$G$ is $SL(n,\mathbb C)$ or $SU(n)$}
\State $S\gets$ the eigenvalues of largest argument satisfying Eq.~\eqref{eq:flipcount}
\Else
\State $S\gets\{\lambda:\mathrm{Re}\,\lambda<0\}$ \Comment{Eq.~\eqref{eq:involution}; a conjugate pair is one $2\times2$ block of $T$}
\EndIf
\If{$S=\emptyset$} \Return $I$ \EndIf
\State reorder the Schur form so that the eigenvalues in $S$ come first; $Q\gets$ the leading $|S|$ columns of $Z$
\State $P\gets Q\big(Q^\star MQ\big)^{-1}Q^\star M$ \Comment{Eq.~\eqref{eq:sign}; $P=QQ^\dagger$ for $M=I$}
\State \Return $I-2P$
\EndProcedure
\end{algorithmic}
\end{algorithm}

\begin{algorithm}[t]
\caption{The branch-cut aware method for $\mathrm{Spin}(n)$ and $\mathrm{Pin}(n)$ (\texttt{spin\_align}), called by Algorithm~\ref{alg:global}}
\label{alg:spin}
\begin{algorithmic}[1]
\Require data $X,Y$ with $X$ of full row rank; the generators $\gamma_1,\dots,\gamma_n$ of the spinor representation; $\alpha(g)=g$ for $\mathrm{Spin}(n)$ and $\alpha(g)=\Gamma g\Gamma$ for $\mathrm{Pin}(n)$
\Procedure{SpinAlign}{$X,Y$}
\State $g_0\gets YX^+$
\State $(R_0)_{ab}\gets\tr(\gamma_a\,\alpha(g_0)\gamma_b g_0^{-1})/d$ \Comment{Eq.~\eqref{eq:covering}}
\State $r_R\gets$ \Call{SchurReflection}{$R_0,SO(n)$ or $O(n)$} \Comment{Algorithm~\ref{alg:schur}}
\State $A\gets\proj_{\mathfrak{so}(n)}\log(r_R^{-1}R_0)$
\State $\tilde r\gets\tilde u_1\cdots\tilde u_k$ as in Eq.~\eqref{eq:spinlift}
\State $B\gets\tfrac14\sum_{a,b}A_{ab}\,\gamma_a\gamma_b$ \Comment{Eq.~\eqref{eq:spinlift}}
\State $\zeta\gets\mathrm{sign}\,\mathrm{Re}\,\tr\big((\tilde re^{B}X)^\dagger Y\big)$
\State \Return $\zeta\,\tilde r\exp\big(B+\proj_{\mathfrak g}\log(e^{-B}\tilde r^{-1}\zeta g_0)\big)$ \Comment{Eq.~\eqref{eq:spinout}}
\EndProcedure
\end{algorithmic}
\end{algorithm}

\subsection{Projections onto the classical Lie algebras}
\label{sec:projections}
For every group of the form of Eq.~\eqref{eq:formgroup} with a unitary $M$, the orthogonal projection generalizes Eq.~\eqref{eq:proj2x2}. The Lie algebra $\mathfrak g_M$ consists of the matrices $L$ with
\begin{equation}
L^\star M+ML=0,\qquad\text{that is}\qquad L=-\tau(L),\quad \tau(L)=M^{-1}L^\star M.
\label{eq:algcond}
\end{equation}
Since $M^\star=\pm M$, $\tau^2$ is the identity, and since $M$ is unitary, $\tau$ preserves the inner product of Eq.~\eqref{eq:inner}. The eigenspaces of $\tau$ for $-1$ and $+1$ are therefore orthogonal, and the first is $\mathfrak g_M$, so the orthogonal projection is the $-1$ part of $L$.
\begin{theorem}[Projection for groups preserving a form]\label{thm:formproj}
Let $M$ be unitary with $M^\star=\pm M$, and let $\mathfrak g_M$ be as in Eq.~\eqref{eq:formgroup}. Then for every $L$,
\begin{equation}
\proj_{\mathfrak g_M}(L)=\frac12\left(L-\tau(L)\right)=\frac12\left(L-M^{-1}L^\star M\right).
\label{eq:formproj}
\end{equation}
\end{theorem}
\par The projection commutes with conjugation by the group: for $g\in G_M$, $(g^{-1}Lg)^\star=g^\star L^\star g^{-\star}$ and $M^{-1}g^\star=g^{-1}M^{-1}$, so
\begin{equation}
\tau(g^{-1}Lg)=g^{-1}\tau(L)g,\qquad \proj_{\mathfrak g_M}(g^{-1}Lg)=g^{-1}\proj_{\mathfrak g_M}(L)\,g,
\label{eq:equivariant}
\end{equation}
so the Lie algebra method, whose logarithm and exponential also commute with conjugation, gives the same estimate in any two coordinate systems related by an element of $G_M$. The hypotheses hold for $M=I$, $\eta$ and $J$, which give the rows of Table~\ref{tab:projections}, with the trace removed for the special groups. For $\mathrm{Spin}(n)$, the projection is the expansion in the orthonormal basis $B_k$ of bivectors, the normalized products $\gamma_i\gamma_j$, $i<j$, of a Clifford representation, and for the Euclidean and similarity groups it acts blockwise, with the translation column kept and the rotation block projected onto $\mathfrak{so}(n)$, plus $\frac{\tr A}{n}I$ for the scaled rigid registration group $\mathrm{Sim}(n)$ \citep{umeyama1991}.
\begin{table}[t]
\centering
\caption{Orthogonal projections onto the Lie algebras of the classical groups, with respect to the inner product of Eq.~\eqref{eq:inner}. Here $\eta=\mathrm{diag}(I_p,-I_q)$, $J=\left[\begin{smallmatrix}0&I\\-I&0\end{smallmatrix}\right]$, $\{B_k\}$ is the orthonormal basis of bivectors given in the Supplementary Material, and $L=\left[\begin{smallmatrix}A&v\\w^T&c\end{smallmatrix}\right]$ in the last row. The rows for the groups preserving a form follow from Theorem~\ref{thm:formproj}, the trace rows from the orthogonality of $I$ to the traceless matrices, the Euclidean and similarity rows from the orthogonality of the blocks, and the spin row is proven in the Supplementary Material.}
\label{tab:projections}
\begin{tabular}{lll}
\hline
Group & Lie algebra & Projection of $L$\\
\hline
$GL(n,\mathbb F)$ & all matrices & $L$\\
$SL(n,\mathbb F)$ & $\tr A=0$ & $L-\frac{\tr L}{n}I$\\
$SO(n)$, $SO(n,\mathbb C)$ & $A^T=-A$ & $\frac12(L-L^T)$\\
$U(n)$ & $A^\dagger=-A$ & $\frac12(L-L^\dagger)$\\
$U(p,q)$ & $A^\dagger\eta+\eta A=0$ & $\frac12(L-\eta L^\dagger\eta)$\\
$SU(n)$ & $A^\dagger=-A$, $\tr A=0$ & $S-\frac{\tr S}{n}I$, $S=\frac12(L-L^\dagger)$\\
$SO(p,q)$ & $A^T\eta+\eta A=0$ & $\frac12(L-\eta L^T\eta)$\\
$Sp(2n,\mathbb F)$ & $A^TJ+JA=0$ & $\frac12(L+JL^TJ)$\\
$\mathrm{Spin}(n)$, $\mathrm{Pin}(n)$ & span of $\gamma_i\gamma_j$ & $\sum_k\langle B_k,L\rangle B_k$\\
$SE(n)$ & $\left[\begin{smallmatrix}a&v\\0&0\end{smallmatrix}\right]$, $a\in\mathfrak{so}(n)$ & $\left[\begin{smallmatrix}\frac12(A-A^T)&v\\0&0\end{smallmatrix}\right]$\\
$\mathrm{Sim}(n)$ & $\left[\begin{smallmatrix}cI+a&v\\0&0\end{smallmatrix}\right]$, $a\in\mathfrak{so}(n)$, $c\in\mathbb R$ & $\left[\begin{smallmatrix}\frac12(A-A^T)+\frac{\tr A}{n}I&v\\0&0\end{smallmatrix}\right]$\\
\hline
\end{tabular}
\end{table}

\subsection{When is the naive method exact?}
\label{sec:exact}
For noiseless data, $g_0=g_\ast$, so Algorithm~\ref{alg:naive} returns $g_\ast$ exactly when the principal logarithm of $g_\ast$ exists and lies in $\mathfrak g$. It can fail in two ways. Either (1) an eigenvalue of $g_\ast$ lies on the branch cut, the negative real axis, and the principal logarithm does not exist, or (2) the principal logarithm exists, but the correct logarithm lies on another Riemann sheet so that the projection returns the wrong element. Table~\ref{tab:failures} shows both. The samples of the accuracy benchmarks in \S\ref{sec:accuracy}, drawn as $\exp(sA)$ with $A\in\mathfrak g$ and a small $s$, are recovered exactly, because $sA$ is then the principal logarithm. \S\ref{sec:global} tests elements outside this set, which require Algorithm~\ref{alg:global}.
\par For the real groups, and for $O(n,\mathbb C)$, $SO(n,\mathbb C)$ and $Sp(2n,\mathbb C)$, only the first failure occurs, because the principal logarithm of $g_\ast$, whenever it exists, lies in $\mathfrak g$: it is real for a real matrix, since the principal branch sends conjugate eigenvalues to conjugate logarithms \citep[Theorem 1.31]{higham2008}, it satisfies the form condition for every group preserving a form \citep{hmmt2005}, and it is traceless for $SL(n,\mathbb R)$ because $\tr\log g_\ast=\log\det g_\ast=0$. In these groups, Algorithm~\ref{alg:naive} is therefore exact if and only if no eigenvalue of $g_\ast$ is a negative real number, and the first three rows of Table~\ref{tab:failures} show what happens otherwise. For $GL(n,\mathbb C)$, $U(n)$ and $U(p,q)$, the first failure is milder: a negative real eigenvalue leaves the principal logarithm undefined, but either choice $\pm i\pi$ of its logarithm lies in the algebra, in $U(p,q)$ because the eigenvalues pair as $\lambda$ and $1/\bar\lambda$ with the same argument, so a logarithm of $g_\ast$ in $\mathfrak g$ exists, with the cut along any ray from the origin avoiding the spectrum, and the exponential map is onto these three groups. Under noise, the eigenvalue splits into a pair on both sides of the branch cut, and Algorithm~\ref{alg:global} is needed as for the other groups.
\par For $SL(n,\mathbb C)$, $SU(n)$, $\mathrm{Spin}(n)$ and $\mathrm{Pin}(n)$, the second failure occurs as well. For a complex group with a determinant condition, the trace of the principal logarithm is $i\sum_j\arg\lambda_j$, which need not vanish, and then the projection removes the trace and returns the wrong element with no eigenvalue on the branch cut. The fourth row of Table~\ref{tab:failures} is $\omega I\in SU(3)$, whose correct logarithm lies on another Riemann sheet. $\mathrm{Spin}(n)$ and $\mathrm{Pin}(n)$ inherit this failure through their complex spinor representations.
\par To address the cases in which the naive Lie algebra method may fail, the idea is to factor out the negative eigenvalues, $g_0=rg'$, so that $\log g'\in\mathfrak g$.
\begin{table}[t]
\centering
\caption{Failures of the naive Lie algebra method (Algorithm~\ref{alg:naive}) on noiseless data with $X=I$, so that $g_0=g_\ast$. In the first three rows an eigenvalue of $g_\ast$ lies on the branch cut and the principal logarithm does not exist. In the last two it exists but lies on the wrong Riemann sheet for the group, and the projection discards the part of it outside $\mathfrak g$. Here $N_2=\left[\begin{smallmatrix}0&1\\0&0\end{smallmatrix}\right]$, $K=\left[\begin{smallmatrix}0&-1\\1&0\end{smallmatrix}\right]$, $\omega=e^{2\pi i/3}$, $S_{ij}=\frac12\gamma_i\gamma_j$ and $\Gamma=\gamma_1\gamma_2\cdots\gamma_6$. All five are recovered exactly by Algorithm~\ref{alg:global}.}
\label{tab:failures}
\footnotesize
\setlength{\tabcolsep}{3.5pt}
\renewcommand{\arraystretch}{1.3}
\begin{tabular}{lllll}
\hline
Group & $g_\ast$ & Eigenvalues & Principal logarithm & Output $g$, error\\
\hline
$SO(2)$ & $R(\pi)=-I_2$ & $-1,-1$ & undefined; $\pm\pi K$ equally close & undefined\\
$SL(2,\mathbb R)$ & $-I-N_2$ & $-1,-1$ & none real; not an exponential, Eq.~\eqref{eq:notexp} & undefined\\
$GL^+(2,\mathbb R)$ & $\mathrm{diag}(-1,-2)$ & $-1,-2$ & none real & undefined\\
$SU(3)$ & $\omega I$ & $\omega,\omega,\omega$ & $\tfrac{2\pi i}{3}I$, trace $2\pi i$; projects to $0$ & $g=I$, error $3$\\
$\mathrm{Spin}(6)$ & $e^{0.8\pi(S_{12}+S_{34}+S_{56})}$ & $e^{\pm1.2\pi i}$, $e^{\pm0.4\pi i}$ thrice & $-0.2\pi(S_{12}+S_{34}+S_{56})+\tfrac{\pi}{2}\Gamma$ & $g=-\Gamma g_\ast$, error $4$\\
\hline
\end{tabular}
\end{table}
\subsection{Covering the whole group}
\label{sec:components}
Let $r$ be an element of $G$, computed from $g_0$, such that $r^{-1}g_0$ has no eigenvalue on the branch cut. Then
\begin{equation}
g=r\exp\big(\proj_{\mathfrak g}(\log(r^{-1}g_0))\big),
\label{eq:globalreal}
\end{equation}
with $g\in G$ by construction, is the approximate solution that we seek in the vector alignment problem, where the factor $r^{-1}$ rotates the eigenvalues of $g_0$ onto the Riemann sheet covering the positive real axis, and $r$ restores the rotation afterwards.
\par In other words, $r$ is a Householder-like reflection across the invariant subspace $W$ of $g_0$ for the eigenvalues to be negated, orthogonal with respect to the form $M$. The next theorem states it, and \ref{sec:reflection} derives it.
\begin{theorem}[The reflection $r$]\label{thm:reflection}
Let $M$ be invertible with $M^\star=\pm M$, let $W$ be a subspace on which the form $\langle u,v\rangle=u^\star Mv$ is nondegenerate, with orthogonal complement $W^\perp=\{v:\langle w,v\rangle=0\text{ for all }w\in W\}$, and let $Q$ be a matrix whose columns are a basis of $W$. Suppose that $r$ negates $W$ and fixes a complement of $W$, and that $r$ is a member of the group,
\begin{equation}
r=-I\ \text{on }W,\qquad r=I\ \text{on a complement of }W,\qquad r^\star Mr=M.
\label{eq:rjobs}
\end{equation}
Then the complement is $W^\perp$, and $r$ is uniquely determined,
\begin{equation}
r=I-2P,\qquad P=Q\big(Q^\star MQ\big)^{-1}Q^\star M,
\label{eq:sign}
\end{equation}
where $P$ is the projector onto $W$ along $W^\perp$ and satisfies $P^2=P$ and $P^\star M=MP$. Conversely, the $r$ of Eq.~\eqref{eq:sign} satisfies Eq.~\eqref{eq:rjobs} and $r^2=I$. If $W$ is the invariant subspace of some $g\in G_M$ belonging to a set of its eigenvalues, then $r$ commutes with $g$, and $rg$ has those eigenvalues negated and the others unchanged.
\end{theorem}
Numerically (Algorithm~\ref{alg:schur}), $Q$ consists of the leading columns of the unitary factor $Z$ of the Schur decomposition $g_0=ZTZ^\dagger$, reordered so that the eigenvalues to be negated come first, and it is real for a real $g_0$ with the real Schur form, whose $2\times2$ blocks for conjugate pairs are negated together. This costs one Schur decomposition, $O(n^3)$ operations, and is stable when eigenvalues nearly coincide and individual eigenvectors are ill conditioned.
\par Since $P^\star M=MP$, $r\in G$ for noisy data as well, when $g_0\notin G$. For exact data $g_0\in G_M$, the invariant subspaces of different eigenvalues are orthogonal for the form, Eq.~\eqref{eq:eigorth} in \ref{sec:reflection}, so $rg_0$ has the eigenvalues in $W$ negated and the others unchanged, the last statement of Theorem~\ref{thm:reflection}, and under noise $rg_0$ is within the noise of this. The groups differ only in the choice of $W$ (Table~\ref{tab:W}).
\begin{table}[t]
\centering
\caption{The subspace $W$ reflected by $r=I-2P$ in Eq.~\eqref{eq:sign}, for each group. $N$ denotes the invariant subspace of $g_0$ for the eigenvalues with $\mathrm{Re}\,\lambda<0$, Eq.~\eqref{eq:involution}, and $S$ a set of eigenvalues of $g_0$ satisfying the trace condition of Eq.~\eqref{eq:flipcount}, taken among those of largest argument. The last column is the form $M$ entering $P$, or the replacement for $P$ where there is no form.}
\label{tab:W}
\footnotesize
\begin{tabular}{>{\raggedright\arraybackslash}p{0.25\linewidth}>{\raggedright\arraybackslash}p{0.25\linewidth}>{\raggedright\arraybackslash}p{0.36\linewidth}}
\hline
Group & Subspace $W$ & Form in $P$\\
\hline
$O(n)$, $SO(n)$, $U(n)$ & $N$ & $M=I$, $P=QQ^\dagger$\\
$O(n,\mathbb C)$, $SO(n,\mathbb C)$ & $N$ & $M=I$ with $\star=T$,\newline $P=Q(Q^TQ)^{-1}Q^T$\\
$O(p,q)$, $SO(p,q)$, $U(p,q)$ & $N$ & $M=\eta$\\
$Sp(2n,\mathbb F)$ & $N$ & $M=J$\\
$GL(n,\mathbb R)$, $SL(n,\mathbb R)$, $GL(n,\mathbb C)$ & $N$ & no form; $M=I$ in Eq.~\eqref{eq:sign}, $P=QQ^\dagger$ with the Schur vectors $Q$ of $N$, so $r$ is a unitary reflection\\
$E(n)$, $SE(n)$, $\mathrm{Sim}(n)$, $E(p,q)$, $SE(p,q)$ & $N$ of the rotation or Lorentz block & as for that block; $r$ is extended by $1$ on the translation\\
$SL(n,\mathbb C)$, $SU(n)$ & span of the Schur vectors of $S$ & $M=I$, $P=Q_SQ_S^\dagger$\\
$\mathrm{Spin}(n)$, $\mathrm{Pin}(n)$ & $N$ of the covered $R_0\in SO(n)$, resp.\ $O(n)$ & $M=I$ in $SO(n)$ or $O(n)$; $r_R$ is lifted to $\tilde r$ by Eq.~\eqref{eq:spinlift}\\
\hline
\end{tabular}
\end{table}
\subsection{Groups preserving a form and the real linear groups}
\label{sec:formgroups}
For $GL(n,\mathbb R)$, $SL(n,\mathbb R)$, $O(n)$, $SO(n)$, $O(p,q)$, $SO(p,q)$, $Sp(2n,\mathbb R)$, $E(n)$, $SE(n)$, $\mathrm{Sim}(n)$, $E(p,q)$, $SE(p,q)$, $GL(n,\mathbb C)$, $U(n)$, $U(p,q)$, $O(n,\mathbb C)$, $SO(n,\mathbb C)$ and $Sp(2n,\mathbb C)$, $W$ is the invariant subspace of $g_0$ for the eigenvalues in the open left half plane,
\begin{equation}
W=N=\text{the invariant subspace of }g_0\text{ for }\{\lambda:\mathrm{Re}\,\lambda<0\},
\label{eq:involution}
\end{equation}
so that $r$ is the matrix sign function of $g_0$ when $g_0\in G$ has no eigenvalue on the imaginary axis \citep{higham2008}, and $rg_0$ has no eigenvalue in the open left half plane, so none on the branch cut (Fig.~\ref{fig:branch}). For example, $r=-I$ for $R(\pi)$, $-I-N_2$ and $\mathrm{diag}(-1,-2)$. For $GL(n,\mathbb R)$, $SL(n,\mathbb R)$ and $GL(n,\mathbb C)$, which preserve no form, Eq.~\eqref{eq:sign} is used with $M=I$, so $r=I-2QQ^\dagger$ is a unitary reflection, $\mathrm{diag}(-I,I)$ in the Schur basis of $g_0$, and $rg_0$ is block triangular with the leading block negated, which negates exactly the eigenvalues in $N$. In every case $r\in G$, real for real $g_0$, with $\det r=(-1)^{\dim N}$. This is $1$ when $\det g_0>0$, since the eigenvalues of a real $g_0$ in the open left half plane are conjugate pairs or an even number of negative reals, and $-1$ for an element of $O(n)$, $O(p,q)$ or $GL(n,\mathbb R)$ with negative determinant, which is how $r$ reaches the other component. Hence $rg_\ast\in G$ has no eigenvalue on the branch cut, $\log(rg_\ast)\in\mathfrak g$ by \S\ref{sec:exact}, and Eq.~\eqref{eq:globalreal} returns $g_\ast$. \ref{sec:examples} works through a numerical example in $SO(3)$.
\subsection{The complex special groups}
\label{sec:complexspecial}
For $SU(n)$ and $SL(n,\mathbb C)$, the Lie algebra consists of traceless matrices, but the principal logarithm of a group element need not be traceless: every logarithm $L$ of $g$ with $\det g=1$ has $e^{\tr L}=1$, so $\tr L=2\pi ik$ with $k\in\mathbb Z$, and for the principal logarithm $k=\frac1{2\pi}\sum_j\theta_j$ with $\theta_j\in(-\pi,\pi]$ the arguments of the eigenvalues, rounded to the nearest integer on noisy data. For example, $\omega I\in SU(3)$ with $\omega=e^{2\pi i/3}$ has $k=1$, and its principal logarithm $\frac{2\pi i}3 I$, with trace $2\pi i$, does not lie in $\mathfrak{su}(3)$.
\par Projecting such a logarithm onto the traceless matrices subtracts $\frac{2\pi ik}{n}I$ and returns $e^{-2\pi ik/n}g$ in place of $g$. Negating an eigenvalue $\lambda_i$ instead changes its argument from $\theta_i$ to $\theta_i+s_i\pi$ with $s_i=-\mathrm{sign}\,\theta_i$. For a set $S$ of eigenvalues with Schur vectors $Q_S$, Eq.~\eqref{eq:sign} with $M=I$ gives
\begin{equation}
r=I-2Q_SQ_S^\dagger,\qquad \det r=(-1)^{|S|},\qquad \tr\log(rg_0)=2\pi i\Big(k+\tfrac12\sum_{i\in S}s_i\Big),
\label{eq:flip}
\end{equation}
so the principal logarithm of $rg_0$ is traceless exactly when
\begin{equation}
\sum_{i\in S}s_i=-2k.
\label{eq:flipcount}
\end{equation}
The method takes $S$ to be the $b$ eigenvalues of largest positive argument and the $a$ eigenvalues of largest negative argument with $a-b=-2k$, choosing $a$ and $b$ to minimize $\sum_{i\in S}\pi(\pi-2|\theta_i|)$, the change in $\sum_i|\log\lambda_i|^2$ under the negation. Such $a$ and $b$ exist, since the arguments with the sign of $k$ sum to at least $2\pi|k|$ while each is at most $\pi$. Then Eq.~\eqref{eq:flipcount} holds, $|S|=a+b$ is even, so $r\in SU(n)\subset SL(n,\mathbb C)$, and since the cost of an eigenvalue is smallest on the branch cut and positive for $|\theta_i|<\pi/2$, $S$ contains every eigenvalue on the branch cut and no positive real one. Hence $\log(rg_\ast)\in\mathfrak g$ and Eq.~\eqref{eq:globalreal} returns $g_\ast$ (Theorem~\ref{thm:global}(b)). For $\omega I$, $S$ is two of the three copies of $\omega$ (Fig.~\ref{fig:branch}(b)), since negating all three would give $r=-I$ with $\det r=-1$. \ref{sec:examples} works through a numerical example in $SU(3)$.
\subsection{The spin groups}
\label{sec:spin}
$\mathrm{Spin}(n)$ is the double cover of $SO(n)$, a connected group with a two-to-one homomorphism onto $SO(n)$ whose kernel is $\{\pm I\}$, so that every rotation is covered by two spinors $\pm g$ \citep{lawsonMichelsohn}. We represent it by $d\times d$ matrices, $d=2^{\lfloor n/2\rfloor}$, generated by Hermitian matrices $\gamma_1,\dots,\gamma_n$ with $\gamma_a^2=I$ and $\gamma_a\gamma_b=-\gamma_b\gamma_a$ for $a\ne b$, and its Lie algebra is spanned by the products $\gamma_a\gamma_b$ (Table~\ref{tab:projections}). The homomorphism is conjugation on the generators, $g\gamma_bg^{-1}=\sum_aR_{ab}\gamma_a$ \citep[Chapter I, \S2]{lawsonMichelsohn}, and $\exp(\theta\gamma_a\gamma_b/2)$ covers the rotation by $\theta$ in the $(a,b)$ plane, so that a full turn of the rotation is half a turn of the spinor, $\exp(\pi\gamma_a\gamma_b)=-I$.
\par $\mathrm{Pin}(n)$, the double cover of $O(n)$, is generated by the unit vectors $\tilde u=\sum_au^a\gamma_a$, Hermitian with $\tilde u^2=I$ in our convention $\gamma_a^2=I$, where the opposite sign convention gives a group which covers $O(n)$ in the same way but is not isomorphic to this one in general. Its even products form $\mathrm{Spin}(n)$, and $\tilde u$ covers the reflection $I-2uu^T$ under the twisted homomorphism $\alpha(g)\gamma_bg^{-1}=\sum_aR_{ab}\gamma_a$, where $\alpha$ negates the odd elements \citep[Chapter I, \S2]{lawsonMichelsohn}. We realize $\alpha(g)=\Gamma g\Gamma$ with $\Gamma=\gamma'_{n+1}$ by taking $\gamma_a=i\gamma'_a\gamma'_{n+1}$ from a representation $\gamma'_1,\dots,\gamma'_{n+1}$ with one more generator, of dimension $2^{\lfloor(n+1)/2\rfloor}$, and for $\mathrm{Spin}(n)$, $\alpha$ is the identity. The method negates the eigenvalues in $SO(n)$ or $O(n)$ and lifts the result. Applied to $g_0$, the homomorphism gives
\begin{equation}
(R_0)_{ab}=\tfrac1d\,\tr\big(\gamma_a\,\alpha(g_0)\gamma_bg_0^{-1}\big),\qquad d=2^{\lfloor n/2\rfloor}\text{ or }2^{\lfloor(n+1)/2\rfloor}.
\label{eq:covering}
\end{equation}
Algorithm~\ref{alg:global} for $SO(n)$ or $O(n)$ on $(I,R_0)$ returns $r_R\exp(A)$, with $r_R$ the reflection of a subspace of dimension $k$, even for $SO(n)$, and $A\in\mathfrak{so}(n)$, and both factors lift:
\begin{equation}
B=\tfrac14\sum_{a,b}A_{ab}\,\gamma_a\gamma_b\in\mathfrak{spin}(n),\qquad
\tilde r=\tilde u_1\tilde u_2\cdots\tilde u_k,\quad \tilde u_j=\sum_au_j^a\gamma_a,
\label{eq:spinlift}
\end{equation}
with $u_1,\dots,u_k$ an orthonormal basis of the subspace reversed by $r_R$, so that $\tilde r\in\mathrm{Spin}(n)$ for even $k$ and $\tilde r$ lies in the odd component of $\mathrm{Pin}(n)$ for odd $k$. The two spinors covering $r_R\exp(A)$ are $\pm\tilde r\exp(B)$, and the data fix the sign, $\zeta=\mathrm{sign}\,\mathrm{Re}\,\tr\big((\tilde re^{B}X)^\dagger Y\big)$:
\[
\begin{tikzcd}[row sep=large, column sep=9em]
g_0 \arrow[r, "\text{Algorithm~\ref{alg:spin}}"] \arrow[d, "\pi"', "\text{Eq.~\eqref{eq:covering}}"] & \zeta\tilde r\exp(B)\in\mathrm{Spin}(n) \arrow[d, "\pi"', shift left=1.5ex] \\
R_0 \arrow[r, "\text{Algorithm~\ref{alg:global} in }SO(n)"'] & r_R\exp(A)\in SO(n) \arrow[u, "\pm\text{ lift, Eq.~\eqref{eq:spinlift}}", shift left=1.5ex]
\end{tikzcd}
\]
with the bottom row in $O(n)$ for $\mathrm{Pin}(n)$. The output is the projected logarithm of $g_0$ on the Riemann sheet of $B$,
\begin{equation}
g=\zeta\,\tilde r\exp\big(B+\proj_{\mathfrak g}\log(e^{-B}\tilde r^{-1}\zeta g_0)\big),
\label{eq:spinout}
\end{equation}
which on noiseless data is $\zeta\tilde r\exp(B)=g_\ast$. The full proofs are in the Supplementary Material.
\begin{figure}[t]
\centering
\begin{minipage}[c]{0.42\linewidth}\centering
\resizebox{\linewidth}{!}{\begin{tikzpicture}
\begin{axis}[
  width=11cm, height=8.6cm,
  view={-38}{22},
  axis lines=none,
  xmin=-2.6, xmax=2.6, ymin=-1.7, ymax=1.7, zmin=-13.5, zmax=10.5,
  clip=false,
  colormap={lower}{color=(gray!55) color=(gray!55)},
]
\addplot3[fill=gray!8, draw=gray!60, fill opacity=0.9] coordinates
  {(-1.6,-1.6,-13) (1.6,-1.6,-13) (1.6,1.6,-13) (-1.6,1.6,-13) (-1.6,-1.6,-13)};
\addplot3[domain=0:360, samples=80, samples y=0, gray!80, thin, dotted]
  ({cos(x)},{sin(x)},{-13});
\addplot3[very thick, black] coordinates {(-1.6,0,-13) (0,0,-13)};
\addplot3[thin, gray] coordinates {(0,0,-13) (1.6,0,-13)};
\addplot3[only marks, mark=*, mark size=1.6pt, black] coordinates {(-1,0,-13) (1,0,-13) (0,0,-13)};
\node[anchor=north, font=\footnotesize] at (axis cs:-1,0,-13) {$-1$};
\node[anchor=north west, font=\footnotesize] at (axis cs:1,0,-13) {$1$};
\node[anchor=north, font=\footnotesize] at (axis cs:-1.2,-0.1,-13.3) {branch cut};
\node[anchor=west, font=\footnotesize] at (axis cs:2.0,1.4,-13) {$\lambda$ plane};
\addplot3[surf, shader=flat, opacity=0.45, draw=gray!70, line width=0.1pt,
  domain=0.12:1.45, y domain=-540:-180, samples=7, samples y=25, colormap name=lower, point meta=0]
  ({x*cos(y)},{x*sin(y)},{y*pi/180});
\addplot3[surf, shader=flat, opacity=0.75, draw=blue!40!black, line width=0.1pt,
  domain=0.12:1.45, y domain=-180:180, samples=7, samples y=25,
  colormap={prin}{color=(blue!25) color=(blue!25)}, point meta=0]
  ({x*cos(y)},{x*sin(y)},{y*pi/180});
\addplot3[surf, shader=flat, opacity=0.45, draw=gray!70, line width=0.1pt,
  domain=0.12:1.45, y domain=180:540, samples=7, samples y=25, colormap name=lower, point meta=0]
  ({x*cos(y)},{x*sin(y)},{y*pi/180});
\addplot3[very thick, red!70!black] coordinates {(-0.12,0,3.14159) (-1.45,0,3.14159)};
\addplot3[very thick, red!70!black] coordinates {(-0.12,0,-3.14159) (-1.45,0,-3.14159)};
\addplot3[domain=-540:540, samples=160, samples y=0, thick, blue!60!black]
  ({cos(x)},{sin(x)},{x*pi/180});
\addplot3[dashed, gray] coordinates {(-1,0,-13) (-1,0,9.6)};
\addplot3[only marks, mark=*, mark size=2pt, red!70!black] coordinates {(-1,0,3.14159) (-1,0,-3.14159)};
\addplot3[only marks, mark=o, mark size=2pt, gray!40!black] coordinates {(-1,0,9.42478) (-1,0,-9.42478)};
\node[anchor=west, font=\footnotesize, red!70!black] at (axis cs:-1,0,3.14159) {$\ i\pi$};
\node[anchor=west, font=\footnotesize, red!70!black] at (axis cs:-1,0,-3.14159) {$\ -i\pi$};
\node[anchor=west, font=\footnotesize, gray!40!black] at (axis cs:-1,0,9.42478) {$\ 3i\pi$};
\node[anchor=west, font=\footnotesize, gray!40!black] at (axis cs:-1,0,-9.42478) {$\ -3i\pi$};
\addplot3[only marks, mark=*, mark size=1.8pt, blue!60!black] coordinates
  {(-0.98,0.17,2.97) (-0.98,-0.17,-2.97)};
\node[anchor=south, font=\footnotesize, align=center] at (axis cs:0,0,10.2)
  {$\mathrm{Im}\log\lambda$};
\draw[-stealth, thin] (axis cs:0,0,-13) -- (axis cs:0,0,10);
\node[anchor=west, font=\footnotesize, blue!40!black, align=left] at (axis cs:1.9,-0.9,-1.5)
  {principal branch\\ $|\mathrm{Im}\log\lambda|<\pi$:\\ one turn};
\node[anchor=east, font=\footnotesize, align=right, xshift=-36pt] at (axis cs:-2.2,-0.7,6.0)
  {one circuit of the origin\\ raises $\log$ by $2\pi i$};
\end{axis}
\end{tikzpicture}}
\par\footnotesize (a)
\end{minipage}\hfill
\begin{minipage}[c]{0.55\linewidth}\centering
\resizebox{\linewidth}{!}{
\newcommand{\minihelicoid}[3]{%
\begin{scope}[xshift=#1]
\begin{axis}[
  width=4.6cm, height=6.2cm, view={-38}{22}, axis lines=none,
  xmin=-1.6, xmax=1.6, ymin=-1.6, ymax=1.6, zmin=-9.5, zmax=9.5, clip=false,
  colormap={lower}{color=(gray!55) color=(gray!55)}]
\addplot3[surf, shader=flat, opacity=0.4, draw=gray!70, line width=0.1pt,
  domain=0.12:1.3, y domain=-540:-180, samples=5, samples y=19, colormap name=lower, point meta=0]
  ({x*cos(y)},{x*sin(y)},{y*pi/180});
\addplot3[surf, shader=flat, opacity=0.7, draw=blue!40!black, line width=0.1pt,
  domain=0.12:1.3, y domain=-180:180, samples=5, samples y=19,
  colormap={prin}{color=(blue!25) color=(blue!25)}, point meta=0]
  ({x*cos(y)},{x*sin(y)},{y*pi/180});
\addplot3[surf, shader=flat, opacity=0.4, draw=gray!70, line width=0.1pt,
  domain=0.12:1.3, y domain=180:540, samples=5, samples y=19, colormap name=lower, point meta=0]
  ({x*cos(y)},{x*sin(y)},{y*pi/180});
\addplot3[domain=-540:540, samples=120, samples y=0, thin, blue!60!black]
  ({cos(x)},{sin(x)},{x*pi/180});
\addplot3[dashed, gray] coordinates {(-0.5,0.866,-9.2) (-0.5,0.866,9.2)};
\addplot3[only marks, mark=o, mark size=1.8pt, gray!40!black] coordinates
  {(-0.5,0.866,-4.18879) (-0.5,0.866,8.37758)};
\addplot3[only marks, mark=*, mark size=2.4pt, red!70!black] coordinates {(-0.5,0.866,2.0944)};
\node[anchor=east, font=\footnotesize, red!70!black, xshift=-3pt] at (axis cs:-0.5,0.866,2.0944) {$\tfrac{2\pi i}{3}$};
\ifnum#3=1
\addplot3[only marks, mark=*, mark size=2.4pt, green!50!black] coordinates {(0.5,-0.866,-1.0472)};
\addplot3[dashed, green!50!black] coordinates {(0.5,-0.866,-9.2) (0.5,-0.866,9.2)};
\node[anchor=west, font=\footnotesize, green!50!black] at (axis cs:0.5,-0.866,-1.0472) {$\ -\tfrac{\pi i}{3}$};
\fi
\node[anchor=south, font=\footnotesize] at (axis cs:0,0,9.6) {#2};
\end{axis}
\end{scope}}
\resizebox{0.72\linewidth}{!}{\begin{tikzpicture}
\minihelicoid{0cm}{$\lambda_1=\omega$}{0}
\minihelicoid{4.2cm}{$\lambda_2=\omega$}{1}
\minihelicoid{8.4cm}{$\lambda_3=\omega$}{1}
\node[anchor=north, font=\footnotesize, align=center, text width=12.2cm] at (6.3cm,-0.1cm)
  {principal choice on every eigenvalue ({\color{red!70!black}$\bullet$}): trace $2\pi i$, not in $\mathfrak{su}(3)$\\
   two eigenvalues reflected, $\omega\mapsto-\omega$, half a turn down ({\color{green!50!black}$\bullet$}): trace $0$, in $\mathfrak{su}(3)$};
\end{tikzpicture}}}
\par\footnotesize (b)
\end{minipage}
\par\medskip
\begin{minipage}[c]{0.6\linewidth}\centering
\includegraphics[width=\linewidth]{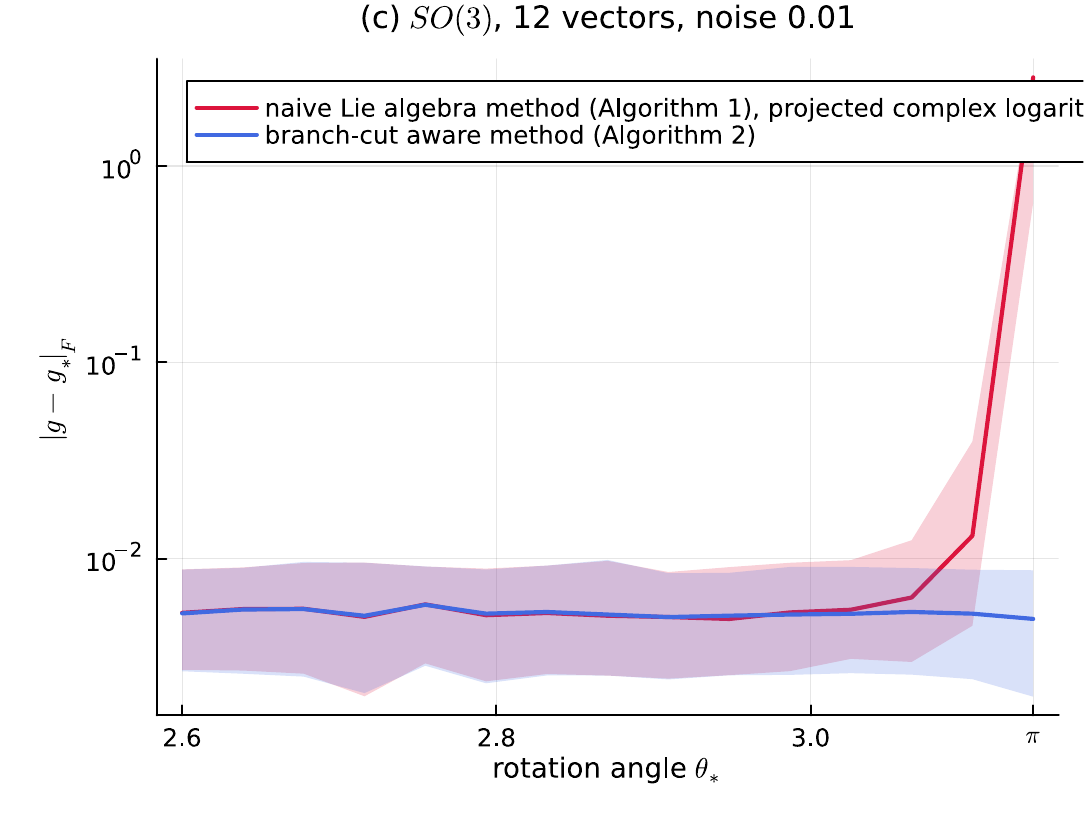}
\par\footnotesize (c)
\end{minipage}
\caption{The branch of the logarithm. (a) The Riemann surface of $\log\lambda$, a helicoid of height $\mathrm{Im}\log\lambda$ on which the helix is the lift of the unit circle. The principal branch is the single turn $|\mathrm{Im}\log\lambda|<\pi$, cut along the negative real axis. Above $\lambda=-1$ the two points $\pm i\pi$ lie on its two edges, equally close to zero, which is the ambiguity $\pm\pi K$ for the rotation $R(\pi)$, and a conjugate pair $-1\pm i\epsilon$ lifts to two points near the two edges whose logarithms form a real generator. (b) Each eigenvalue of a matrix carries its own copy of the surface, shown for $\omega I\in SU(3)$, $\omega=e^{2\pi i/3}$. The principal choice (red) has trace $2\pi i$, so the projection removes the trace and the exponential returns $I$. The reflection $r$ of Eq.~\eqref{eq:flip} moves two of the logarithms by half a turn to $-i\pi/3$ (green), and the trace becomes zero. (c) Median error and $10$th to $90$th percentile band over $200$ trials in $SO(3)$ with $12$ vectors and noise $0.01$, against the rotation angle $\theta_\ast$, for the naive method with the real projection of the complex principal logarithm and for Algorithm~\ref{alg:global}. Near $\pi$ the two eigenvalues of $g_0$ near $-1$ split into two negative reals in about $70\%$ of the samples, and the naive method returns a matrix near $I$. Script \texttt{reproduce/figure\_branch.jl}.}
\label{fig:branch}
\end{figure}
\begin{table}[t]
\centering
\caption{Component structure of the groups of Theorem~\ref{thm:global}, with $p,q\ge1$. The second column counts the connected components, i.e.\ the cosets of $G^0$, the third gives coset representatives, the fourth states whether $\exp(\mathfrak g)$ is all of $G^0$, and the last names the discrete factor which Algorithm~\ref{alg:global} supplies: the involution $r$ reaches the other cosets and the points of $G^0$ outside $\exp(\mathfrak g)$, for the complex special groups it is the reflection which repairs the trace, and for $\mathrm{Spin}(n)$ and $\mathrm{Pin}(n)$ a lift $\tilde r$ of the involution of the covered rotation or reflection, up to the central sign; $\tilde u$ is a unit vector $\sum_au^a\gamma_a$. $P$ and $T$ are reflections of one space and of one time axis of $O(p,q)$, for $O(3,1)$ equivalently parity and time reversal, and $\rho=\mathrm{diag}(-1,1,\dots,1)$ a reflection.}
\label{tab:components}
\footnotesize
\setlength{\tabcolsep}{4pt}
\begin{tabular}{>{\raggedright\arraybackslash}p{0.25\linewidth}>{\raggedright\arraybackslash}p{0.09\linewidth}>{\raggedright\arraybackslash}p{0.1\linewidth}>{\raggedright\arraybackslash}p{0.27\linewidth}>{\raggedright\arraybackslash}p{0.15\linewidth}}
\hline
Group & Compo\-nents & Represen\-tatives & $\exp(\mathfrak g)=G^0$ & Factor\\
\hline
$GL(n,\mathbb R)$ & 2 & $I,\rho$ & no, e.g.\ $\mathrm{diag}(-1,-2)$ & $r$\\
$SL(n,\mathbb R)$, $Sp(2n,\mathbb R)$, $Sp(2n,\mathbb C)$ & 1 & $I$ & no, e.g.\ $-I-N_2$ & $r$\\
$O(n)$, $E(n)$ & 2 & $I,\rho$ & yes & $r$\\
$SO(n)$, $SE(n)$, $\mathrm{Sim}(n)$, $GL(n,\mathbb C)$, $U(n)$ & 1 & $I$ & yes & $r$ (branch only)\\
$O(p,q)$, $E(p,q)$ & 4 & $I,P,T,PT$ & no for $\min(p,q)\ge2$ \citep{nishikawa1983} & $r$\\
$SO(p,q)$, $SE(p,q)$ & 2 & $I,PT$ & yes iff $\min(p,q)\le1$ \citep{nishikawa1983} & $r$\\
$O(n,\mathbb C)$ & 2 & $I,\rho$ & no for $n\ge4$, e.g.\ $-U$ with $U$ unipotent of Jordan type $(3,1)$ \citep{djokovicThang1995} & $r$\\
$SO(n,\mathbb C)$ & 1 & $I$ & no for $n\ge4$, same example & $r$\\
$U(p,q)$ & 1 & $I$ & yes (\S\ref{sec:exact}) & $r$ (branch only)\\
$SL(n,\mathbb C)$ & 1 & $I$ & no, e.g.\ $-I-N_2$ & $r$ (reflection)\\
$SU(n)$ & 1 & $I$ & yes & $r$ (reflection, branch only)\\
$\mathrm{Spin}(n)$ & 1 & $I$ & yes & $\pm\tilde r$ (lift of $r_R\in SO(n)$, branch only)\\
$\mathrm{Pin}(n)$ & 2 & $I,\tilde u$ & yes & $\pm\tilde r$ (lift of $r_R\in O(n)$)\\
\hline
\end{tabular}
\end{table}
\begin{theorem}[Exactness of Algorithms~\ref{alg:global} and~\ref{alg:spin}]\label{thm:global}
Let $X$ have full row rank and $Y=g_\ast X$ with $g_\ast\in G$.
\par (a) \emph{Groups preserving a form and the real linear groups.} If $G$ is $GL(n,\mathbb R)$, $SL(n,\mathbb R)$, $O(n)$, $SO(n)$, $O(p,q)$, $SO(p,q)$, $Sp(2n,\mathbb R)$, $E(n)$, $SE(n)$, $\mathrm{Sim}(n)$, $E(p,q)$, $SE(p,q)$, $GL(n,\mathbb C)$, $U(n)$, $U(p,q)$, $O(n,\mathbb C)$, $SO(n,\mathbb C)$ or $Sp(2n,\mathbb C)$, and $r$ is the reflection of Eq.~\eqref{eq:sign} across $W=N$ of Eq.~\eqref{eq:involution}, with $M=I$ for the groups preserving no form, then $r\in G$, $r^{-1}g_\ast\in G$ has no negative real eigenvalue, $\log(r^{-1}g_\ast)\in\mathfrak g$, and Eq.~\eqref{eq:globalreal} returns $g_\ast$.
\par (b) \emph{The complex special groups.} If $G$ is $SL(n,\mathbb C)$ or $SU(n)$, let $S$ be a set of eigenvalues of $g_\ast$, counted with multiplicity, which satisfies Eq.~\eqref{eq:flipcount}, contains every negative real eigenvalue and no positive real one, and $r$ the reflection of Eq.~\eqref{eq:flip}. Such a set exists. Then $r\in SU(n)$, $r^{-1}g_\ast\in G$ has no negative real eigenvalue, $\log(r^{-1}g_\ast)\in\mathfrak g$, and Eq.~\eqref{eq:globalreal} returns $g_\ast$.
\par (c) \emph{The spin groups.} If $G$ is $\mathrm{Spin}(n)$ or $\mathrm{Pin}(n)$, let $R\in SO(n)$, respectively $O(n)$, be the element covered by $g_\ast$, $R=r_R\exp(A)$ its factorization of part (a), and $\tilde r$, $B$ the lifts of Eq.~\eqref{eq:spinlift}. Then $g_\ast=\zeta\,\tilde r\exp(B)$ for one sign $\zeta=\pm1$, and Eq.~\eqref{eq:spinout} returns $g_\ast$.
\par In particular every element of these groups is $r\exp(A)$ with $A\in\mathfrak g$ and $r$ an involution in $G$, or for $\mathrm{Spin}(n)$ a lift of one up to the central sign, and the branch-cut aware method is onto $G$.
\end{theorem}
The proof is in the Supplementary Material. Since $X$ has full row rank, $g_\ast$ is the unique minimizer of $\|gX-Y\|_F$ over $G$, so on noiseless data the branch-cut aware method returns the constrained least squares solution, to which Newton's method on the group (\S\ref{sec:correction}) also converges, and under noise the two agree to first order only (\S\ref{sec:noise}). When every eigenvalue of $g_0$ has positive real part and, for the complex special groups, $k=0$, then $r=I$ and Algorithm~\ref{alg:global} coincides with Algorithm~\ref{alg:naive}, as for the near-identity samples of \S\ref{sec:accuracy}.

\section{Behavior under noise}
\label{sec:noise}
\subsection{A first order error bound}
For $Y=g_\ast X+E$, the first step gives $g_0=g_\ast+EX^+$ with $\|EX^+\|_F\le\|E\|_F/\sigma_{\min}(X)$, where $\sigma_{\min}(X)$ is the smallest singular value of $X$. The logarithm amplifies this perturbation by at most
\begin{equation}
\kappa_{\log}(g)=\int_0^1\big\|\big(t(g-I)+I\big)^{-1}\big\|_2^2\,dt,
\label{eq:kappalog}
\end{equation}
the exponential by at most $e^{\|A\|_2}$, and the projection not at all.
\begin{theorem}[First order error]\label{thm:firstorder}
Let $X$ have full row rank, let $g_\ast=\exp(A)$ with $A\in\mathfrak g$ and every eigenvalue of $A$ in the strip $|\mathrm{Im}\,\lambda|<\pi$, and let $Y=g_\ast X+E$. Then the output $g$ of the Lie algebra method satisfies
\begin{equation}
g-g_\ast=L_{\exp}\big(A,\proj_{\mathfrak g}(L_{\log}(g_\ast,EX^+))\big)+O(\|E\|_F^2),
\label{eq:firstorder}
\end{equation}
and in particular
\begin{equation}
\|g-g_\ast\|_F\le e^{\|A\|_2}\,\kappa_{\log}(g_\ast)\,\frac{\|E\|_F}{\sigma_{\min}(X)}+O(\|E\|_F^2).
\label{eq:bound}
\end{equation}
\end{theorem}
The factor $1/\sigma_{\min}(X)$ is large for nearly degenerate data and depends on the experiment, not on the method. The factor $\kappa_{\log}(g_\ast)$ is large near the failure cases of \S\ref{sec:exact}: for unitary $g_\ast$ with eigenvalues $e^{i\theta_j}$ it equals $\theta_{\max}/\sin\theta_{\max}$ with $\theta_{\max}=\max_j|\theta_j|$, infinite for a rotation by $\pi$, and for a boost with rapidity $\phi$ it equals $e^{\phi}$, the same growth as $e^{\|A\|_2}$. For the compact groups, $e^{\|A\|_2}$ can be replaced by one.

\subsection{Compact groups and the least squares solution}
\label{sec:compact}
For the orthogonal and unitary groups the factor $\kappa_{\log}$ is unnecessary. At $g_\ast$ the group has the tangent space $\mathcal T=\{Kg_\ast:K\in\mathfrak g\}$, and $\Pi_{\mathcal T}$ denotes the orthogonal projection onto $\mathcal T$.
\begin{theorem}[First order error for compact groups]\label{thm:compact}
Under the assumptions of Theorem~\ref{thm:firstorder}, let $G$ be $U(n)$, $SU(n)$ or $SO(n)$. Then
\begin{equation}
g-g_\ast=\Pi_{\mathcal T}(EX^+)+O(\|E\|_F^2),\qquad\text{so}\qquad\|g-g_\ast\|_F\le\frac{\|E\|_F}{\sigma_{\min}(X)}+O(\|E\|_F^2).
\label{eq:compact}
\end{equation}
\end{theorem}
The Fr\'echet derivative of the logarithm maps $\mathcal T$ into $\mathfrak g$ and the orthogonal complement of $\mathcal T$ into that of $\mathfrak g$, so the projection discards exactly the part of the perturbation orthogonal to the group. It allows a comparison with the least squares solution over the group, $\hat g=\arg\min_{h\in G}\|hX-Y\|_F$.
\begin{corollary}[Agreement with the least squares solution for isotropic data]\label{cor:isotropic}
Under the assumptions of Theorem~\ref{thm:compact}, suppose that the input vectors are isotropic, $XX^\dagger=cI$ for some $c>0$. Then $\hat g-g=O(\|E\|_F^2)$, i.e. the Lie algebra method agrees with the least squares solution over the group to first order in the noise.
\end{corollary}
For data which are not isotropic, the least squares solution projects $EX^+$ onto $\mathcal T$ in the inner product weighted by $XX^\dagger$, and the two estimates differ at first order.
\par Theorems~\ref{thm:firstorder} and \ref{thm:compact} are stated for Algorithm~\ref{alg:naive} and hold for Algorithm~\ref{alg:global} for small deviations, provided that no eigenvalue of $g_\ast$ lies on the imaginary axis, so that $r$ depends smoothly on $g_0$ near $g_\ast$. Since $r(g_0)\in G$ for every $g_0$, the first order change $\delta r$ of $r$ satisfies $\delta r\,h_\ast\in\mathcal T$ with $h_\ast=r_\ast^{-1}g_\ast$, and the output of Algorithm~\ref{alg:global} is $r(g_0)$ times the output of Algorithm~\ref{alg:naive} on $r(g_0)^{-1}g_0$. For a compact group, left multiplication by $r_\ast$ is an isometry, so Theorem~\ref{thm:compact} gives $g-g_\ast=\delta r\,h_\ast+\Pi_{\mathcal T}(EX^+-\delta r\,h_\ast)=\Pi_{\mathcal T}(EX^+)$ up to $O(\|E\|_F^2)$, the same first order error, which also extends Theorem~\ref{thm:compact} to $O(n)$. For the other groups, Theorem~\ref{thm:firstorder} at $h_\ast$ gives a bound of the same form, with the constant enlarged by the sensitivity of $P$ to $g_0$, finite under the hypothesis.

\subsection{Gauss--Newton refinement}
\label{sec:correction}
When the least squares solution over the group is desired, a few Gauss--Newton steps on the group reach it from the Lie algebra estimate. With an orthonormal basis $\{B_k\}_{k=1}^{d}$ of $\mathfrak g$ and the current estimate $g$, nearby elements are $\exp(\delta)g$ with $\delta=\sum_kc_kB_k$, the residual $R=Y-gX$ changes by $-\delta gX$ to first order, and the best step solves
\begin{equation}
c=D^+\mathrm{vec}(R),\qquad D=\big[\mathrm{vec}(B_1gX)\ \cdots\ \mathrm{vec}(B_dgX)\big],
\label{eq:corrstep}
\end{equation}
with real and imaginary parts stacked for complex groups, after which $g\leftarrow\exp(\delta)g$ (Algorithm~\ref{alg:correct}). This is Gauss--Newton with the exponential map as retraction \citep{absil2008,mahonyManton2002}: the Jacobian at $\delta=0$ is the explicit matrix $D$, whereas Newton's method on coordinates $\theta_k$ of $\exp(\sum_k\theta_kB_k)$ differentiates through the exponential and forms a Hessian. The start is within $O(\|E\|_F)$ of the solution, convergence is linear under noise \citep{absil2008}, there is no step control, and the tolerance $\tau=10^{-13}$ is reached after a median of eight steps for $SO(4)$ at noise $0.05$. The refinement should not be used when the inputs $X$ are themselves noisy: the objective is then biased, as in any regression with errors in the variables, and moving toward its minimizer can increase the error.
\begin{algorithm}[t]
\caption{Gauss--Newton refinement of the Lie algebra estimate (\texttt{gauss\_newton\_refine})}
\label{alg:correct}
\begin{algorithmic}[1]
\Require data $X,Y$ with $X$ of full row rank; orthonormal basis $\{B_k\}_{k=1}^{d}$ of $\mathfrak g$; tolerance $\tau=10^{-13}$; maximum number of steps $N=12$
\Procedure{GaussNewtonRefine}{$X,Y,G$}
\State $g \gets$ \Call{NaiveAlign}{$X,Y,G$} \Comment{or \Call{BranchCutAlign}{$X,Y,G$}}
\For{$i = 1,\dots,N$}
\State $R \gets Y - gX$;\quad $D \gets \big[\mathrm{vec}(B_1 gX)\ \cdots\ \mathrm{vec}(B_d gX)\big]$
\State $c \gets D^{+}\,\mathrm{vec}(R)$;\quad $\delta \gets \sum_{k} c_k B_k$
\If{$\|\delta\|_F \le \tau\max(1,\|g\|_F)$} \textbf{break} \EndIf
\State $g \gets \exp(\delta)\,g$
\EndFor
\State \Return $g$
\EndProcedure
\end{algorithmic}
\end{algorithm}

\section{Benchmarks}
\label{sec:benchmarks}
\subsection{Software and setup}
The Julia module \texttt{FastLieAlign.jl} \citep{zenodo2026} implements Algorithms~\ref{alg:naive} to~\ref{alg:correct} as \path{naive_align}, \path{branch_cut_align}, \path{negative_involution} with \path{even_reflection}, \path{spin_align} and \path{gauss_newton_refine}. The comparison methods, BFGS \citep{fletcher2000} and damped Newton with a finite difference or an exact Hessian, started at the identity (cold) or at the Lie algebra estimate (warm), minimize $\|\exp(\sum_k\theta_kB_k)X-Y\|_F^2$ over $\theta_k$ with derivatives by automatic differentiation. The scripts \path{reproduce/reproduce.jl} (the figures and Table~\ref{tab:adnewton}(a)), \path{reproduce/figure_branch.jl} (Fig.~\ref{fig:branch}), \path{reproduce/whole_group_benchmark.jl} (Table~\ref{tab:global}) and \path{reproduce/extra_experiments.jl} (the other checks quoted below) regenerate the results. Timings are medians of single wall clock measurements on an Apple M5 Max MacBook Pro and vary by about $10\%$ between runs.
\par The true elements are $g_\ast=\exp(s\,\proj_{\mathfrak g}(Z))$ with $Z$ standard Gaussian (complex for complex groups) and $s=0.3$, the inputs are $m=3n$ Gaussian vectors, and the noise is Gaussian with standard deviation $\epsilon$ in every component of $Y$. Each configuration is repeated $1{,}000$ times for Table~\ref{tab:adnewton}(a) and the histograms and $100$ times for Fig.~\ref{fig:time}.

\subsection{Accuracy and timing}
\label{sec:accuracy}
\label{sec:results-adnewton}
For $SO(4)$ at $\epsilon=0.05$ (Table~\ref{tab:adnewton}(a)), Kabsch, the Newton variants and the Gauss--Newton refinement reach the same least squares solution, with median error $3.545\times10^{-2}$, and the Lie algebra method is $17\%$ worse, as \S\ref{sec:compact} predicts for $12$ Gaussian input vectors. The polar factor of $g_0$, the nearest rotation to the unconstrained estimate, matches the Lie algebra method to three digits, since to first order in the noise both discard the component of $g_0-g_\ast$ normal to the group. It is cheaper, but it relies on the Euclidean inner product, which the noncompact groups do not preserve. The refinement costs about seven times the Lie algebra method, which needs about $11\,\mu$s, and the fastest Newton variant about $20$ times. This close to the identity, a cold start costs the same as a warm one. For $Sp(4,\mathbb R)$ with $s=1$ and the step cap raised to $50$, the cold start needed up to $50$ steps, against at most six from the Lie algebra estimate. BFGS is slower than every Newton variant and is omitted from the table.
\begin{table}[t]
\centering
\caption{(a) Median time and median error $\|g-g_\ast\|_F$ over $1{,}000$ random $SO(4)$ alignments with noise $\epsilon=0.05$. The polar factor row is the nearest rotation to the unconstrained least squares estimate $g_0=YX^+$. Warm starts begin at the Lie algebra estimate and include its cost, cold starts begin at the identity, and the Newton baselines are our own implementations with automatic differentiation through the matrix exponential. (b) Median error for the Lorentz group $SO(3,1)^0$ against the scale $s$ of the true transformation $g_\ast=\exp(s\,\proj_{\mathfrak g}(Z))$, with $12$ Gaussian input vectors and noise $0.01$, over $300$ trials. The ratio column is the median over trials of the ratio of the two errors, and the last column is the median of $e^{\|\log g_\ast\|_2}\|E\|_F/\sigma_{\min}(X)$, the factor of Eq.~\eqref{eq:bound} without $\kappa_{\log}$, which is itself of order $e^{\|\log g_\ast\|_2}$ here.}
\label{tab:adnewton}
\footnotesize
(a)\par\smallskip
\begin{tabular}{>{\raggedright\arraybackslash}p{0.5\linewidth}rrr}
\hline
Method & Median time & Time ratio & Median error\\
\hline
Kabsch (SVD) & $3\,\mu\mathrm s$ & $0.3\times$ & $3.545\times10^{-2}$\\
Polar factor of $YX^+$ & $5\,\mu\mathrm s$ & $0.5\times$ & $4.149\times10^{-2}$\\
Lie algebra method & $11\,\mu\mathrm s$ & $1.0\times$ & $4.146\times10^{-2}$\\
Gauss--Newton refinement, warm start & $79\,\mu\mathrm s$ & $6.9\times$ & $3.545\times10^{-2}$\\
Gauss--Newton, cold start at $I$ & $76\,\mu\mathrm s$ & $6.7\times$ & $3.545\times10^{-2}$\\
Newton, exact Hessian, warm start & $223\,\mu\mathrm s$ & $20\times$ & $3.545\times10^{-2}$\\
Newton, finite difference Hessian, warm start & $403\,\mu\mathrm s$ & $35\times$ & $3.545\times10^{-2}$\\
\hline
\end{tabular}
\par\medskip (b)\par\smallskip
\begin{tabular}{lcccc}
\hline
Scale $s$ & Lie algebra method & Gauss--Newton refinement & Ratio & First order factor\\
\hline
$0.3$ & $8.9\times10^{-3}$ & $7.0\times10^{-3}$ & $1.2$ & $5.9\times10^{-2}$\\
$1$ & $1.7\times10^{-2}$ & $7.1\times10^{-3}$ & $2.4$ & $1.8\times10^{-1}$\\
$2$ & $7.0\times10^{-2}$ & $7.3\times10^{-3}$ & $11$ & $7.8\times10^{-1}$\\
$3$ & $2.8\times10^{-1}$ & $7.4\times10^{-3}$ & $50$ & $2.2$\\
\hline
\end{tabular}
\end{table}
\par The other groups behave alike (Fig.~\ref{fig:time} and the Supplementary Material). Without noise, every method reaches a median relative error between $10^{-16}$ and $10^{-14}$, as \S\ref{sec:exact} predicts. With noise, the refinement, Newton's method and BFGS find the same least squares solution, and the Lie algebra estimate has somewhat larger errors at one to two orders of magnitude less time than direct minimization, with the refinement several times faster than Newton's method.
\par The scale $s=0.3$ keeps $g_\ast$ near the identity, where $\kappa_{\log}\approx1$ and $e^{\|A\|_2}\approx1.35$. For larger Lorentz transformations (Table~\ref{tab:adnewton}(b)), the refined error is independent of $s$, as the least squares solution should be, while the unrefined error grows roughly like $e^{\|A\|_2}$, from $1.2$ times the refined error at $s=0.3$ to $50$ times at $s=3$. The two Fr\'echet derivatives in Eq.~\eqref{eq:firstorder} partly cancel, as they cancel completely in the compact case, and a sharper bound for the noncompact groups is open. At $s\ge2$, some samples have eigenvalues of $A$ with $|\mathrm{Im}\,\lambda|\ge\pi$, outside the hypothesis of Theorem~\ref{thm:firstorder}, and in $3$ trials at $s=2$ and $7$ at $s=3$, $g_0$ has a negative real eigenvalue, so Algorithm~\ref{alg:naive} is undefined and the trial is excluded from the medians. Algorithm~\ref{alg:global} handles these (Table~\ref{tab:global}).
\begin{figure}[t]
\centering
\includegraphics[width=\linewidth]{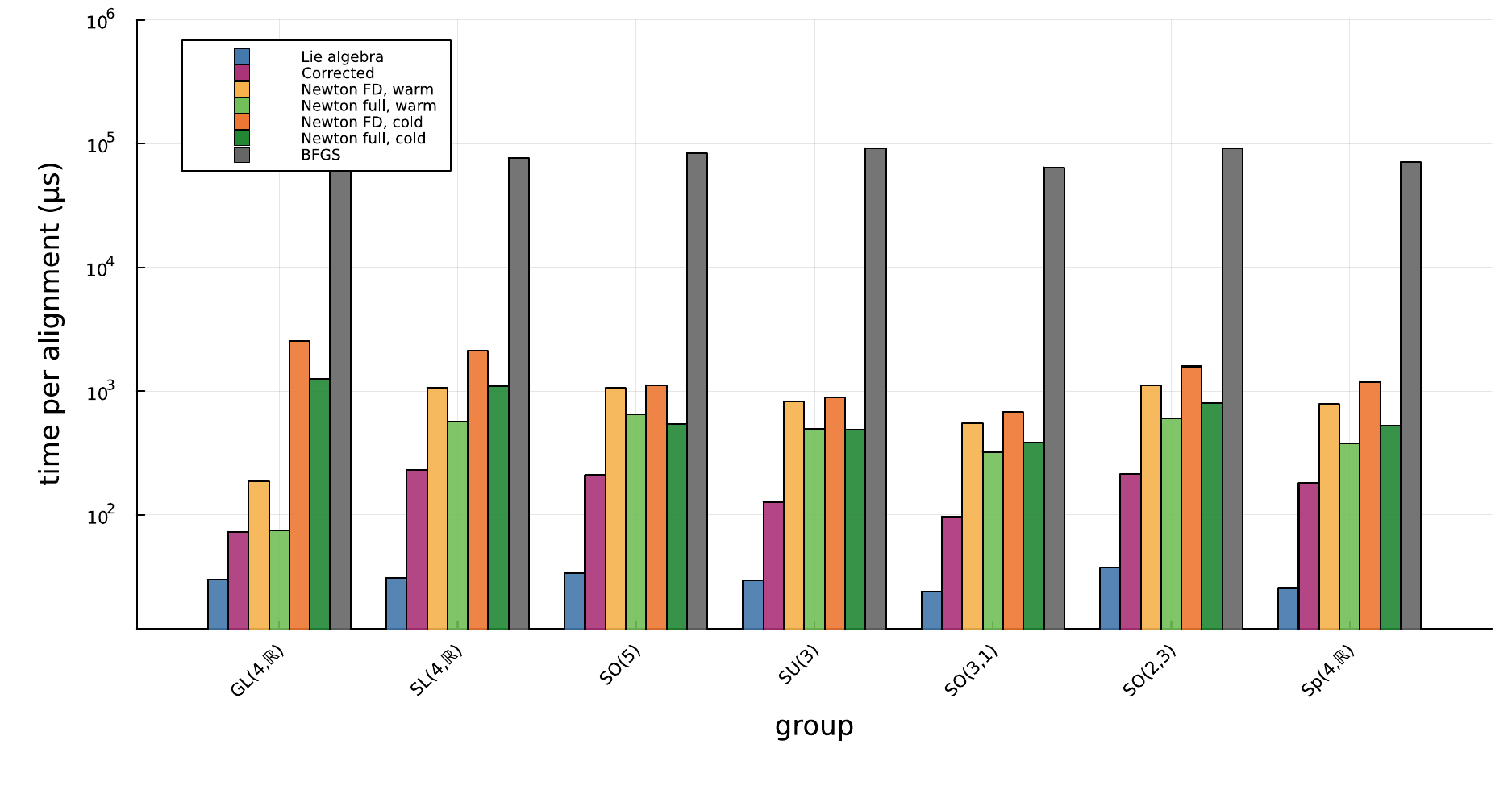}
\caption{Median time per alignment over $100$ trials for seven groups and seven methods with noise $\epsilon=0.05$: the Lie algebra method, its Gauss--Newton refinement (``Corrected''), Newton's method with a finite difference (FD) or exact (full) Hessian from a warm or cold start, and BFGS.}
\label{fig:time}
\end{figure}

\subsection{Elements beyond the naive method}
\label{sec:global}
\label{sec:results-global}
Table~\ref{tab:global} benchmarks Algorithm~\ref{alg:global} on true elements beyond the reach of Algorithm~\ref{alg:naive}: elements of the other components of $O(3)$, $O(3,1)$ and $SO(3,1)$, elements of the identity component which are not exponentials in $SL(2,\mathbb R)$, $Sp(4,\mathbb R)$ and $GL(2,\mathbb R)$, rotations by $\pi$ in $SO(3)$, $SE(3)$, $\mathrm{Sim}(3)$ and $SO(3,\mathbb C)$, roots of unity times a small element in $SU(3)$, $SL(3,\mathbb C)$, $SU(10)$ and $SL(10,\mathbb C)$, the rotation of Table~\ref{tab:failures} in $\mathrm{Spin}(6)$ and $\mathrm{Spin}(7)$, and lifts of reflections in $\mathrm{Pin}(3)$, $\mathrm{Pin}(4)$ and $O(3,\mathbb C)$. Each true element is multiplied by a random $\exp(sA)$ with $s=0.3$ so that it is not a special matrix, except for the rotations by $\pi$, whose eigenvalue $-1$ is left degenerate. On the real groups, the naive method is undefined in most trials, and its fallback returns an element at distance two to five from the truth, in the wrong component or on the wrong Riemann sheet. On the complex and spin groups, it is always defined and always wrong, at distance two to ten. Algorithm~\ref{alg:global} recovers every case to the noise level, and Algorithm~\ref{alg:correct} started from it improves the error by $5$ to $47\%$, except on $GL(2,\mathbb R)$, where the Lie algebra estimate already equals the unconstrained least squares solution. The cost is within $30\%$ of the naive method on the real and complex groups, where the Schur decomposition is added to the logarithm, and about twice as large on the spin groups, which add the covering map and a second logarithm.
\begin{table}[t]
\centering
\caption{Algorithm~\ref{alg:global} on elements outside the reach of the naive method: median error $\|g-g_\ast\|_F$ over $200$ trials with $m=3n$ vectors and noise $0.01$. The Algorithm~\ref{alg:naive} column is its output where it is defined and otherwise the real projection of the complex principal logarithm. ``Undefined'' is the fraction of trials in which the principal logarithm of $g_0$ does not exist. ``Refined'' is Algorithm~\ref{alg:correct} started from the output of Algorithm~\ref{alg:global}. Times are medians in microseconds on the machine of Table~\ref{tab:adnewton}, and the Algorithm~\ref{alg:naive} time is that of the complex principal logarithm and its real projection, without the check that rejects it. $B$ is a random boost, $u$ a random unit vector, $N_2$ the nilpotent matrix of Eq.~\eqref{eq:notexp}, $\omega=e^{2\pi i/3}$, $\omega_{10}=e^{2\pi i/10}$, and $h=\exp(0.3\,\proj_{\mathfrak g}(Z))$ a random element near the identity. Script \texttt{reproduce/whole\_group\_benchmark.jl}.}
\label{tab:global}
\footnotesize
\setlength{\tabcolsep}{4pt}
\begin{tabular}{llrrrrrr}
\hline
Group & $g_\ast$ & Alg.~\ref{alg:naive} & Undefined & Alg.~\ref{alg:global} & Refined & $t$ Alg.~\ref{alg:naive} & $t$ Alg.~\ref{alg:global}\\
\hline
$SO(3)$ & rotation by $\pi$ & $2.83$ & $0.70$ & $0.0063$ & $0.0054$ & $9$ & $10$\\
$O(3)$ & $(I-2uu^T)h$ & $2.00$ & $1.00$ & $0.0062$ & $0.0054$ & $8$ & $9$\\
$O(3,1)$ & $TB$ & $2.10$ & $1.00$ & $0.0084$ & $0.0070$ & $11$ & $12$\\
$O(3,1)$ & $PB$ & $2.02$ & $1.00$ & $0.0084$ & $0.0070$ & $11$ & $13$\\
$SO(3,1)$ & $PTB$ & $3.81$ & $1.00$ & $0.0084$ & $0.0070$ & $12$ & $11$\\
$SL(2,\mathbb R)$ & $-(I+N_2)h$ & $3.34$ & $0.56$ & $0.0089$ & $0.0068$ & $6$ & $6$\\
$GL(2,\mathbb R)$ & $\mathrm{diag}(-1,-2)h$ & $4.69$ & $0.89$ & $0.0086$ & $0.0086$ & $6$ & $6$\\
$Sp(4,\mathbb R)$ & $-(I+N_2)$ on $(q_1,p_1)$, times $h$ & $3.31$ & $0.58$ & $0.0128$ & $0.0097$ & $11$ & $13$\\
$SE(3)$ & rotation by $\pi$ with translation & $3.01$ & $0.73$ & $0.0104$ & $0.0071$ & $10$ & $12$\\
$\mathrm{Sim}(3)$ & scaled rotation by $\pi$ with translation & $2.70$ & $0.71$ & $0.0108$ & $0.0079$ & $10$ & $12$\\
$SU(3)$ & $\omega h$ & $3.00$ & $0$ & $0.0076$ & $0.0067$ & $10$ & $12$\\
$SL(3,\mathbb C)$ & $\omega h$ & $3.66$ & $0$ & $0.0117$ & $0.0104$ & $10$ & $13$\\
$SU(10)$ & $\omega_{10}^3h$ & $3.72$ & $0$ & $0.0156$ & $0.0142$ & $66$ & $91$\\
$SL(10,\mathbb C)$ & $\omega_{10}^3h$ & $9.53$ & $0$ & $0.0235$ & $0.0222$ & $72$ & $99$\\
$\mathrm{Spin}(6)$ & $e^{0.8\pi(S_{12}+S_{34}+S_{56})}h$ & $4.00$ & $0$ & $0.0069$ & $0.0057$ & $34$ & $76$\\
$\mathrm{Spin}(7)$ & $e^{0.8\pi(S_{12}+S_{34}+S_{56})}h$ & $4.00$ & $0$ & $0.0080$ & $0.0067$ & $35$ & $97$\\
$\mathrm{Pin}(3)$ & $\tilde uh$ & $2.83$ & $0$ & $0.0036$ & $0.0032$ & $12$ & $27$\\
$\mathrm{Pin}(4)$ & $\tilde uh$ & $2.83$ & $0$ & $0.0059$ & $0.0048$ & $13$ & $29$\\
$SO(3,\mathbb C)$ & rotation by $\pi$ & $3.22$ & $0$ & $0.0101$ & $0.0054$ & $10$ & $12$\\
$O(3,\mathbb C)$ & $\mathrm{diag}(-1,1,1)h$ & $2.06$ & $0$ & $0.0070$ & $0.0055$ & $9$ & $12$\\
\hline
\end{tabular}
\end{table}

\section{Discussion}
\label{sec:discussion}
The benchmarks of \S\ref{sec:benchmarks} separate three regimes. For the compact groups, Kabsch is exact and cheapest, the Lie algebra method is within $17\%$ of it in error at a few microseconds, as Theorem~\ref{thm:compact} predicts for anisotropic inputs, and the Gauss--Newton refinement reaches the least squares solution at about seven times the cost of the closed form, still three to five times below Newton's method. For the noncompact groups, there is no Procrustes solution. The Lie algebra method gives an element of the group with first order error (Theorem~\ref{thm:firstorder}), which grows with the size of the transformation, from $1.2$ times the refined error at $s=0.3$ to $50$ times at $s=3$ (Table~\ref{tab:adnewton}(b)), so for large transformations the estimate should start Algorithm~\ref{alg:correct} rather than be used alone. Beyond the reach of the principal logarithm, the naive method is undefined or wrong by a distance of order one, and Algorithm~\ref{alg:global} is at the noise level in every case of Table~\ref{tab:global}, at a cost within $30\%$ of the naive method on the real and complex groups and about twice it on $\mathrm{Spin}(n)$ and $\mathrm{Pin}(n)$. The refinement started from it gains a further $10$ to $30\%$, which it cannot do from the naive estimate, since Newton's method does not change component.
\par For $SO(3)$, the problem is the orthogonal Procrustes problem, solved in closed form by \citet{kabsch} and \citet{horn}, with the Kabsch approach extending to $SO(N)$. These solutions use the positive definite inner product which the group preserves and do not apply to $SO(p,q)$ or to the symplectic group. The indefinite case has been treated through polar decompositions in indefinite scalar product spaces \citep{kintzel2005,bolshakov1997} and in hyperbolic space \citep{tabaghiDokmanic2021}, and the nearest symplectic matrix problem, the least squares problem with $X=I$, by least squares optimization \citep{fiori2011,gaoSonAbsilStykel2021}. Direct minimization of $\|gX-Y\|_F$ over a parametrization of the group evaluates the matrix exponential and its derivatives many times per alignment, whereas the Lie algebra method needs one logarithm and one exponential. The two examples below show the method on a unitary group of dimension $256$, where the least squares solution is the polar factor and the refinement is out of reach, and on a Euclidean group, where the classical route is already the least squares solution.
\subsection{Quantum circuits on eight qubits}
\label{sec:qubits}
A quantum circuit on $n$ qubits is a unitary $U\in U(2^n)$ whose columns are its output states for the computational basis states, so with $X=I$ and $Y$ the measured output amplitudes the problem is alignment in $U(2^n)$ with isotropic inputs, for which Corollary~\ref{cor:isotropic} predicts agreement with the least squares solution, the polar factor of $Y$, to first order. We take $n=8$, $d=256$, and three circuits: the evolution $\exp(-iHt)$ for $t=0.1$ under the transverse field Ising Hamiltonian $H=-\sum_qZ_qZ_{q+1}-\sum_qX_q$, a random circuit of depth $20$ made of Haar random single-qubit gates and CNOT gates on alternating pairs, and the circuit which prepares the GHZ state, a layer of Hadamard gates followed by a ladder of CNOT gates, whose eigenvalue $-1$ has multiplicity $14$. Complex Gaussian noise of standard deviation $\epsilon$ is added to every amplitude, and Table~\ref{tab:qubits} reports medians over $50$ realizations (script \texttt{reproduce/qubits.jl}).
\begin{table}[t]
\centering
\caption{Recovery of an eight-qubit circuit $U\in U(256)$ from its noisy action on the computational basis: median relative error $\|\hat U-U\|_F/\|U\|_F$ over $50$ realizations for the polar factor of $Y$, which is the least squares unitary, Algorithm~\ref{alg:naive} and Algorithm~\ref{alg:global}, and the median distance $\|\hat U_2-\hat U_{\mathrm{polar}}\|_F/\|U\|_F$ between the last two.}
\label{tab:qubits}
\footnotesize
\begin{tabular}{llcccc}
\hline
Circuit & $\epsilon$ & Polar factor & Alg.~\ref{alg:naive} & Alg.~\ref{alg:global} & Alg.~\ref{alg:global} vs polar\\
\hline
Ising evolution & $10^{-3}$ & $1.13\times10^{-2}$ & $1.13\times10^{-2}$ & $1.13\times10^{-2}$ & $1.3\times10^{-5}$\\
& $10^{-2}$ & $1.13\times10^{-1}$ & $1.13\times10^{-1}$ & $1.13\times10^{-1}$ & $1.4\times10^{-3}$\\
Random circuit & $10^{-3}$ & $1.13\times10^{-2}$ & $1.14\times10^{-2}$ & $1.13\times10^{-2}$ & $3.5\times10^{-5}$\\
& $10^{-2}$ & $1.13\times10^{-1}$ & $1.47\times10^{-1}$ & $1.13\times10^{-1}$ & $3.6\times10^{-3}$\\
GHZ circuit & $10^{-3}$ & $1.13\times10^{-2}$ & $2.10\times10^{-1}$ & $1.13\times10^{-2}$ & $3.5\times10^{-5}$\\
& $10^{-2}$ & $1.13\times10^{-1}$ & $2.57\times10^{-1}$ & $1.13\times10^{-1}$ & $3.6\times10^{-3}$\\
\hline
\end{tabular}
\end{table}
\par For the Ising evolution, whose eigenvalues stay away from the branch cut, the three estimates agree to the digits shown, and the distance between Algorithm~\ref{alg:global} and the polar factor grows by a factor of $100$ when the noise grows by $10$, the quadratic agreement of Corollary~\ref{cor:isotropic}. The Lie algebra estimate also returns the generator: $i\log\hat U/t$ recovers $H$ with relative Frobenius error $0.030$ at $\epsilon=10^{-3}$ and $0.30$ at $10^{-2}$. For the random circuit, whose $256$ eigenvalues are spread around the unit circle with gaps of order $0.02$, pairs near $-1$ split across the branch cut under noise, and Algorithm~\ref{alg:naive} loses accuracy at $\epsilon=10^{-2}$. For the GHZ circuit, the multiple eigenvalue $-1$ splits in every realization, and Algorithm~\ref{alg:naive} is wrong by a distance of order one already at $\epsilon=10^{-3}$, as in Fig.~\ref{fig:branch}(c), while Algorithm~\ref{alg:global} stays at the noise level and within $3.5\times10^{-5}$ of the polar factor. One alignment takes about $0.1$\,s at $d=256$, dominated by the Schur decomposition and the logarithm.

\subsection{Rigid registration of point clouds}
\label{sec:results-se}
\par For rigid registration, the classical route \citep{kabsch,horn}, which is also the solution of Wahba's problem in attitude determination \citep{wahba1965,shusterOh1981,markley1988}, centers both point sets at their centroids $c_A$ and $c_B$, solves for the rotation $R\in SO(3)$, and maps $p_{A,i}$ to $R(p_{A,i}-c_A)+c_B$. In homogeneous coordinates $\tilde p=(p,1)^T$, the same transformation is
\begin{equation}
	g
	=\begin{bmatrix}I&c_B\\0&1\end{bmatrix}\begin{bmatrix}R&0\\0&1\end{bmatrix}\begin{bmatrix}I&-c_A\\0&1\end{bmatrix}=\begin{bmatrix}R&c_B-Rc_A\\0&1\end{bmatrix},
\end{equation}
and the Lie algebra method identifies $g$ directly, treating $SE(n)$, $E(n)$, $SE(p,q)$ and $E(p,q)$ like the other classical groups.
\par For $1{,}000$ rigid motions with $s=0.6$, a standard Gaussian translation and $15$ Gaussian points, both routes followed by Algorithm~\ref{alg:correct} recover the motion to about $10^{-15}$ without noise and give the same refined error with noise, since both converge to the least squares motion. Before refinement, the centroid route is already the least squares solution, and the homogeneous estimate has a median error $1.16$ times larger, as Theorem~\ref{thm:compact} does not cover $SE(3)$ and homogeneous data are never isotropic because of the row of ones. The value of the homogeneous formulation is that the same code handles the Poincar\'e group $SE(3,1)$ and other augmented groups.

\section{Conclusion}
\label{sec:conclusion}
The Lie algebra method estimates an element of a matrix Lie group from pairs of vectors in three steps, an unconstrained least squares solve, a matrix logarithm and a projection onto the Lie algebra followed by the exponential. We have derived the projection for the classical groups from the defining form of each group, shown that factoring an involution out of the least squares matrix makes the method exact on the whole group, including the components not reached by the exponential and the elements whose principal logarithm lies on the wrong Riemann sheet, and bounded the error under noise to first order. For the compact groups and isotropic inputs, the estimate agrees with the least squares solution over the group to first order, and it is within $17\%$ of it for Gaussian inputs. For the noncompact groups, the estimate is a group element with an explicit first order error, and the Gauss--Newton refinement reaches the least squares solution in a few steps at a cost below that of Newton's method. The Julia implementation covers the groups of Table~\ref{tab:projections}, including the spin and pin groups through their covering maps.
\par There are some potential areas for future improvement. First, Eq.~\eqref{eq:firstorder} gives the first order error in closed form, through the Fr\'echet derivatives of the logarithm and the exponential at the estimate, and the part of the perturbation which the projection discards is equally explicit. A second order correction could therefore be computed from the estimate itself, without the Jacobian of Algorithm~\ref{alg:correct}, and for each group the derivatives have a structure which the correction could exploit: for the unitary groups the discarded part is the component normal to the group, and for the Lorentz and symplectic groups it is the component which the form does not pair with the tangent space. Second, the factors $e^{\|A\|_2}$ and $\kappa_{\log}(g_\ast)$ of Eq.~\eqref{eq:bound} grow with the distance of the estimate from the identity, which is known once $\log g_0$ is computed, so the method itself can report how far its estimate is from the least squares solution and when the refinement is worth its cost. The same quantity can guide a recentering, in which the estimate is used to move the data toward the identity and the method is applied again to the residual. Third, the quaternionic groups and $SU(p,q)$ with $p,q\ge1$ were not treated here.

\section*{CRediT authorship contribution statement}
\textbf{Congzhou M. Sha:} Conceptualization, Formal analysis, Funding acquisition, Investigation, Methodology, Software, Validation, Visualization, Writing -- original draft, Writing -- review \& editing.

\section*{Declaration of competing interest}
The author declares no known competing financial interests or personal relationships that could have appeared to influence the work reported in this paper.

\section*{Data availability}
The Julia package, its documentation and the scripts which reproduce every figure and table of \S\ref{sec:benchmarks} and the example of \S\ref{sec:qubits} are developed at \href{https://github.com/mikesha2/FastLieAlign.jl}{github.com/mikesha2/FastLieAlign.jl}, and the version used here is archived on Zenodo \citep{zenodo2026}.

\section*{Acknowledgements}
The author gratefully acknowledges the MIT--Anthropic AI for Science Study for \$20{,}000 in usage credits, which supported the AI-assisted parts of this work described in the declaration below. The author is also a 2026--2027 Doximity Clinical AI Fellow.

\section*{Funding}
This research did not receive any specific grant from funding agencies in the public, commercial, or not-for-profit sectors. The usage credits from the MIT--Anthropic AI for Science Study are acknowledged above.

\section*{Declaration of generative AI and AI-assisted technologies in the manuscript preparation process}
During the preparation of this work the author used Anthropic's Claude models (Claude Haiku 4.5, Sonnet 4.6, Opus 4.8, Opus 5.5 and Fable 5.1) to assist with writing portions of the Julia code and with revising the manuscript text, including drafting proofs and checking references. After using these tools, the author reviewed and edited the content as needed and takes full responsibility for the content of the published article.

\appendix
\section{Background}
\label{sec:background}
\subsection{Matrix Lie groups and their motivation}
A \textbf{Lie group} is a group that is also a smooth manifold. The study of such groups arises from considering continuous physical symmetries, in which there is no preferred direction or orientation \emph{a priori}, and gauge theories, in which measured quantities should be independent of an arbitrary choice of coordinates.
\par The symmetries of physics are mostly the transformations which preserve a bilinear or sesquilinear form on some field $\mathbb F$, i.e. $\mathbb R^n$ or $\mathbb C^n$. Let $M$ be the invertible matrix of the form, so that the scalar product of two vectors is $\langle u,v\rangle=u^\star Mv$, with $\star$ the transpose for a bilinear form and the conjugate transpose for a sesquilinear form \citep[Chapter 1]{hall2015}; for indefinite forms, see \citet{glr2005}. The physical requirement is that $\forall u,v$ the scalar products be unchanged under $g$:
\begin{equation}
	u^\star Mv=\langle u,v\rangle=\langle gu,gv\rangle=u^\star g^\star Mg\,v,
\end{equation}
yielding the condition
\begin{equation}
	G_M=\{g:\ g^\star Mg=M\}.
	\label{eq:formgroup}
\end{equation}
The pair $(\mathbb F^n,\langle\cdot,\cdot\rangle)$ is a scalar product space and $G_M$ is its automorphism group \citep{mmt2006}.

\par A \textbf{matrix Lie group} $G$ is a closed subgroup of $GL(n,\mathbb F)$, the invertible $n\times n$ matrices over an infinite field $\mathbb F$ (e.g. $\mathbb Q$, $\mathbb R$, or $\mathbb C$) \citep{hall2015}. The \textbf{matrix exponential} of a square matrix $A$ is the series
\begin{equation}
\exp(A)=\sum_{k=0}^{\infty}\frac{A^k}{k!},
\label{eq:matexp}
\end{equation}
which converges for every $A$. It satisfies $\exp(0)=I$, $\exp((s+t)A)=\exp(sA)\exp(tA)$, $\exp(A+B)=\exp(A)\exp(B)$ when $AB=BA$, $\exp(SAS^{-1})=S\exp(A)S^{-1}$, $\frac{d}{dt}\exp(tA)=A\exp(tA)$, and $\det\exp(A)=e^{\tr A}$ \citep{hall2015}. The curve $t\mapsto\exp(tA)$ is thus a one parameter group of matrices with velocity $A$ at the identity.
\subsection{Matrix Lie algebras}
The \textbf{Lie algebra} $\mathfrak g$ of $G$ is the vector space of matrices $A$ with $\exp(tA)\in G$ for every real $t$ \citep{hall2015}, i.e. the tangent space at the identity of $G$. A simple example is the rotation group of the plane,
\begin{equation}
	SO(2)=\left\{R(\theta)=\begin{bmatrix}\cos\theta&-\sin\theta\\ \sin\theta&\cos\theta\end{bmatrix}:\theta\in\mathbb R\right\},
	\label{eq:so2}
\end{equation}
with $R(\theta)=\exp(\theta K)$, $K=\left[\begin{smallmatrix}0&-1\\1&0\end{smallmatrix}\right]$, and whose Lie algebra $\mathfrak{so}(2)$ is the line of antisymmetric matrices $\theta K$.
\par By applying the fact that the Lie algebra is the tangent space at the identity, we then define the Lie algebra $\mathfrak g_M$ of $G_M$ as
\begin{equation}
	\mathfrak g_M=\{A:\ A^\star M+MA=0\},
\end{equation}
by differentiating $\exp(tA)^\star M\exp(tA)=M$ at $t=0$.
\par With $M=I$, this construction yields the orthogonal group $O(n)$ for $\star=T$ and the unitary group $U(n)$ for $\star=\dagger$, with Lie algebras the antisymmetric and the skew Hermitian matrices. With $M=\eta=\mathrm{diag}(I_p,-I_q)$ and $\star=T$, this construction gives the indefinite orthogonal group $O(p,q)$, which for $(p,q)=(3,1)$ is the Lorentz group $O(3,1)$. For $M=J=\left[\begin{smallmatrix}0&I_n\\-I_n&0\end{smallmatrix}\right]$, we recover the symplectic group $Sp(2n,\mathbb F)$. The additional condition $\det g=1$ yields the special groups $SL(n)$, $SO(n)$, $SU(n)$ and $SO(p,q)$, whose Lie algebras are the traceless matrices in the corresponding algebra, since $\det\exp(A)=e^{\tr A}$. $O(n)$, $SO(n)$, $O(p,q)$ and $SO(p,q)$ may be extended by translations to the Euclidean and Poincar\'e groups $E(n)$, $SE(n)$, $E(p,q)$ and $SE(p,q)$ (\S\ref{sec:results-se}).
\par A matrix group is \textbf{compact} when its elements form a closed and bounded set for some matrix norm. The rotation and unitary groups are compact, since every entry of an orthogonal or unitary matrix has absolute value at most one. The Lorentz group is not compact: the boost with rapidity $\phi$ in the $(1+1)$ dimensional Minkowski plane,
\begin{equation}
B(\phi)=\exp\left(\phi\begin{bmatrix}0&1\\1&0\end{bmatrix}\right)=\begin{bmatrix}\cosh\phi&\sinh\phi\\ \sinh\phi&\cosh\phi\end{bmatrix},
\label{eq:boost}
\end{equation}
has entries which grow without bound as $\phi\to\infty$. The elements of $G$ joined to $I$ by a continuous path form the \textbf{identity component} $G^0$, e.g. $SO(n)$ in $O(n)$ and the proper orthochronous transformations $SO(3,1)^0$ in $O(3,1)$.
\par Every $\exp(A)$ with $A\in\mathfrak g$ lies in $G^0$, but some elements of $G^0$ are not exponentials. In the connected group $SL(2,\mathbb R)$, a real traceless $A$ has eigenvalues $\pm\mu$ or $\pm i\theta$, so $\exp(A)$ has the positive eigenvalues $e^{\pm\mu}$ or the pair $e^{\pm i\theta}$. Both equal $-1$ only when $\theta=\pi$, where $A$ has distinct eigenvalues and $\exp(A)=-I$ is diagonalizable. Hence
\begin{equation}
-\begin{bmatrix}1&1\\0&1\end{bmatrix}\in SL(2,\mathbb R)
\label{eq:notexp}
\end{equation}
is the exponential of no element of $\mathfrak{sl}(2,\mathbb R)$, and the method of \S\ref{sec:method} must be extended to reach it (\S\ref{sec:exact}).

\subsection{The matrix logarithm}
\par A matrix $L$ is a logarithm of $g$ if $\exp(L)=g$. Every invertible matrix has one: write $g=SJS^{-1}$ in Jordan form,
\begin{equation}
J=\begin{bmatrix}J_1&&\\&\ddots&\\&&J_s\end{bmatrix},\qquad
J_j=\begin{bmatrix}\lambda_j&1&&\\&\lambda_j&\ddots&\\&&\ddots&1\\&&&\lambda_j\end{bmatrix}=\lambda_jI+N_j,
\label{eq:jordan}
\end{equation}
where each $\lambda_j\ne0$ because $g$ is invertible and each $N_j$, with ones on the superdiagonal and zeros elsewhere, is nilpotent of order $n_j$, i.e. $N_j^{n_j}=0$. For a given Jordan block $\lambda_j I+N_j$, take
\begin{equation}
\log(\lambda_j I+N_j)=(\log\lambda_j)\,I+\sum_{k=1}^{n_j-1}\frac{(-1)^{k+1}}{k}\Big(\frac{N_j}{\lambda_j}\Big)^{k},
\label{eq:jordanlog}
\end{equation}
which is a finite sum since $N^n=0$. The matrix exponential of this matrix recovers $\lambda_jI + N_j$ \citep{higham2008}. Then $L=S\log(J)S^{-1}$. Each $\log\lambda$ is a choice of branch, so the matrix has infinitely many logarithms, which can also be seen directly, since $I$ commutes with $L$:
\begin{equation}
\exp(L+2\pi ik\,I)=\exp(L)\exp(2\pi ik\,I)=g\,e^{2\pi ik}I=g,\qquad k\in\mathbb Z.
\label{eq:manylogs}
\end{equation}
For a real matrix the same multiplicity appears among the real logarithms, as $R(\theta)$ has the logarithms $(\theta+2\pi k)K$: a rotation by $\theta$ is also a rotation by $\theta+2\pi k$. If every real eigenvalue of $g$ is strictly positive, then $g$ has a unique logarithm whose eigenvalues lie in the strip $|\mathrm{Im}\,\lambda|<\pi$, the \textbf{principal logarithm} $\log g$ \citep{higham2008}. Complex and positive eigenvalues cause no difficulty. Negative real ones do, as for $R(\pi)=-I$, whose smallest logarithms $\pm\pi K$ are equally close to zero and equally valid (Fig.~\ref{fig:branch}). The principal logarithm is real for real $g$, satisfies $\log(g^{-1})=-\log g$, and commutes with similarity, transposition and conjugation \citep{higham2008}. It is computed by inverse scaling and squaring \citep{almohyLog}, and the exponential by scaling and squaring \citep{almohyExp}.
\par The Jordan form shows that every invertible matrix, diagonalizable or not, has a logarithm, which the exactness statements of \S\ref{sec:method} need for elements such as Eq.~\eqref{eq:notexp}. A least squares matrix computed from data with continuously distributed noise is diagonalizable with probability one. What remains of a Jordan block under noise is ill conditioning. A perturbation of size $\varepsilon$ splits the eigenvalue of a $2\times2$ block by order $\sqrt\varepsilon$ and leaves two nearly parallel eigenvectors, while the invariant subspace they span is well conditioned. For this reason, \S\ref{sec:components} builds its projector from Schur vectors of an invariant subspace, and the proofs in the Supplementary Material are stated for invariant subspaces, which covers the exact and the noisy case alike. 

\subsection{Inner product, projection and complex groups}
\label{sec:inner}
Distances between matrices are measured with the real inner product
\begin{equation}
\langle A,B\rangle=\mathrm{Re}\,\tr(A^\dagger B),\qquad \|A\|_F=\langle A,A\rangle^{1/2},
\label{eq:inner}
\end{equation}
the Frobenius inner product for real matrices. The real part is needed because the Lie algebras of complex groups are real vector spaces ($iA\notin\mathfrak u(n)$ for $A\in\mathfrak u(n)$), and the orthogonal projection $\proj_{\mathfrak g}(L)$ onto a real subspace $\mathfrak g$ is the element of $\mathfrak g$ closest to $L$ in $\|\cdot\|_F$. For $\mathfrak{so}(2)$,
\begin{equation}
\proj_{\mathfrak{so}(2)}\begin{bmatrix}a&b\\c&d\end{bmatrix}=\frac{c-b}{2}\,K=\frac12\left(\begin{bmatrix}a&b\\c&d\end{bmatrix}-\begin{bmatrix}a&b\\c&d\end{bmatrix}^T\right),
\label{eq:proj2x2}
\end{equation}
the antisymmetric part, since the symmetric part is orthogonal to every antisymmetric matrix. Complex matrices can be handled in real arithmetic with the embedding $\varphi(A+iB)=\left[\begin{smallmatrix}A&-B\\B&A\end{smallmatrix}\right]$, which preserves sums and products and so commutes with the exponential, $\exp(\varphi(Z))=\varphi(\exp Z)$. It sends the phase $e^{i\theta}$ to the rotation $R(\theta)$. Our implementation uses this embedding, and the results of the main text are stated for complex matrices directly.

\section{The reflection $r$}
\label{sec:reflection}
This appendix proves Theorem~\ref{thm:reflection} by deriving the factor $r$ of \S\ref{sec:components} from its two jobs. First, it must negate the eigenvalues of $g_0$ which lie on or near the negative real axis, and second, it must be a member of the group. Let $W$ be the subspace belonging to the eigenvalues to be negated, the span of their eigenvectors, or of their generalized eigenvectors when eigenvalues repeat. Say that we have some projection operator $P$ such that the range of $P$ is $W$,
\begin{equation}
P^2=P,\qquad P\mathbb F^n=W.
\end{equation}
Then $r$ must have the form
\begin{equation}
r=I-2P,
\label{eq:reflection}
\end{equation}
since $Px=x$ for $x\in W$ gives $rx=-x$, and $Px=0$ on the kernel of $P$ gives $rx=x$. Every such $r$ satisfies $r^2=I-4P+4P^2=I$. By Eq.~\eqref{eq:formgroup}, $r$ is a member of $G$ exactly when $r^\star Mr=M$, and
\begin{equation}
r^\star Mr=M-2\big(P^\star M+MP\big)+4P^\star MP,
\end{equation}
so $r\in G$ if and only if
\begin{equation}
2P^\star MP=P^\star M+MP.
\label{eq:membership}
\end{equation}
Multiplying Eq.~\eqref{eq:membership} on the right by $P$ and using $P^2=P$ gives $2P^\star MP=P^\star MP+MP$, that is, $P^\star MP=MP$. Multiplying it on the left by $P^\star$ gives in the same way $P^\star MP=P^\star M$. Hence $r\in G$ if and only if
\begin{equation}
P^\star M=MP,
\label{eq:selfadjoint}
\end{equation}
the converse being immediate, since Eq.~\eqref{eq:selfadjoint} gives $P^\star MP=MP^2=MP$ and both sides of Eq.~\eqref{eq:membership} equal $2MP$. The condition $P^\star M=MP$ says that $\langle Px,y\rangle=\langle x,Py\rangle$ for all $x,y$, that is, $P$ is self adjoint for the form, and a self adjoint projector has kernel equal to the orthogonal complement of its range: for $x\in W$ and $y$ in the kernel, $\langle x,y\rangle=\langle Px,y\rangle=\langle x,Py\rangle=0$. So the second job fixes the kernel of $P$ to be
\begin{equation}
W^\perp=\{v:\ \langle w,v\rangle=0\ \text{for all }w\in W\},
\end{equation}
the orthogonal complement of $W$ with respect to the form.
\par The projector $P$ with range $W$ and kernel $W^\perp$ can be written down directly. Let $Q$ be an $n\times k$ matrix whose columns are a basis of $W$. For any vector $v$, the vector $Pv$ lies in $W$, so $Pv=Qx$ for some coefficients $x\in\mathbb F^k$, and $v-Pv$ lies in the kernel $W^\perp$, so it is orthogonal to every column of $Q$,
\begin{equation}
Q^\star M(v-Qx)=0,\qquad\text{that is,}\qquad Q^\star MQ\,x=Q^\star Mv.
\label{eq:normal}
\end{equation}
The $k\times k$ matrix $Q^\star MQ$ is the Gram matrix of the form on $W$, invertible when the form is nondegenerate on $W$ \citep[Chapter 2]{glr2005}, so $x=(Q^\star MQ)^{-1}Q^\star Mv$, and $Pv=Qx$ for every $v$ gives
\begin{equation}
P=Q\big(Q^\star MQ\big)^{-1}Q^\star M.
\label{eq:projector}
\end{equation}
Eq.~\eqref{eq:normal} is the normal equation of least squares with the form in place of the Euclidean inner product. The form is nondegenerate on every $W$ used in \S\ref{sec:components}. One checks directly that $P^2=P$, $PQ=Q$ and $PW^\perp=0$. Self adjointness follows from $M^\star=\pm M$: with the same sign in both places,
\begin{equation}
P^\star=M^\star Q\big((Q^\star MQ)^\star\big)^{-1}Q^\star=(\pm M)Q\big(\pm Q^\star MQ\big)^{-1}Q^\star=MQ\big(Q^\star MQ\big)^{-1}Q^\star,
\end{equation}
and multiplying by $M$ on the right gives $P^\star M=MQ(Q^\star MQ)^{-1}Q^\star M=MP$.
\par For $M=I$ and orthonormal $Q$, the Gram matrix is the identity and $P=QQ^\dagger$, the familiar orthogonal projector, so that $r=I-2QQ^\dagger$ is a block Householder reflection. For a general $M$, Eq.~\eqref{eq:projector} is the same reflection with the Euclidean inner product replaced by the form, a $G$-reflector in the terminology of \citet{mmt2004}, who develop these reflections for every scalar product of this kind. Nothing here uses $g_0$, so $r$ lies in $G$ exactly for exact and for noisy data alike.
\par For the last statement of the theorem, let $W$ be the invariant subspace of $g\in G_M$ for a set of its eigenvalues. Two eigenvectors $x$ and $y$ of $g$, with eigenvalues $\lambda$ and $\mu$, satisfy
\begin{equation}
\langle x,y\rangle=\langle gx,gy\rangle=\lambda^\star\mu\,\langle x,y\rangle,
\label{eq:eigorth}
\end{equation}
where $\lambda^\star$ is $\lambda$ for a bilinear form and $\bar\lambda$ for a sesquilinear one, so they are orthogonal for the form unless $\lambda^\star\mu=1$. For the choice $W=N$ of Eq.~\eqref{eq:involution}, $\lambda$ lies in the open left half plane and $\mu$ does not, and $\lambda^\star\mu=1$ never holds, because $\lambda\ne0$ for the invertible $g$, and the real part of $1/\lambda^\star$ has the sign of the real part of $\lambda$. For a general nondegenerate $W$, the eigenvalues of $g$ on $W$ are closed under $\lambda\mapsto1/\lambda^\star$, since an eigenvalue whose partner lies outside $W$ would have its generalized eigenspace orthogonal to all of $W$ and make $W$ degenerate. In both cases the complementary invariant subspace is $W^\perp$ (for a non-diagonalizable $g$, the proof of Theorem~\ref{thm:global}(a) in the Supplement shows this directly for the spectral projector, which is a polynomial in $g$), $P$ is the spectral projector onto $W$, which commutes with $g$, and $rg=g-2Pg$ negates the eigenvalues in $W$ and leaves the others unchanged.

\section{Worked examples}
\label{sec:examples}
Two numerical examples follow Algorithm~\ref{alg:global} step by step on noisy data, one for the reflection $r=I-2P$ of Eq.~\eqref{eq:sign} in $SO(3)$ and one for the reflection $r=I-2Q_SQ_S^\dagger$ of Eq.~\eqref{eq:flip} in $SU(3)$. The scripts \texttt{reproduce/worked\_example.jl} and \texttt{reproduce/worked\_example\_su3.jl} of the archive print every quantity shown.
\par For the real groups, let $g_\ast\in SO(3)$ be the rotation by $\pi$ about the axis $(1,1,1)/\sqrt3$, and let $Y=g_\ast X+E$ with four Gaussian points and Gaussian noise of standard deviation $0.05$, rounded to two decimals,
\begin{equation}
X=\begin{bmatrix}-0.14&-1.34&0.84&-0.63\\0.96&-2.22&-0.07&0.39\\-0.76&1.05&0.26&2.10\end{bmatrix},\qquad
Y=\begin{bmatrix}0.21&-0.46&-0.15&1.92\\-0.78&0.61&0.75&0.76\\0.72&-2.80&0.32&-0.96\end{bmatrix}.
\label{eq:example}
\end{equation}
The least squares matrix and its eigenvalues are
\begin{equation}
g_0=YX^+=\begin{bmatrix}-0.321&0.723&0.684\\0.595&-0.335&0.601\\0.665&0.661&-0.378\end{bmatrix},\qquad
\lambda=-1.021,\ -0.976,\ 0.963.
\end{equation}
The two eigenvalues near $-1$ have split into distinct negative reals, so $g_0$ has no real logarithm and Algorithm~\ref{alg:naive} is undefined. The projected real part of the complex principal logarithm is a rotation by $0.01$, at distance $2.83$ from $g_\ast$. The subspace $N$ is the plane spanned by the eigenvectors of the two negative eigenvalues, $Q_N$ an orthonormal basis of it, and
\begin{equation}
r=I-2Q_NQ_N^T=\begin{bmatrix}-0.352&0.685&0.637\\0.685&-0.275&0.674\\0.637&0.674&-0.373\end{bmatrix},\qquad
r^{-1}g_0=\begin{bmatrix}0.945&-0.063&-0.070\\0.065&1.033&0.049\\-0.051&-0.012&0.982\end{bmatrix},
\end{equation}
with eigenvalues $0.963$, $0.976$ and $1.021$: $r$ is itself a rotation by $\pi$ about the normal of $N$, with $r^Tr=I$ and $\det r=(-1)^{\dim N}=1$, and $r^{-1}g_0$ is close to $I$, so its principal logarithm is real and small. Its projection onto $\mathfrak{so}(3)$ and the output are
\begin{align}
\proj_{\mathfrak{so}(3)}\log(r^{-1}g_0)=\begin{bmatrix}0&-0.065&-0.009\\0.065&0&0.032\\0.009&-0.032&0\end{bmatrix},\\
g=r\exp\big(\proj_{\mathfrak{so}(3)}\log(r^{-1}g_0)\big)\\=\begin{bmatrix}-0.301&0.686&0.662\\0.671&-0.340&0.658\\0.677&0.643&-0.358\end{bmatrix},
\end{align}
for which numerical verification gives $\|g^Tg-I\|_F=8\times10^{-16}$ and $\det g=1$ to twelve digits. The output is a rotation by $3.129$ about an axis within $0.03$ of $(1,1,1)/\sqrt3$, with $\|g-g_\ast\|_F=0.053$. The rotation by $\pi$ is carried by $r$, and the logarithm supplies only the small correction that the noise calls for.
\par For the complex special groups, take $G=SU(3)$, let $g_\ast=\omega I$ and $Y=g_\ast X+E$ with four complex Gaussian vectors and complex Gaussian noise of standard deviation $0.05$, rounded to two decimals,
\begin{equation}
X=\begin{bmatrix}-0.10+0.68i&0.59-0.05i&0.37+1.91i&0.12-0.12i\\-0.54-0.95i&0.18-0.45i&-1.18-1.73i&-1.47+0.67i\\-1.57+0.74i&0.28+1.49i&0.95-1.11i&-1.28-1.41i\end{bmatrix},
\label{eq:example-su}
\end{equation}
\begin{equation}
Y=\begin{bmatrix}-0.46-0.45i&-0.25+0.47i&-1.80-0.66i&0.01+0.20i\\1.11-0.04i&0.33+0.36i&2.14-0.24i&0.12-1.55i\\0.17-1.72i&-1.43-0.54i&0.53+1.37i&1.87-0.41i\end{bmatrix}.
\end{equation}
The least squares matrix is
\begin{equation}
g_0=YX^+=\begin{bmatrix}-0.49+0.83i&0.02+0.01i&-0.03-0.01i\\-0.01-0.05i&-0.47+0.86i&-0.00-0.00i\\-0.00-0.03i&-0.00+0.00i&-0.51+0.86i\end{bmatrix},
\end{equation}
with eigenvalues of arguments $0.675\pi$, $0.663\pi$ and $0.660\pi$, which sum to $1.998\pi$, so $k=1$. No eigenvalue is on the branch cut and the principal logarithm exists, but its trace is $2\pi i$ up to noise. Removing it, Algorithm~\ref{alg:naive} returns a matrix within $0.03$ of $I$, at distance $3.00$ from $g_\ast$. Eq.~\eqref{eq:flipcount} requires two of the three eigenvalues to be negated, the method takes the two of largest argument, and with $Q_S$ the corresponding two Schur vectors,
\begin{equation}
r=I-2Q_SQ_S^\dagger=\begin{bmatrix}-0.02-0.00i&-0.19+0.87i&-0.24-0.37i\\-0.19-0.87i&-0.18-0.00i&-0.29+0.29i\\-0.24+0.37i&-0.29-0.29i&-0.80+0.00i\end{bmatrix},
\end{equation}
a unitary reflection with eigenvalues $-1,-1,1$ and $\det r=1$. The eigenvalues of $r^{-1}g_0$ have arguments $0.660\pi$, $-0.337\pi$ and $-0.325\pi$, which sum to $-0.002\pi$: the two negated eigenvalues moved from near $2\pi/3$ to near $-\pi/3$ and the third stayed. The principal logarithm of $r^{-1}g_0$ has trace $-0.05-0.01i$ and Hermitian part of norm $0.16$, both noise, which the projection onto $\mathfrak{su}(3)$ removes, and the output is
\begin{equation}
g=r\exp\big(\proj_{\mathfrak{su}(3)}\log(r^{-1}g_0)\big)=\begin{bmatrix}-0.51+0.86i&-0.01+0.02i&-0.03+0.00i\\0.01-0.02i&-0.48+0.88i&-0.00-0.00i\\-0.01-0.02i&-0.00-0.00i&-0.51+0.86i\end{bmatrix},
\end{equation}
for which numerical verification gives $\|g^\dagger g-I\|_F=2\times10^{-15}$ and $\det g=1$ to twelve digits, with $\|g-g_\ast\|_F=0.054$.
\par Table~\ref{tab:examples} collects the errors of both examples, as matrices, in the transformed vectors $gX$, and as residuals against the measured $Y$. To first order $g-g_\ast$ is the projection of the noise $EX^+$ onto the tangent space of the group (\S\ref{sec:compact}). The residual of the estimate against $Y$ is smaller than that of $g_\ast$ itself, since the estimate has absorbed part of the noise.
\begin{table}[t]
\centering
\caption{Errors in the two worked examples of this appendix: the distance of the output $g$ of Algorithm~\ref{alg:global} from the true element, the same distance after applying the matrices to the inputs $X$, and the residuals of $g$ and of $g_\ast$ against the measured $Y$, all in the Frobenius norm.}
\label{tab:examples}
\footnotesize
\begin{tabular}{lcccc}
\hline
Example & $\|g-g_\ast\|$ & $\|(g-g_\ast)X\|$ & $\|gX-Y\|$ & $\|g_\ast X-Y\|$\\
\hline
$SO(3)$, Eq.~\eqref{eq:example} & $0.053$ & $0.122$ & $0.250$ & $0.291$\\
$SU(3)$, Eq.~\eqref{eq:example-su} & $0.054$ & $0.130$ & $0.128$ & $0.195$\\
\hline
\end{tabular}
\end{table}

\section*{Supplementary material}
Supplementary material associated with this article, containing the proofs of Sections~\ref{sec:method} and \ref{sec:noise}, and the error histograms for eighteen groups, can be found in the online version.

\clearpage
\setcounter{section}{0}\setcounter{figure}{0}\setcounter{table}{0}\setcounter{equation}{0}
\renewcommand{\thesection}{S\arabic{section}}
\renewcommand{\thefigure}{S\arabic{figure}}
\renewcommand{\thetable}{S\arabic{table}}
\begin{center}\LARGE Supplementary Material\end{center}
\noindent Numbers without the prefix S refer to the main text.

\section{Proofs for Section~\ref{sec:method}}
\label{sec:details}

\subsection{The spin algebra}
We represent $\mathrm{Spin}(n)$ by $d\times d$ matrices, $d=2^{\lfloor n/2\rfloor}$, with Hermitian generators built from the Pauli matrices as in \citet{lawsonMichelsohn}: for $k=\lfloor n/2\rfloor$ and $a=1,\dots,k$, $\gamma_{2a-1}=\sigma_z^{\otimes(a-1)}\otimes\sigma_x\otimes I_2^{\otimes(k-a)}$ and $\gamma_{2a}=\sigma_z^{\otimes(a-1)}\otimes\sigma_y\otimes I_2^{\otimes(k-a)}$, with $\gamma_n=\sigma_z^{\otimes k}$ for odd $n$. The Lie algebra is spanned by $S_{ij}=\frac12\gamma_i\gamma_j$, $i<j$.
\begin{proposition}[Projection onto the spin algebra]\label{prop:spin}
The matrices $B_{ij}=\frac{2}{\sqrt d}S_{ij}$, $i<j$, are orthonormal with respect to Eq.~\eqref{eq:inner}, and $\proj_{\mathfrak{spin}(n)}(L)=\sum_{i<j}\langle B_{ij},L\rangle B_{ij}$.
\end{proposition}
\begin{proof}
$\langle S_{ij},S_{kl}\rangle=\frac14\mathrm{Re}\,\tr(\gamma_j\gamma_i\gamma_k\gamma_l)$ equals $d/4$ for $(i,j)=(k,l)$. Otherwise, after cancelling repeated factors with $\gamma_i^2=I$, the product is $\pm$ a product of two or four distinct generators, whose trace vanishes, since moving one factor cyclically through the odd number of others changes the sign of the trace. The residual $L-\sum\langle B_{ij},L\rangle B_{ij}$ is then orthogonal to every $B_{kl}$.
\end{proof}

\subsection{Covering the whole group}
\label{sec:supp-global}
$\log$ of a matrix denotes the principal logarithm, real for a real matrix without negative real eigenvalues \citep[Theorem 1.31]{higham2008}.

\begin{proof}[Proof of Theorem~\ref{thm:global}(a)]
Let $N$ be the sum of the generalized eigenspaces of $g_\ast$ for the eigenvalues with $\mathrm{Re}\,\lambda<0$, as in Eq.~\eqref{eq:involution}, and $C$ the sum for the remaining eigenvalues. Both are $g_\ast$-invariant and real for real $g_\ast$. With $P$ the projection onto $N$ along $C$ and $r=I-2P$, $r$ commutes with $g_\ast$, and $rg_\ast$ has the eigenvalues $-\lambda$ on $N$ and $\lambda$ on $C$, none a negative real number. Once $r\in G$ is known, $g'=r^{-1}g_\ast\in G$, $\log g'$ exists and lies in $\mathfrak g$ \citep{hmmt2005}, so the method returns $\exp(\log g')=g'$ and $rg'=g_\ast$.
\par\emph{Determinant.} For real $g_\ast$, with $m$ the number of negative real eigenvalues and $c$ the number of conjugate pairs in $N$, counted with multiplicity,
\begin{equation}
\mathrm{sign}\det g_\ast=(-1)^m,\qquad \dim N=m+2c,\qquad \det r=(-1)^{\dim N}=(-1)^m,
\end{equation}
so $\det g_\ast>0$ forces $\det r=1$. This is used below for the special subgroups.
\par\emph{Groups preserving no form.} For $GL(n,\mathbb R)$, $SL(n,\mathbb R)$ and $GL(n,\mathbb C)$, let $Q$ be the leading Schur vectors of $g_\ast$ for the eigenvalues in $N$. Write the Schur form as $g_\ast=ZTZ^\dagger$ with $Z=[Q\ Q']$ unitary, real for the real Schur form, and $T=\left[\begin{smallmatrix}T_{11}&T_{12}\\0&T_{22}\end{smallmatrix}\right]$, where $T_{11}$ carries the eigenvalues in $N$. Then $r=I-2QQ^\dagger=Z\,\mathrm{diag}(-I,I)\,Z^\dagger$ is unitary, real for real $Z$, and
\begin{equation}
rg_\ast=Z\,\mathrm{diag}(-I,I)\,Z^\dagger ZTZ^\dagger=Z\begin{bmatrix}-I&0\\0&I\end{bmatrix}\begin{bmatrix}T_{11}&T_{12}\\0&T_{22}\end{bmatrix}Z^\dagger=Z\begin{bmatrix}-T_{11}&-T_{12}\\0&T_{22}\end{bmatrix}Z^\dagger,
\end{equation}
which is block triangular with diagonal blocks $-T_{11}$ and $T_{22}$, so it has the eigenvalues of $g_\ast$ with those in $N$ negated. Every invertible $r$ lies in $GL$.
\par\emph{Groups preserving a form.} The projection $P$ onto $N$ along $C$ is a polynomial in $g_\ast$, $P=p(g_\ast)$. Since $g_\ast^\star M=Mg_\ast^{-1}$, the adjoint for the form of a polynomial in $g_\ast$ is the polynomial in $g_\ast^{-1}$ with the same coefficients for $\star=T$ and conjugated coefficients for $\star=\dagger$, so $P^\star M=M\bar p(g_\ast^{-1})$. The matrix $\bar p(g_\ast^{-1})$ is the spectral projection of $g_\ast^{-1}$ for its eigenvalues $1/\lambda^\star$ with $\mathrm{Re}\,\lambda<0$, where $\lambda^\star$ is $\lambda$ for $\star=T$ and $\bar\lambda$ for $\star=\dagger$, and since $\lambda\ne0$ for the invertible $g_\ast$ and $\mathrm{Re}(1/\lambda^\star)$ has the sign of $\mathrm{Re}\,\lambda$, it is the projection onto the same $N$ along the same $C$, that is, $P^\star M=MP$. Hence $r^\star Mr=M(I-2P)^2=M$ and $r\in G_M$. This covers $O(n)$, $O(p,q)$, $U(n)$, $U(p,q)$, $O(n,\mathbb C)$, $Sp(2n,\mathbb R)$ and $Sp(2n,\mathbb C)$.
\par\emph{Special subgroups.} For $SL(n,\mathbb R)$, $SO(n)$ and $SO(p,q)$, $r$ lies in $GL(n,\mathbb R)$, $O(n)$ or $O(p,q)$ by the two preceding paragraphs, and $\det r=1$ by the determinant count, since $\det g_\ast=1$, so $r$ lies in the special subgroup. For $SO(n,\mathbb C)$, the eigenvalues of $g_\ast$ pair as $\lambda$ and $1/\lambda$, both in $N$ or both outside, so $\dim N$ has the parity of the multiplicity of $-1$, which is even when $\det g_\ast=1$, and again $\det r=1$. In every case $\det g'=\det g_\ast=1$, so $\tr\log g'=\log\det g'=0$ and $\log g'$ lies in the special algebra.
\par\emph{Euclidean and similarity groups.} These groups act on the direct sum $\mathbb R^n\oplus\mathbb R$ of the original space and the homogeneous coordinate, by $g_\ast=\left[\begin{smallmatrix}R&t\\0&1\end{smallmatrix}\right]$. The method uses the involution $r=r_R\oplus1$, with $r_R$ the involution of the block $R$ from the cases above, which lies in the group. Then $r^{-1}g_\ast=\left[\begin{smallmatrix}r_RR&r_Rt\\0&1\end{smallmatrix}\right]$ is block triangular with diagonal blocks $r_RR$ and $1$, neither with a negative real eigenvalue, and its principal logarithm is block triangular with diagonal blocks $\log(r_RR)$ and $0$ \citep{hmmt2005}, which lies in the Lie algebra. For $\mathrm{Sim}(n)$, $R=sR'$ with $s>0$ and $\log(r_RR)=(\log s)I+\log(r_RR')$.
\end{proof}

\begin{proof}[Proof of Theorem~\ref{thm:global}(b)]
Let $\theta_i\in(-\pi,\pi]$ be the arguments of the eigenvalues $\lambda_i$ of $g_\ast$ and $s_i=-\mathrm{sign}\,\theta_i$. Since $\det g_\ast=1$, $\sum_i\theta_i=2\pi k$ with $k\in\mathbb Z$. Let $m$ be the number of eigenvalues on the branch cut, $\theta_i=\pi$, and $n_+$ and $n_-$ the numbers of positive and negative arguments. The set $S$ consists of the $b$ eigenvalues of largest positive argument together with the $a$ eigenvalues of largest negative argument, for a choice of $a$ and $b$ with
\begin{equation}
b\ge m,\qquad b\le n_+,\qquad a\le n_-,\qquad a-b=-2k,
\end{equation}
so that $S$ contains the eigenvalues on the branch cut, $\sum_{i\in S}s_i=a-b=-2k$, and $|S|=a+b=2b-2k$ is even. Since every eigenvalue in $S$ has a nonzero argument, $S$ contains no positive real eigenvalue.
\par\emph{Existence of $S$.} For $k\ge0$, take $b=\max(m,2k)$ and $a=b-2k$. If $b=2k$, then $b\le n_+$, because $2\pi k\le\sum_{\theta_i>0}\theta_i\le n_+\pi$. If $b=m>2k$, then
\begin{equation}
\sum_{\theta_i<0}\theta_i=2\pi k-\sum_{\theta_i>0}\theta_i\le2\pi k-m\pi=-a\pi,
\end{equation}
while each negative argument exceeds $-\pi$, so $a<n_-$. For $k<0$, take $b=m$ and $a=m-2k$, and the same bound gives $a\le n_-$.
\par\emph{The reflection.} Order a Schur form $g_\ast=QTQ^\dagger$ so that $S$ comes first and let $Q_S$ be its first $|S|$ columns. Then $r=I-2Q_SQ_S^\dagger$ is unitary with $r^2=I$ and $\det r=(-1)^{|S|}=1$, so $r\in SU(n)$. In the Schur basis, $r=\mathrm{diag}(-I,I)$, so $rg_\ast$ has the eigenvalues $-\lambda_i$ for $i\in S$, with arguments $\theta_i+s_i\pi\in(-\pi,\pi)$, and the eigenvalues $\lambda_i\ne-1$ for $i\notin S$.
\par\emph{The logarithm.} The principal logarithm $L$ of $rg_\ast$ exists, since no eigenvalue is on the branch cut, and
\begin{equation}
\tr L=i\sum_i\theta_i+i\pi\sum_{i\in S}s_i=2\pi ik-2\pi ik=0.
\end{equation}
For $SU(n)$, $L$ is skew Hermitian because $rg_\ast$ is unitary \citep{hmmt2005}. Hence $L\in\mathfrak g$, $\exp(\proj_{\mathfrak g}L)=rg_\ast$, and the method returns $r\cdot rg_\ast=g_\ast$. The last sentence of the theorem restates (a) and (b).
\end{proof}

\begin{proof}[Proof of Theorem~\ref{thm:global}(c)]
Let $\pi\colon\mathrm{Pin}(n)\to O(n)$ be the twisted adjoint representation, defined by $\alpha(g)\gamma_bg^{-1}=\sum_a\pi(g)_{ab}\gamma_a$, which is a surjective homomorphism with kernel $\{\pm I\}$ which restricts to the adjoint representation of $\mathrm{Spin}(n)$ onto $SO(n)$ \citep[Chapter I, \S2]{lawsonMichelsohn}. With $\alpha(g)=\Gamma g\Gamma$, a unit vector $\tilde u$ gives $\alpha(\tilde u)\gamma_b\tilde u^{-1}=-\tilde u\gamma_b\tilde u$, which equals $\gamma_b-2u^b\tilde u$, the reflection $I-2uu^T$, and $\tr(\gamma_a\gamma_c)=d\,\delta_{ac}$ gives Eq.~\eqref{eq:covering}. Let $R=\pi(g_\ast)=r_R\exp(A)$ be the factorization of part (a), with $r_R$ the reflection of the subspace $N$ of $R$ of dimension $k$, even when $\det R=1$, and $A=\log(r_RR)\in\mathfrak{so}(n)$.
\par From $\gamma_a\gamma_b\gamma_c-\gamma_c\gamma_a\gamma_b=2\delta_{bc}\gamma_a-2\delta_{ac}\gamma_b$ and $A^T=-A$, the bivector $B=\frac14\sum_{a,b}A_{ab}\gamma_a\gamma_b$ satisfies $[B,\gamma_c]=\sum_aA_{ac}\gamma_a$, hence $e^{B}\gamma_ce^{-B}=\sum_a(e^{A})_{ac}\gamma_a$ and $\pi(e^B)=e^A$. For an orthonormal basis $u_1,\dots,u_k$ of $N$, the $\tilde u_j=\sum_au_j^a\gamma_a$ are Hermitian, anticommute and square to $I$, each covers the reflection $I-2u_ju_j^T$, and $\tilde r=\tilde u_1\cdots\tilde u_k$ covers the product of these commuting reflections, $r_R$. For even $k$, $\tilde u_{2j-1}\tilde u_{2j}=\exp(\frac\pi2\tilde u_{2j-1}\tilde u_{2j})$, so $\tilde r\in\mathrm{Spin}(n)$.
\par Hence $\pi(\tilde re^B)=\pi(g_\ast)$ and $g_\ast=\zeta\tilde re^B$ with $\zeta=\pm1$. On noiseless data $R_0=R$, so the method computes $r_R$, $A$, $\tilde r$ and $B$ exactly, the sign maximizing $\mathrm{Re}\tr((\tilde re^BX)^\dagger Y)=\zeta\|X\|_F^2$ is the correct one, $e^{-B}\tilde r^{-1}\zeta g_0=I$ has logarithm $0$, and the output is $\zeta\tilde re^B=g_\ast$.
\end{proof}

\section{Proofs for Section~\ref{sec:noise}}
\label{sec:proofs}
The Fr\'echet derivatives of the logarithm and the exponential are \citep[Chapters 10 and 11]{higham2008}
\begin{equation}
L_{\log}(g,Z)=\int_0^1\big(t(g-I)+I\big)^{-1}Z\,\big(t(g-I)+I\big)^{-1}dt,\qquad L_{\exp}(A,Z)=\int_0^1e^{(1-s)A}Ze^{sA}\,ds,
\label{eq:dlog}
\end{equation}
so $\|L_{\log}(g,Z)\|_F\le\kappa_{\log}(g)\|Z\|_F$ with $\kappa_{\log}$ as in Eq.~\eqref{eq:kappalog} and $\|L_{\exp}(A,Z)\|_F\le e^{\|A\|_2}\|Z\|_F$. For unitary $g_\ast$ with eigenvalues $e^{i\theta_j}$, $\|(t(g_\ast-I)+I)^{-1}\|_2^{-2}=\min_j\big(1-2t(1-t)(1-\cos\theta_j)\big)$, and integration gives $\kappa_{\log}(g_\ast)=\theta_{\max}/\sin\theta_{\max}$ with $\theta_{\max}=\max_j|\theta_j|$. For a real eigenvalue $\lambda>0$ the integral gives $1/\lambda$, so a boost with rapidity $\phi$ has $\kappa_{\log}=e^{\phi}$.
\par The Gauss--Newton refinement rests on the identity $\|gX-Y\|_F^2=\|(g-g_0)X\|_F^2+\|g_0X-Y\|_F^2$, which follows from the normal equation $(g_0X-Y)X^\dagger=0$: the least squares element of $G$ is the one closest to $g_0$ in the distance weighted by the data, whereas Eq.~\eqref{eq:method} projects in the unweighted inner product. The matrix $D$ of Eq.~\eqref{eq:corrstep} has full column rank when $X$ has full row rank, since $\delta gX=0$ implies $\delta=0$.

\begin{proof}[Proof of Theorem~\ref{thm:firstorder}]
The eigenvalues $e^{\lambda}$ of $g_\ast$ with $|\mathrm{Im}\,\lambda|<\pi$ are not negative reals, so the principal logarithm is smooth near $g_\ast$ with $\log g_\ast=A$, and $F(h)=\exp(\proj_{\mathfrak g}(\log h))$ is smooth near $g_\ast$ with $F(g_\ast)=g_\ast$, since $A\in\mathfrak g$. By the chain rule and the linearity of the projection, $L_F(g_\ast,Z)=L_{\exp}(A,\proj_{\mathfrak g}(L_{\log}(g_\ast,Z)))$, and Taylor's theorem with $g_0-g_\ast=EX^+$ gives Eq.~\eqref{eq:firstorder}. The bound follows from the two bounds after Eq.~\eqref{eq:dlog}, the projection not increasing the norm, and $\|EX^+\|_F\le\|E\|_F/\sigma_{\min}(X)$.
\end{proof}

\begin{proof}[Proof of Theorem~\ref{thm:compact}]
Write $\Delta=EX^+=\Delta_{\mathcal T}+\Delta_\perp$ with $\Delta_{\mathcal T}\in\mathcal T$ and $\Delta_\perp\perp\mathcal T$. We show that $L_{\log}(g_\ast,\cdot)$ maps $\mathcal T$ into $\mathfrak g$ and $\mathcal T^\perp$ into $\mathfrak g^\perp$, so that the projection in Eq.~\eqref{eq:firstorder} discards $\Delta_\perp$.
\par For $K\in\mathfrak g$ and small $t$, $h(t)=\exp(tK)g_\ast\in G$ is close to $g_\ast$, so $\log h(t)\in\mathfrak g$ \citep{hmmt2005}, where for $SU(n)$ the trace stays zero because it is continuous in $t$ and an integer multiple of $2\pi i$. Differentiating at $t=0$ gives $L_{\log}(g_\ast,Kg_\ast)\in\mathfrak g$.
\par Since right multiplication by the unitary $g_\ast$ preserves the inner product, $\mathcal T^\perp=\{Hg_\ast:H\in\mathfrak g^\perp\}$, with $H$ Hermitian for $U(n)$, real symmetric for $SO(n)$, and in the real span of the Hermitian matrices and $iI$ for $SU(n)$. Diagonalize $g_\ast=UDU^\dagger$ with $D=\mathrm{diag}(e^{i\theta_j})$, $\theta_j\in(-\pi,\pi)$. By the Daleckii--Krein formula \citep[Theorem 3.11]{higham2008}, $L_{\log}(g_\ast,Z)=U\big[(U^\dagger ZU)\circ\Phi\big]U^\dagger$ with $\Phi_{jk}=(i\theta_j-i\theta_k)/(e^{i\theta_j}-e^{i\theta_k})$, or $e^{-i\theta_j}$ when $\theta_j=\theta_k$. For $Z=Hg_\ast$ and the Hermitian $W=U^\dagger HU$, the $(j,k)$ entry of $(WD)\circ\Phi$ is $W_{jk}\,\psi(\theta_j-\theta_k)$ with $\psi(x)=ix/(e^{ix}-1)$, $\psi(0)=1$. Since $\overline{\psi(x)}=\psi(-x)$, this matrix is Hermitian, so $L_{\log}(g_\ast,Hg_\ast)$ is Hermitian, and real for $SO(n)$. For $H=icI$ in $SU(n)$, $\log(e^{ict}g_\ast)=ictI+\log g_\ast$ for small $t$, so $L_{\log}(g_\ast,icg_\ast)=icI\perp\mathfrak{su}(n)$.
\par Hence $\proj_{\mathfrak g}(L_{\log}(g_\ast,\Delta))=L_{\log}(g_\ast,\Delta_{\mathcal T})$, and since $\exp\circ\log$ is the identity near $g_\ast$, $L_{\exp}(A,L_{\log}(g_\ast,Z))=Z$, so the right hand side of Eq.~\eqref{eq:firstorder} is $\Delta_{\mathcal T}=\Pi_{\mathcal T}(EX^+)$, of norm at most $\|EX^+\|_F$.
\end{proof}

\begin{proof}[Proof of Corollary~\ref{cor:isotropic}]
With $h=g_\ast+\tau$, $\tau\in\mathcal T$, the objective is $\|\tau X-E\|_F^2$ to first order, whose minimizer over $\mathcal T$ is unique because $\tau X=0$ implies $\tau=0$. By the implicit function theorem a smooth branch of local minimizers passes through $g_\ast$ with this derivative at $E=0$, and since $G$ is compact and $g_\ast$ is the unique minimizer at $E=0$, the global minimizer $\hat g$ lies on this branch for small $E$. Since $X^+=X^\dagger/c$, $X^\dagger X/c$ is the orthogonal projector onto the row space of $X$, so $E$ splits into $EX^+X$ and a part orthogonal to every $\tau X$, and $\|\tau X-EX^+X\|_F^2=c\,\|\tau-EX^+\|_F^2$. The minimizer is therefore $\Pi_{\mathcal T}(EX^+)$, the first order term of Eq.~\eqref{eq:compact}.
\end{proof}

\clearpage
\section{Error histograms}
Figures~S1--S18 show histograms of $\|g-g_\ast\|_F$ over 1{,}000 trials for the eleven groups of the cross-group study and seven further groups, one of each remaining type implemented in the package, with rows the Lie algebra method, the Gauss--Newton refinement, Newton's method and BFGS, and columns the noise levels $\epsilon=0$, $0.01$ and $0.05$, sharing the horizontal axis within each column.

\begin{figure}[p]\centering
\includegraphics[width=\linewidth]{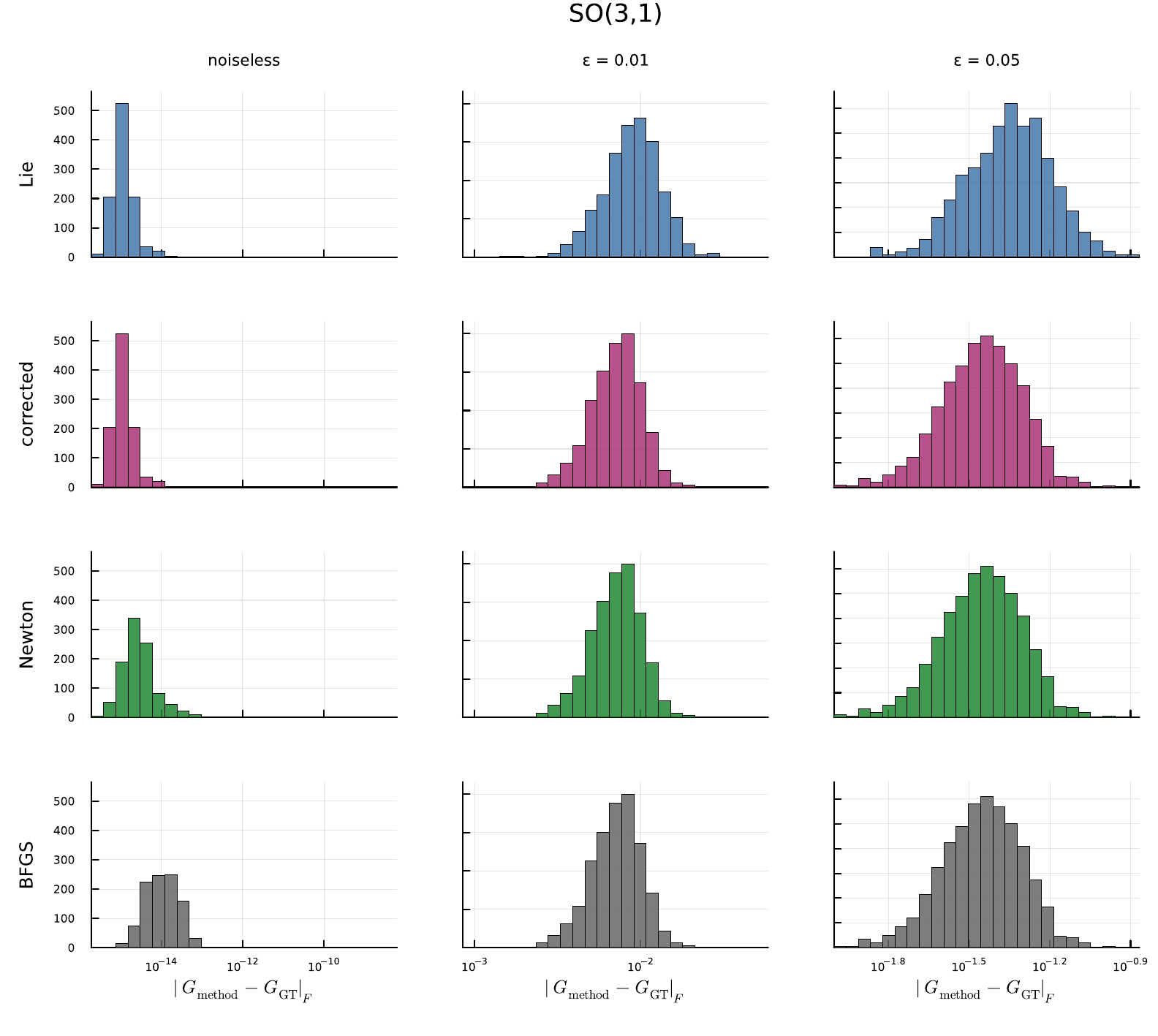}
\caption{Recovery error for the Lorentz group $SO(3,1)$.}
\end{figure}

\begin{figure}[p]\centering
\includegraphics[width=\linewidth]{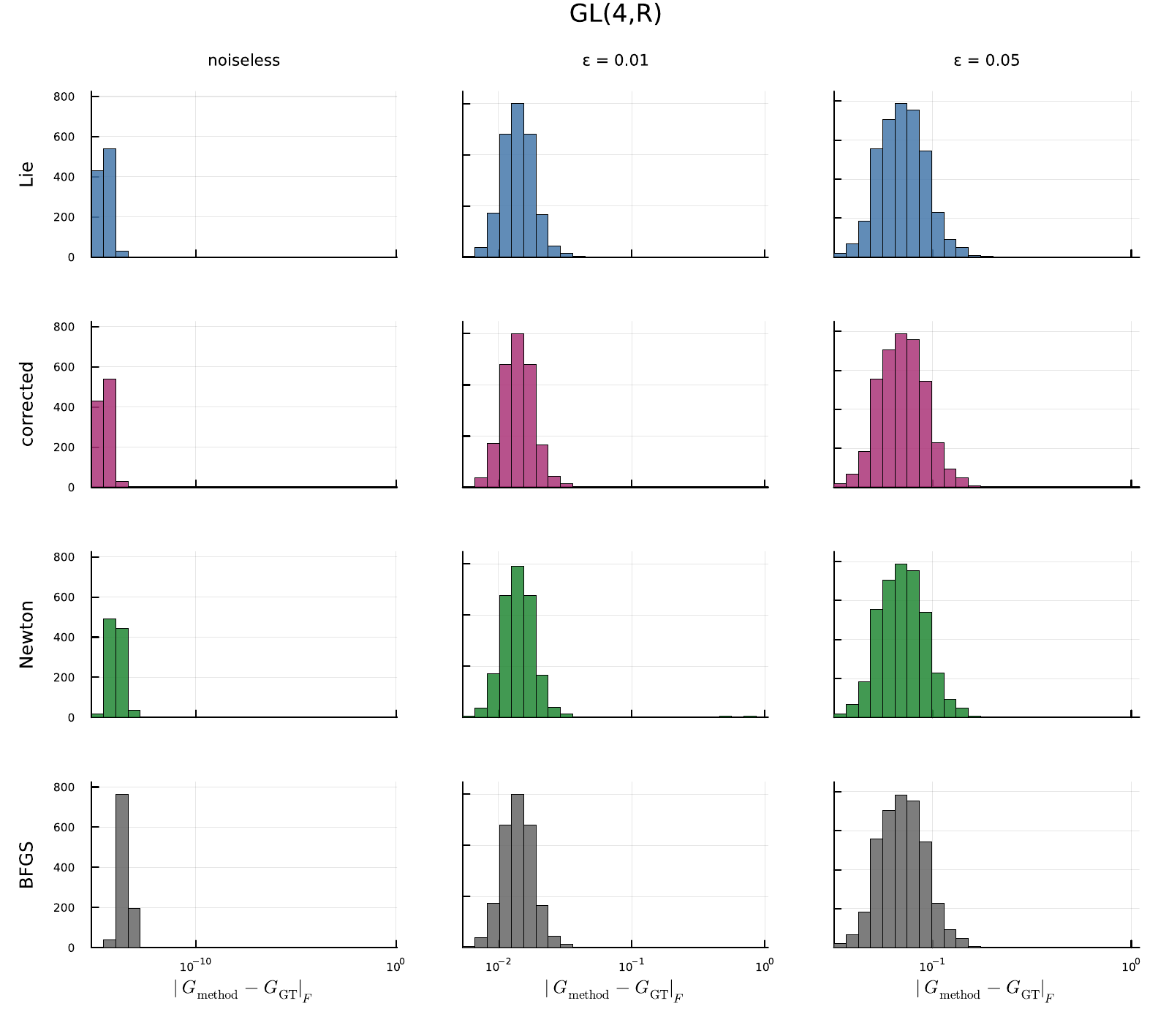}
\caption{Recovery error for $GL(4,\mathbb R)$.}
\end{figure}

\begin{figure}[p]\centering
\includegraphics[width=\linewidth]{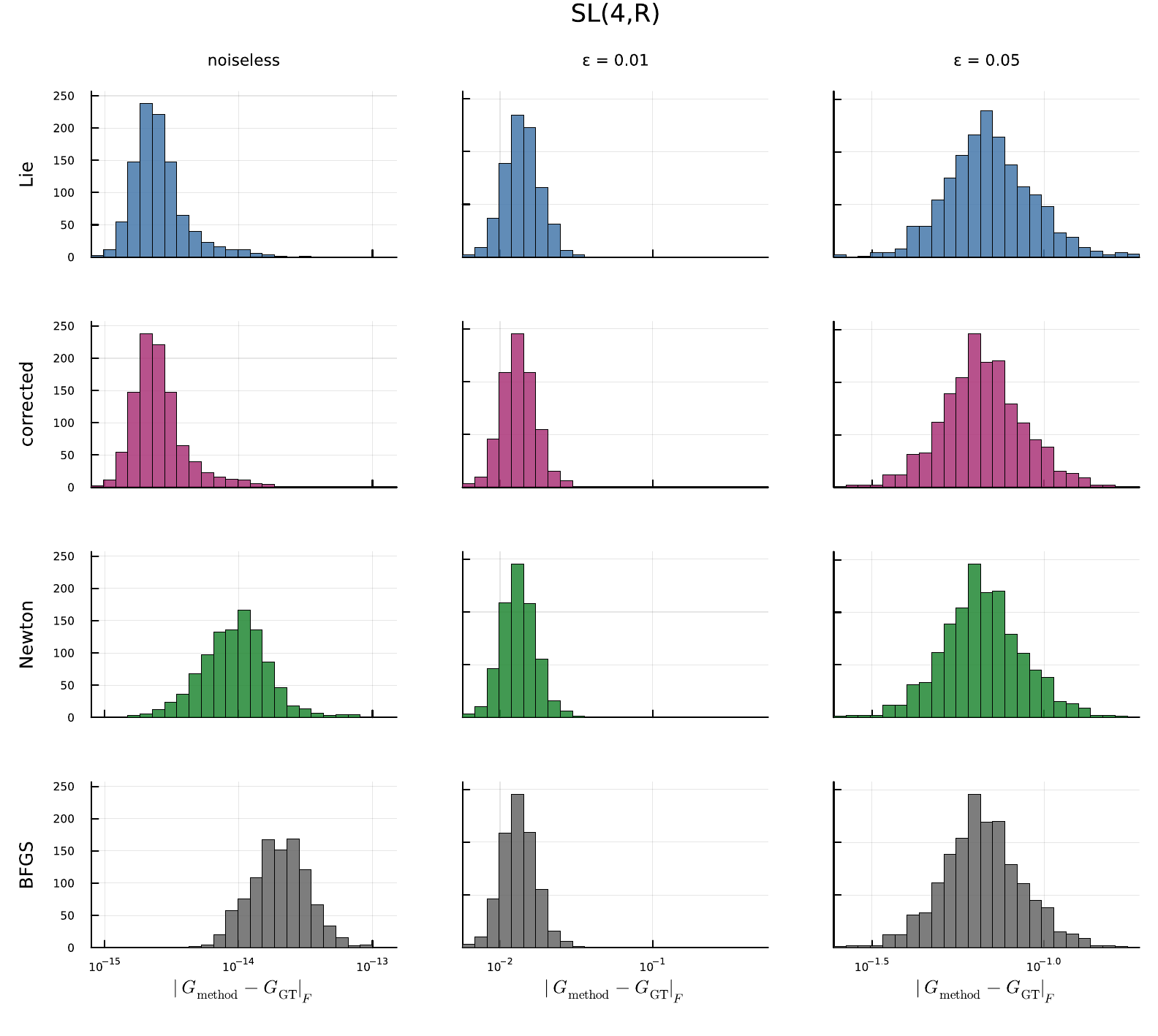}
\caption{Recovery error for $SL(4,\mathbb R)$.}
\end{figure}

\begin{figure}[p]\centering
\includegraphics[width=\linewidth]{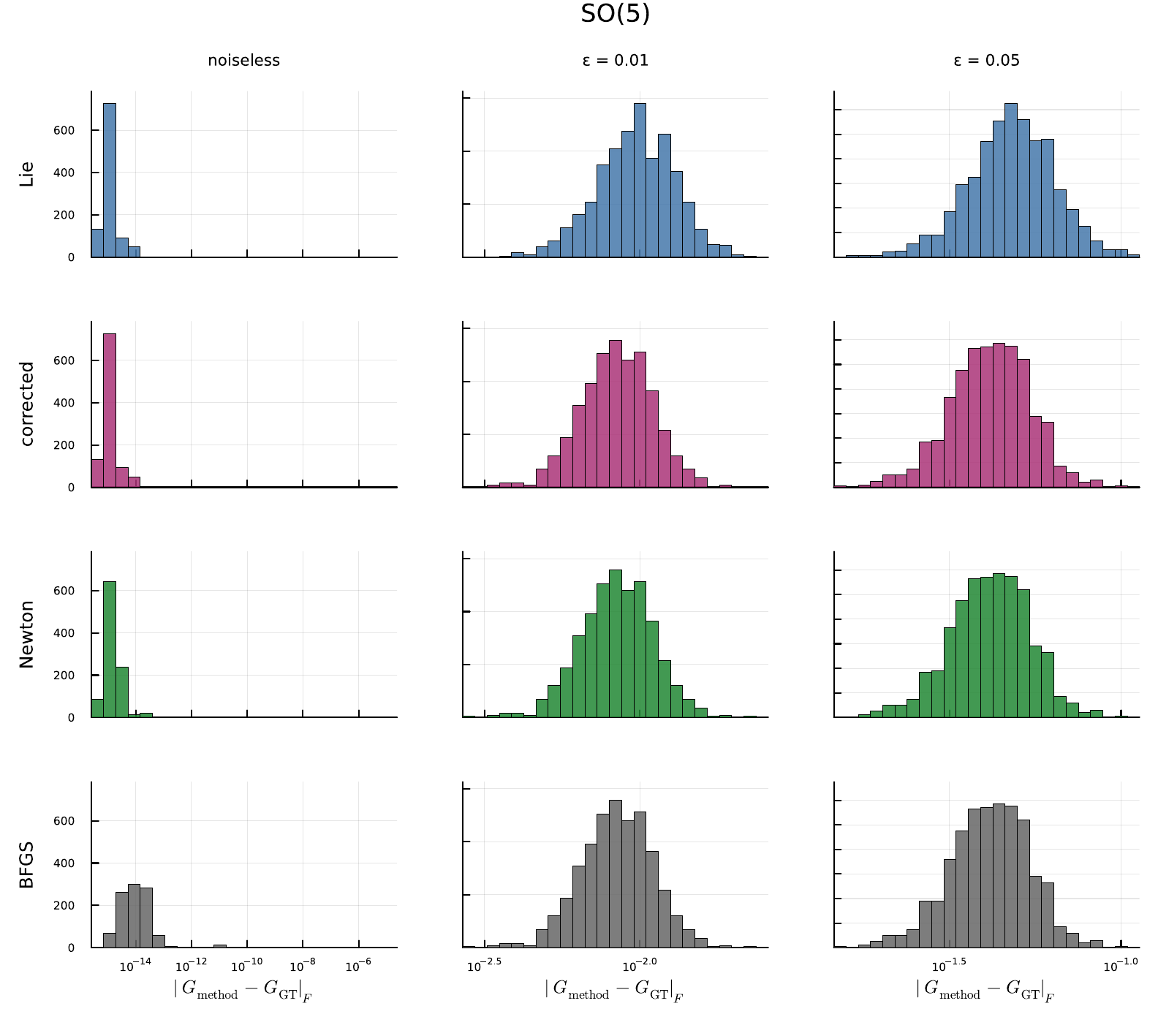}
\caption{Recovery error for $SO(5)$.}
\end{figure}

\begin{figure}[p]\centering
\includegraphics[width=\linewidth]{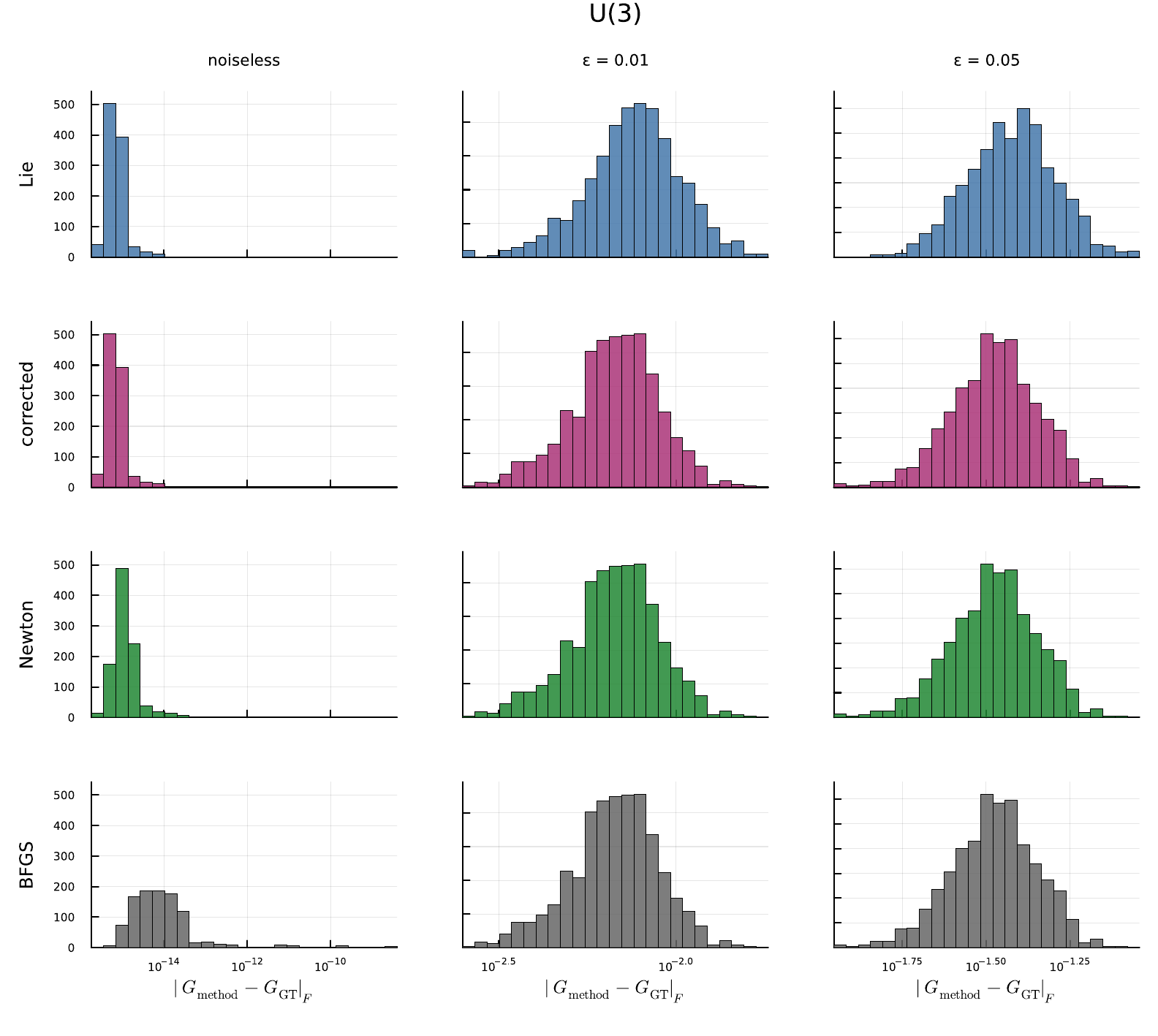}
\caption{Recovery error for $U(3)$.}
\end{figure}

\begin{figure}[p]\centering
\includegraphics[width=\linewidth]{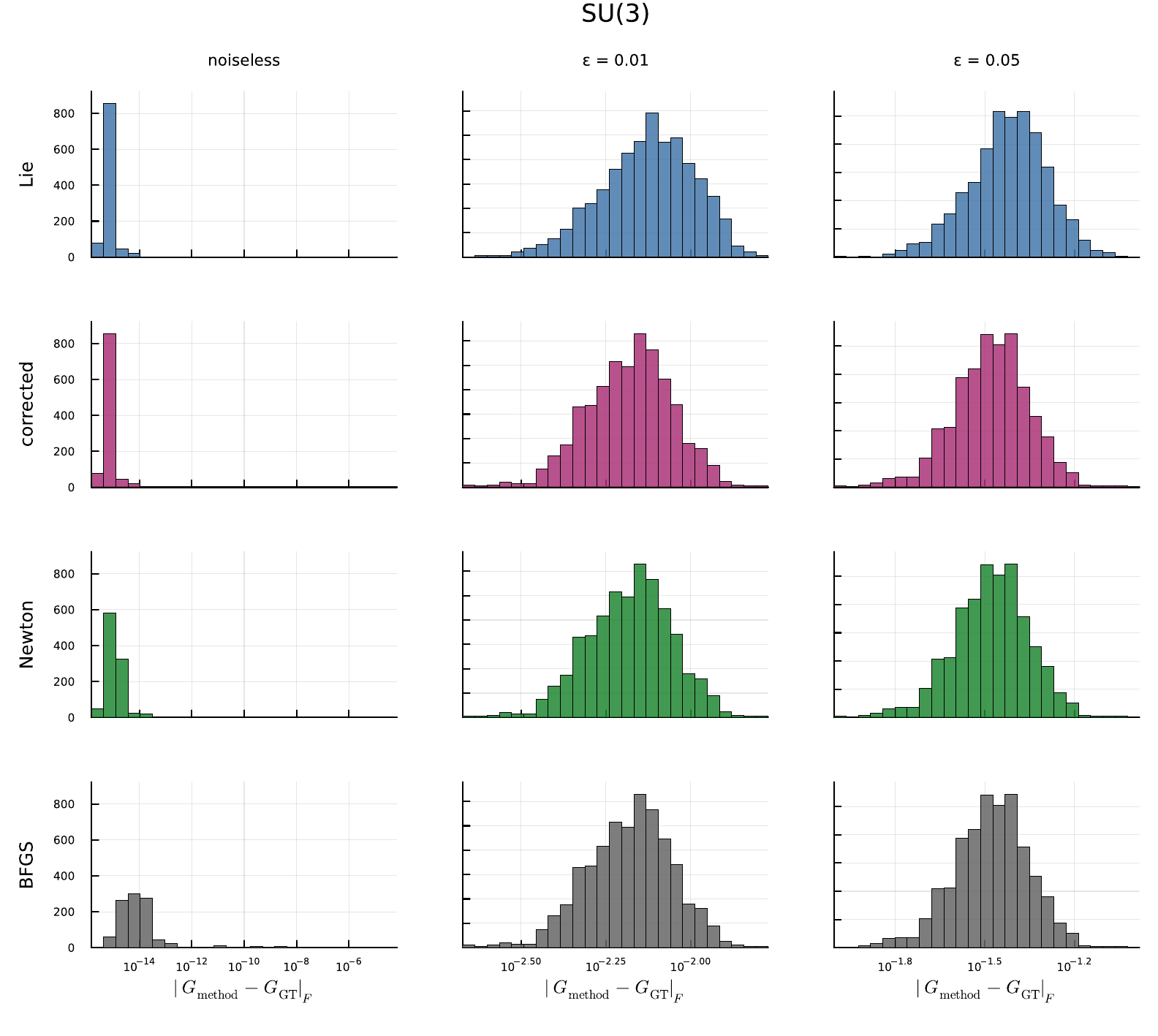}
\caption{Recovery error for $SU(3)$.}
\end{figure}

\begin{figure}[p]\centering
\includegraphics[width=\linewidth]{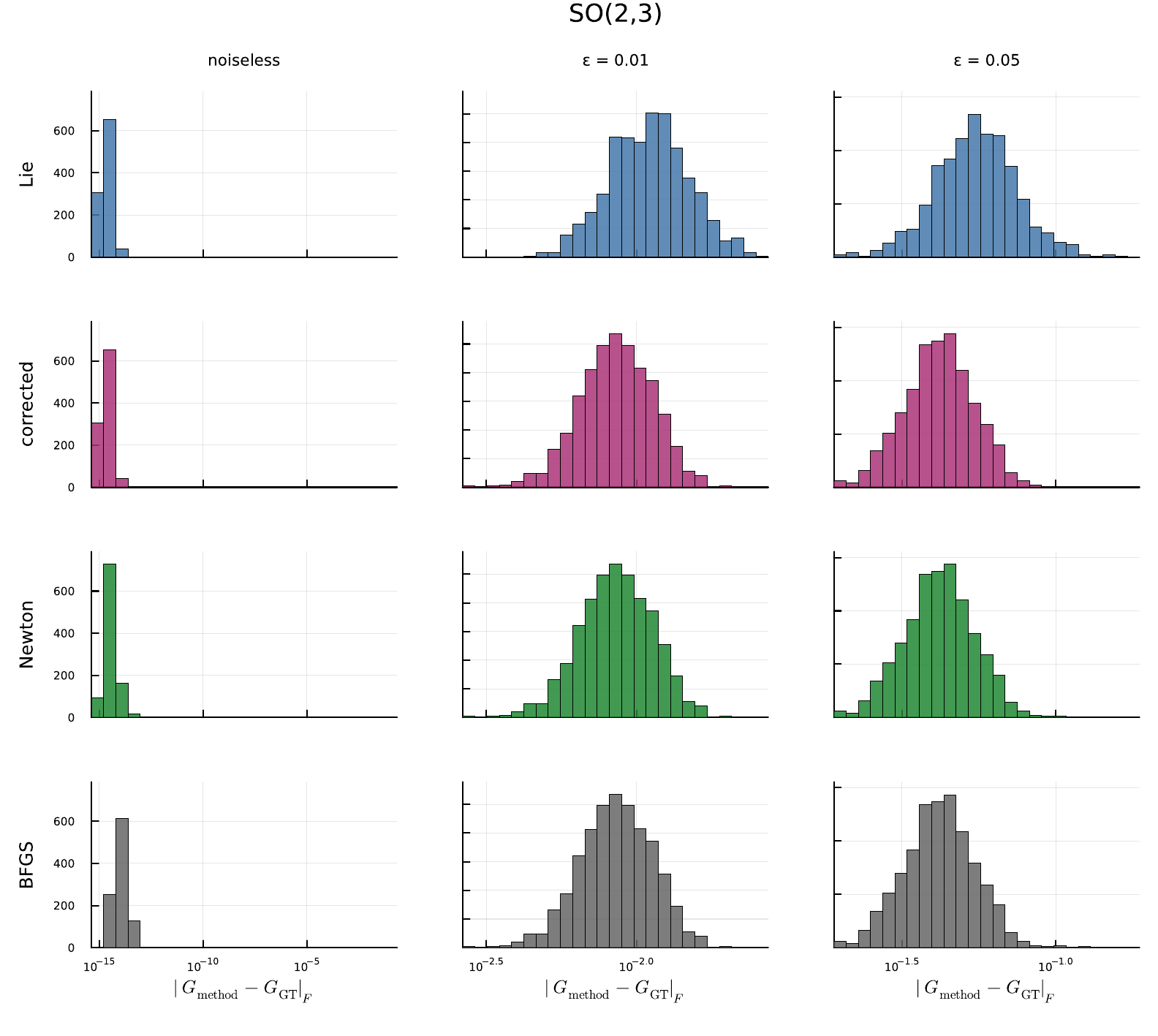}
\caption{Recovery error for $SO(2,3)$.}
\end{figure}

\begin{figure}[p]\centering
\includegraphics[width=\linewidth]{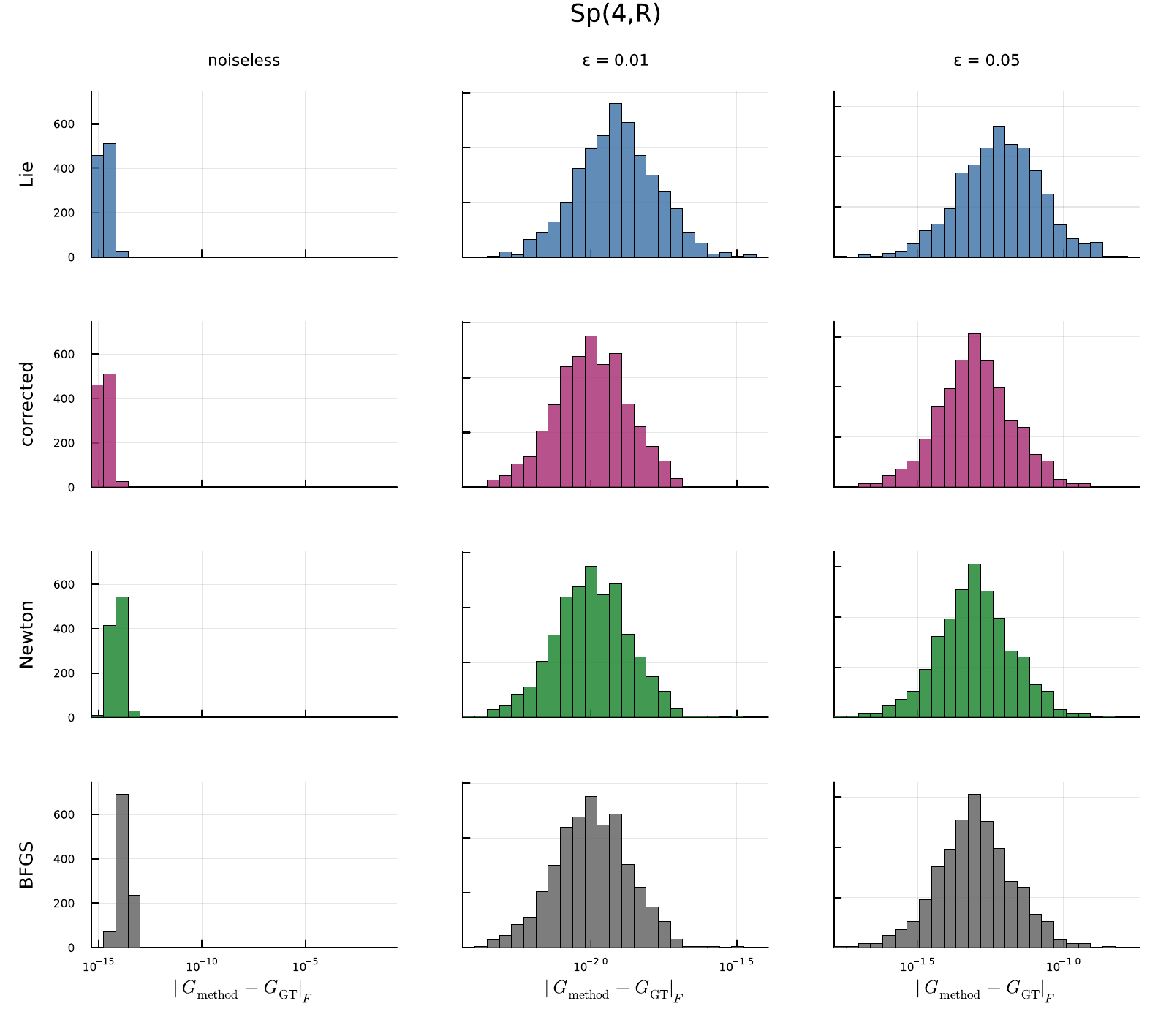}
\caption{Recovery error for $Sp(4,\mathbb R)$.}
\end{figure}

\begin{figure}[p]\centering
\includegraphics[width=\linewidth]{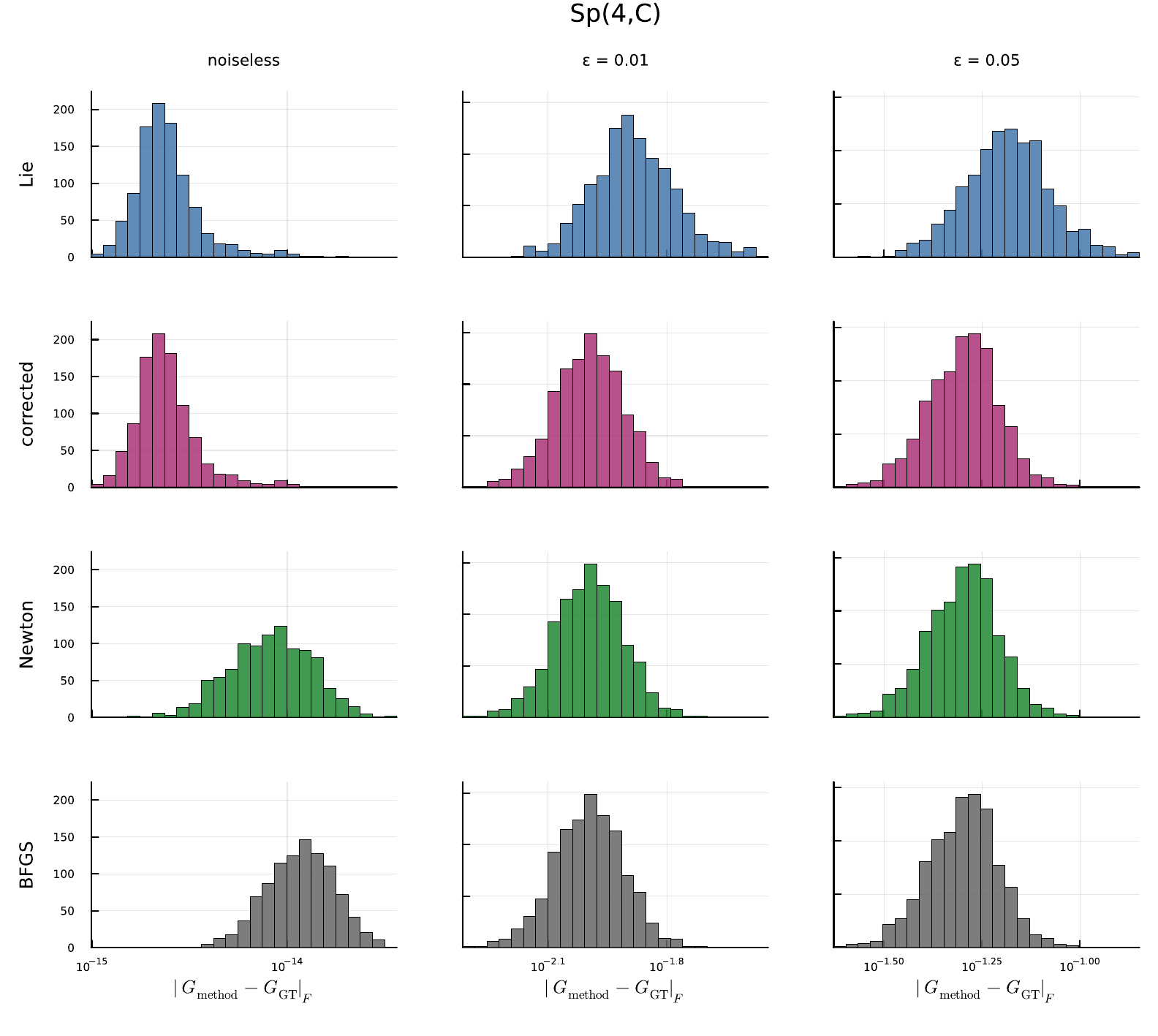}
\caption{Recovery error for $Sp(4,\mathbb C)$.}
\end{figure}

\begin{figure}[p]\centering
\includegraphics[width=\linewidth]{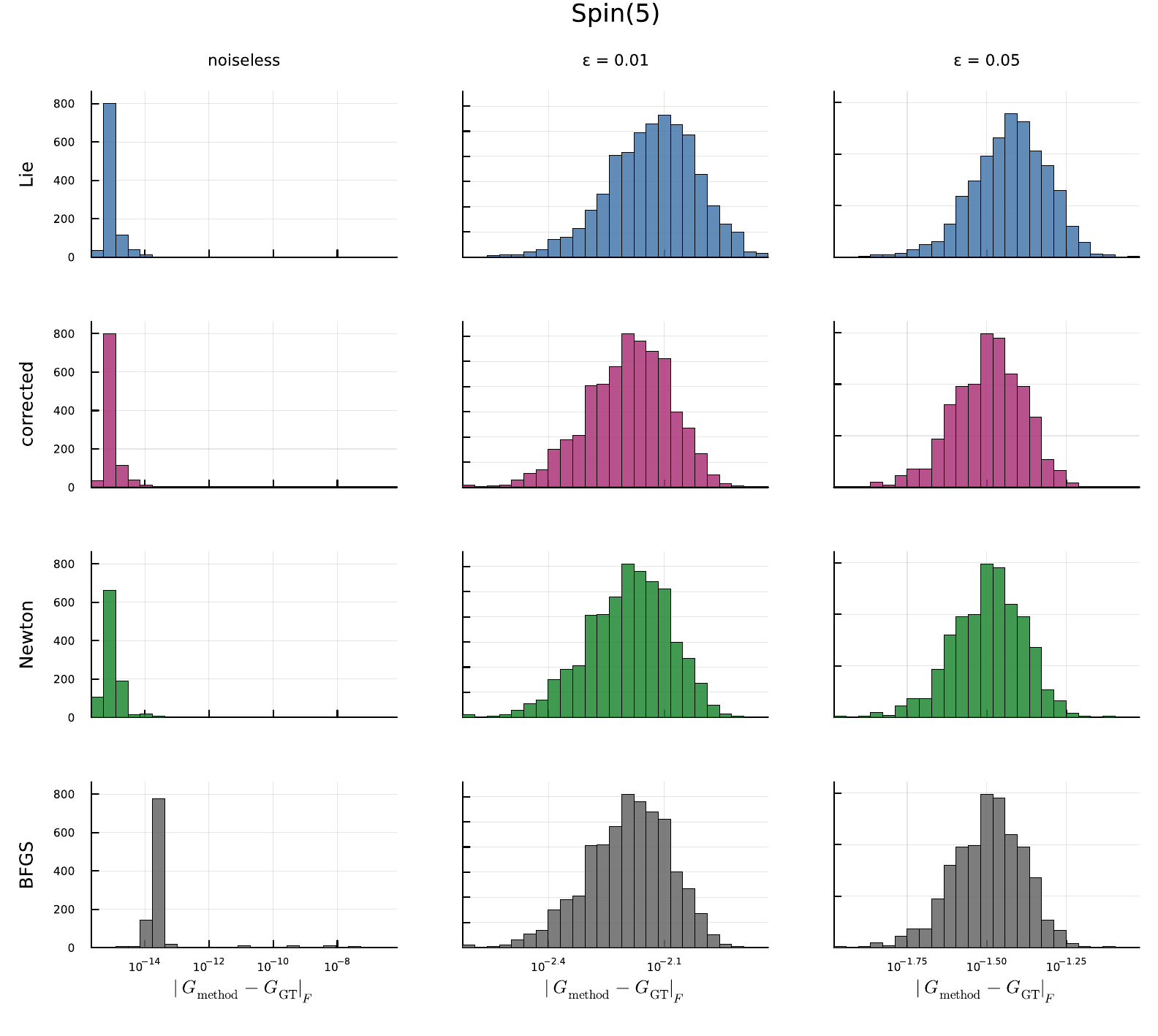}
\caption{Recovery error for $\mathrm{Spin}(5)$.}
\end{figure}

\begin{figure}[p]\centering
\includegraphics[width=\linewidth]{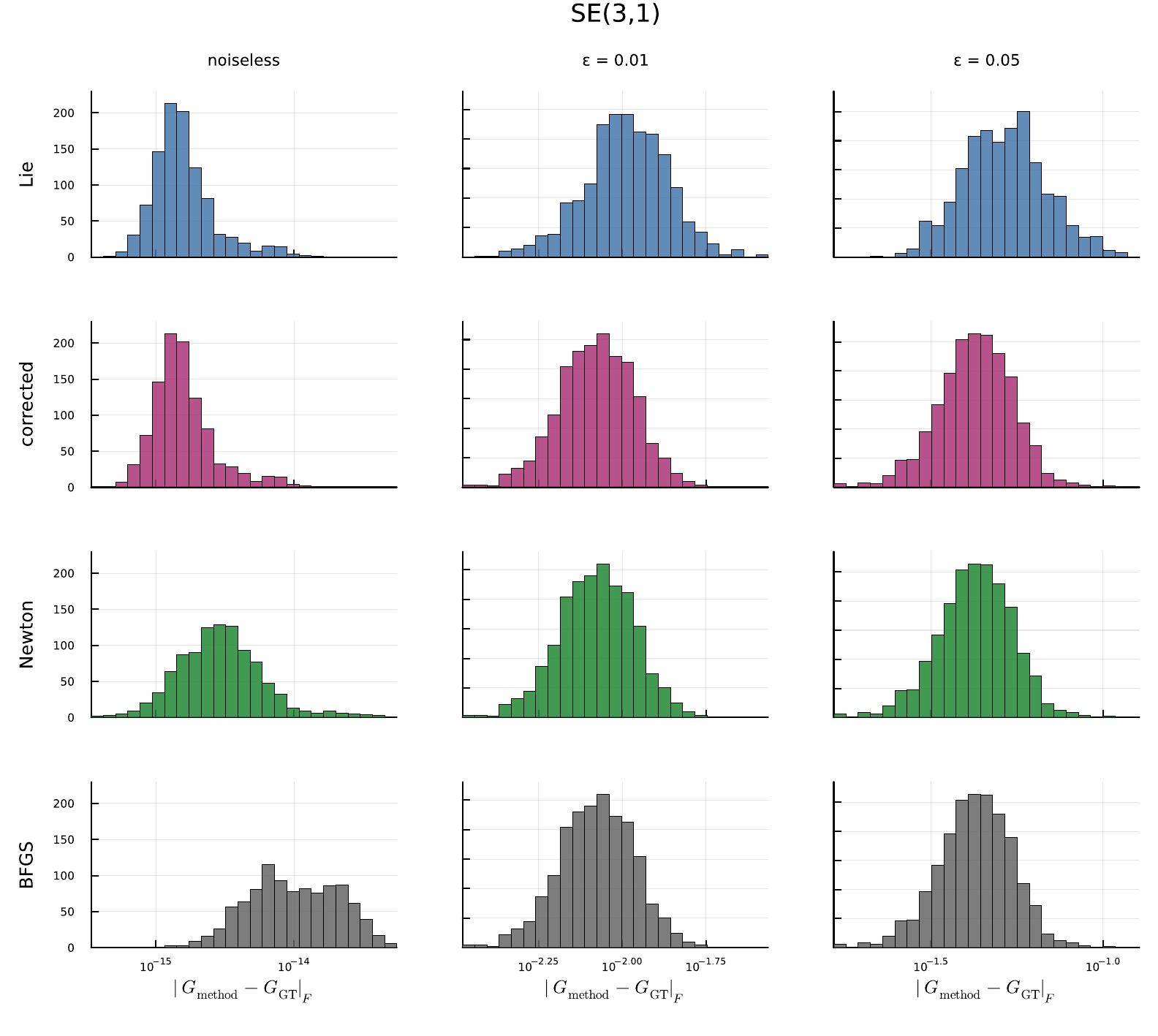}
\caption{Recovery error for $SE(3,1)$, the Poincar\'e group $SO(3,1)^0\ltimes\mathbb R^4$ in homogeneous coordinates.}
\end{figure}

\begin{figure}[p]\centering
\includegraphics[width=\linewidth]{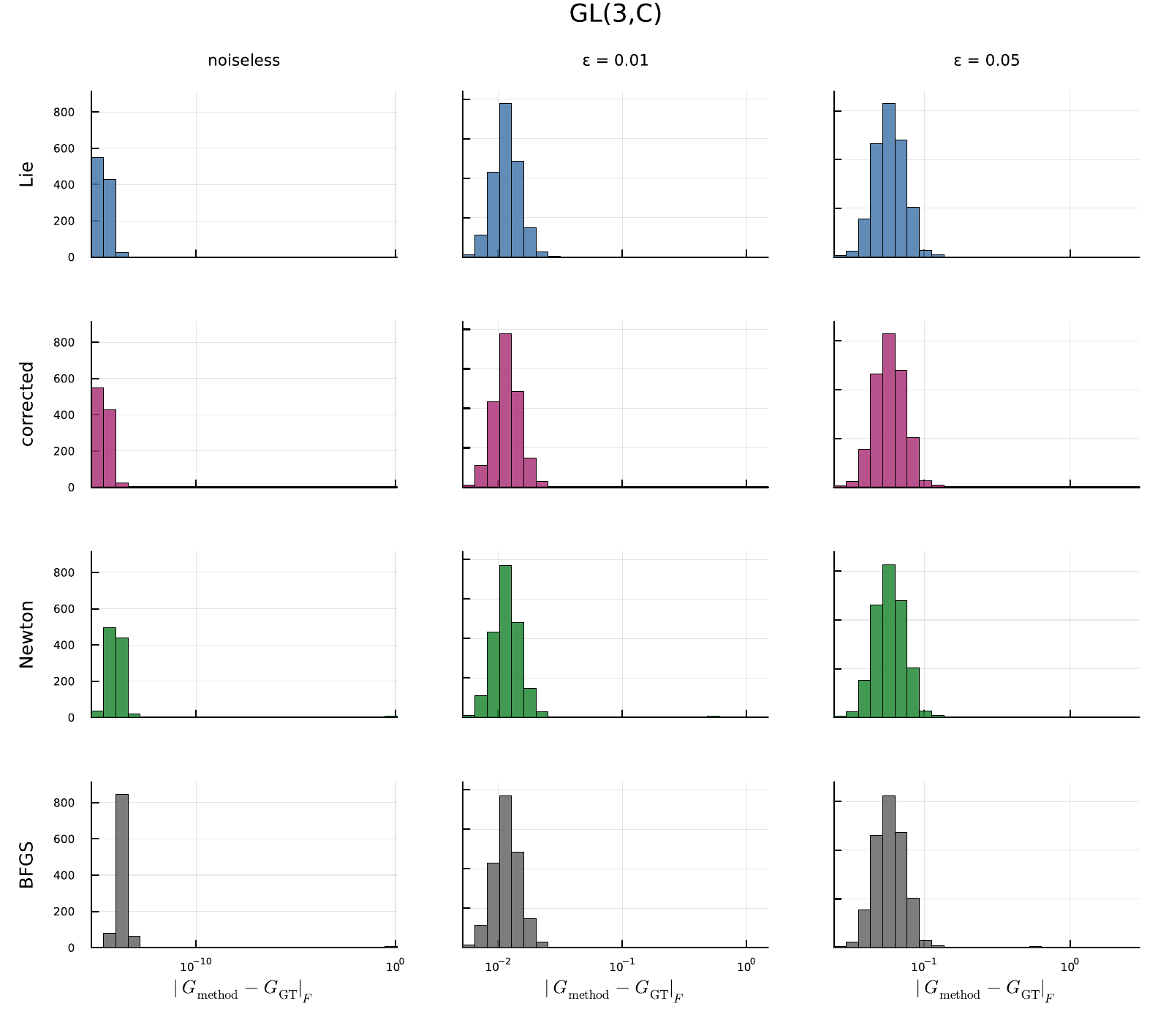}
\caption{Recovery error for $GL(3,\mathbb C)$.}
\end{figure}

\begin{figure}[p]\centering
\includegraphics[width=\linewidth]{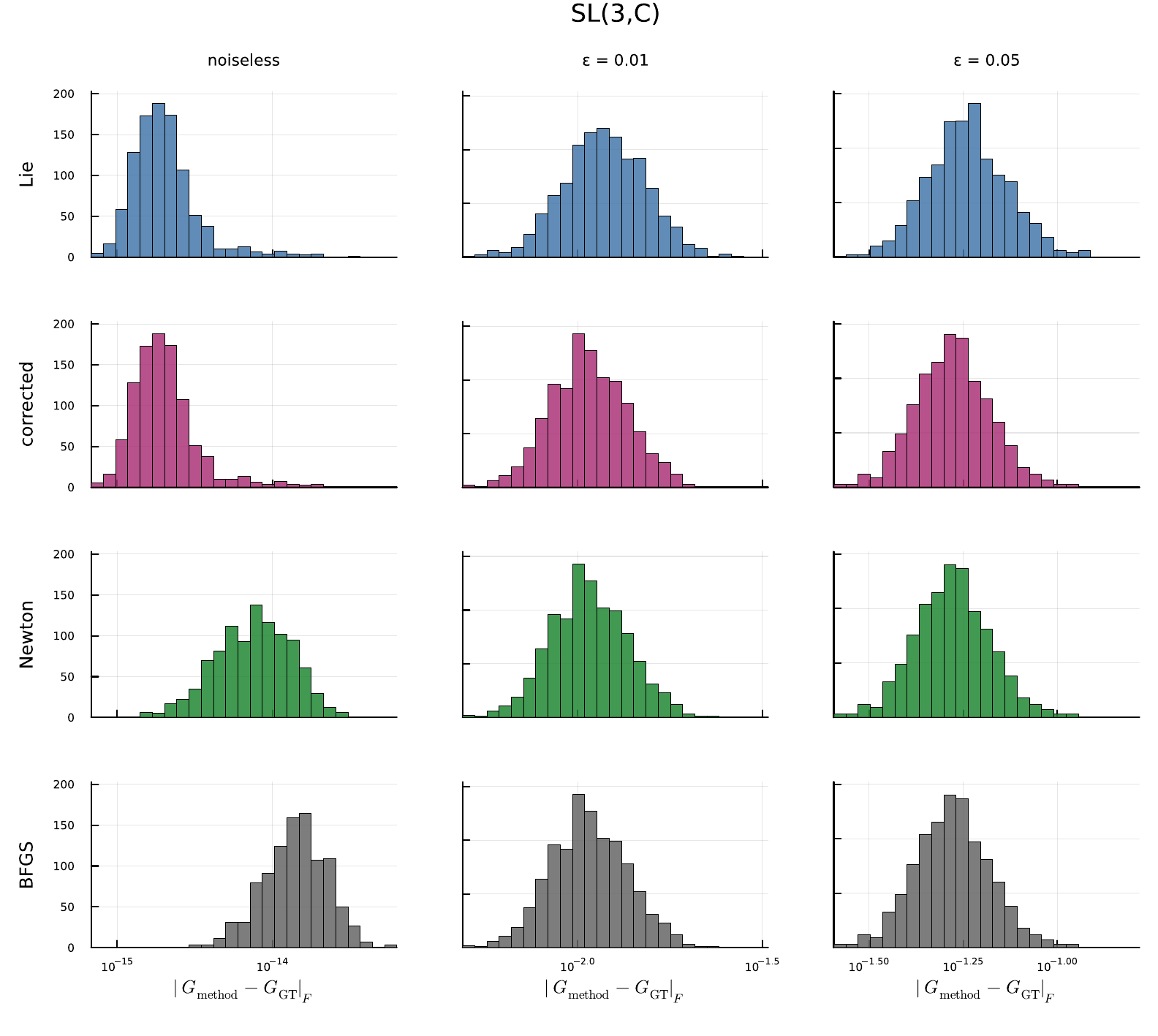}
\caption{Recovery error for $SL(3,\mathbb C)$.}
\end{figure}

\begin{figure}[p]\centering
\includegraphics[width=\linewidth]{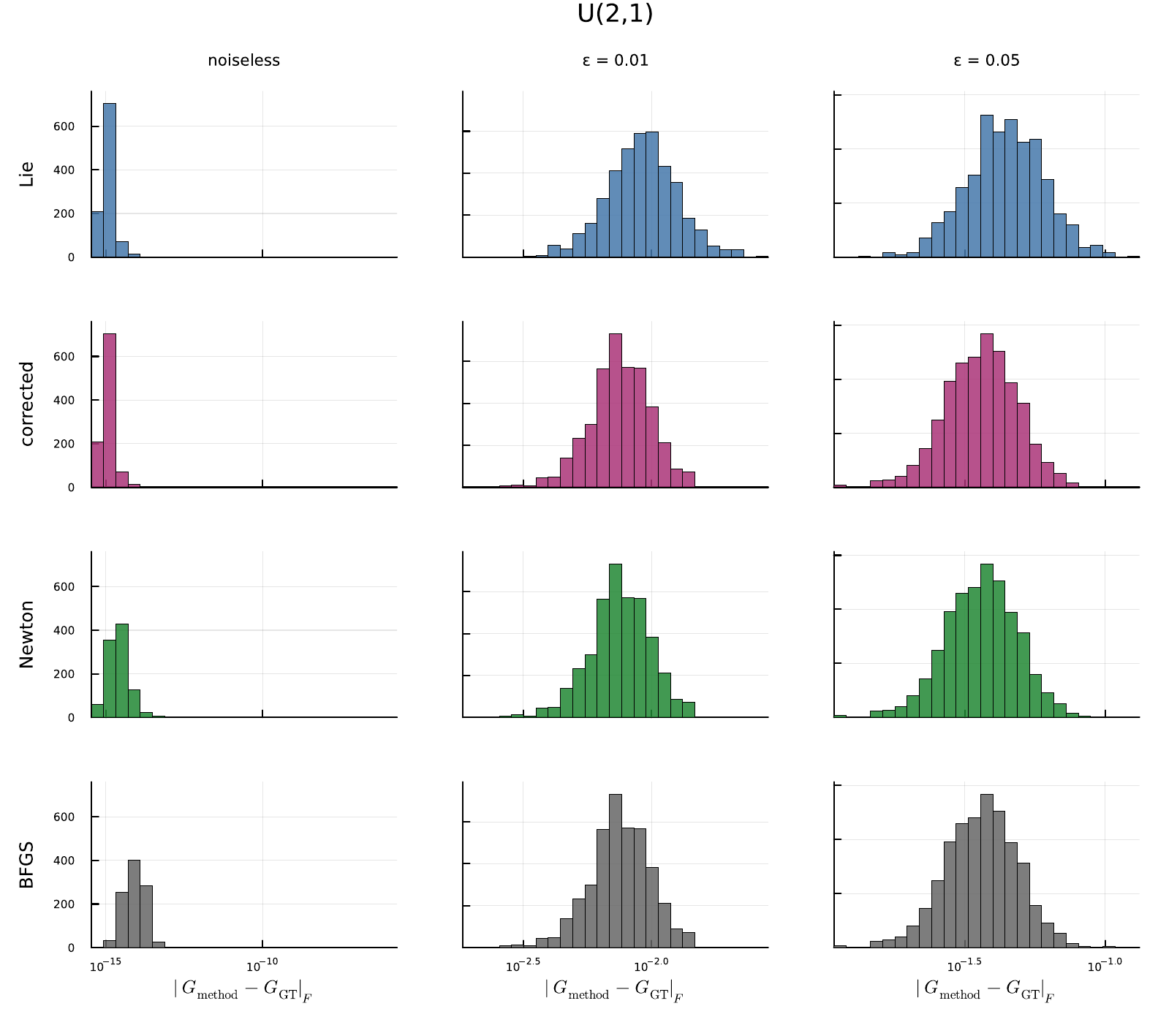}
\caption{Recovery error for $U(2,1)$.}
\end{figure}

\begin{figure}[p]\centering
\includegraphics[width=\linewidth]{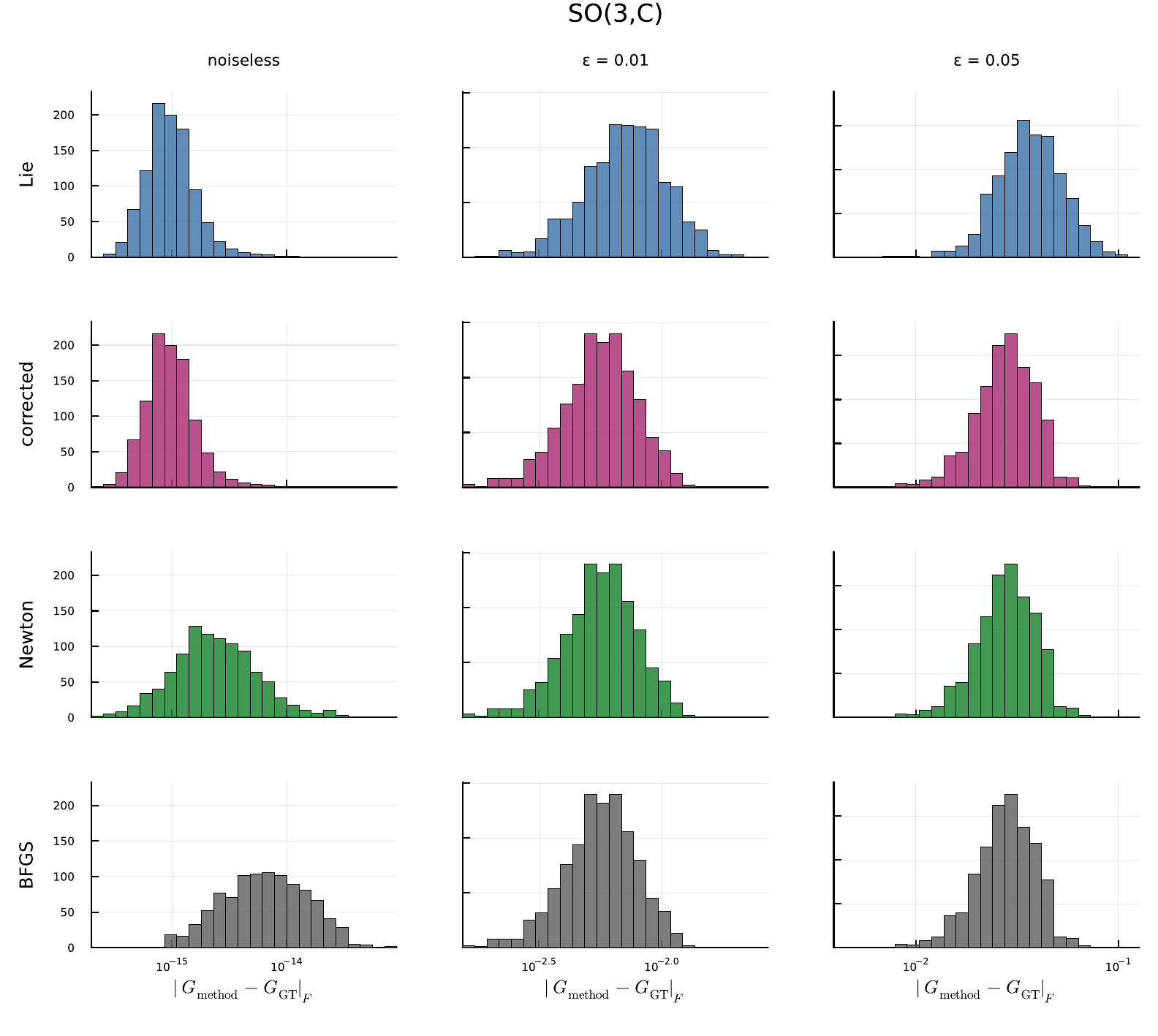}
\caption{Recovery error for $SO(3,\mathbb C)$.}
\end{figure}

\begin{figure}[p]\centering
\includegraphics[width=\linewidth]{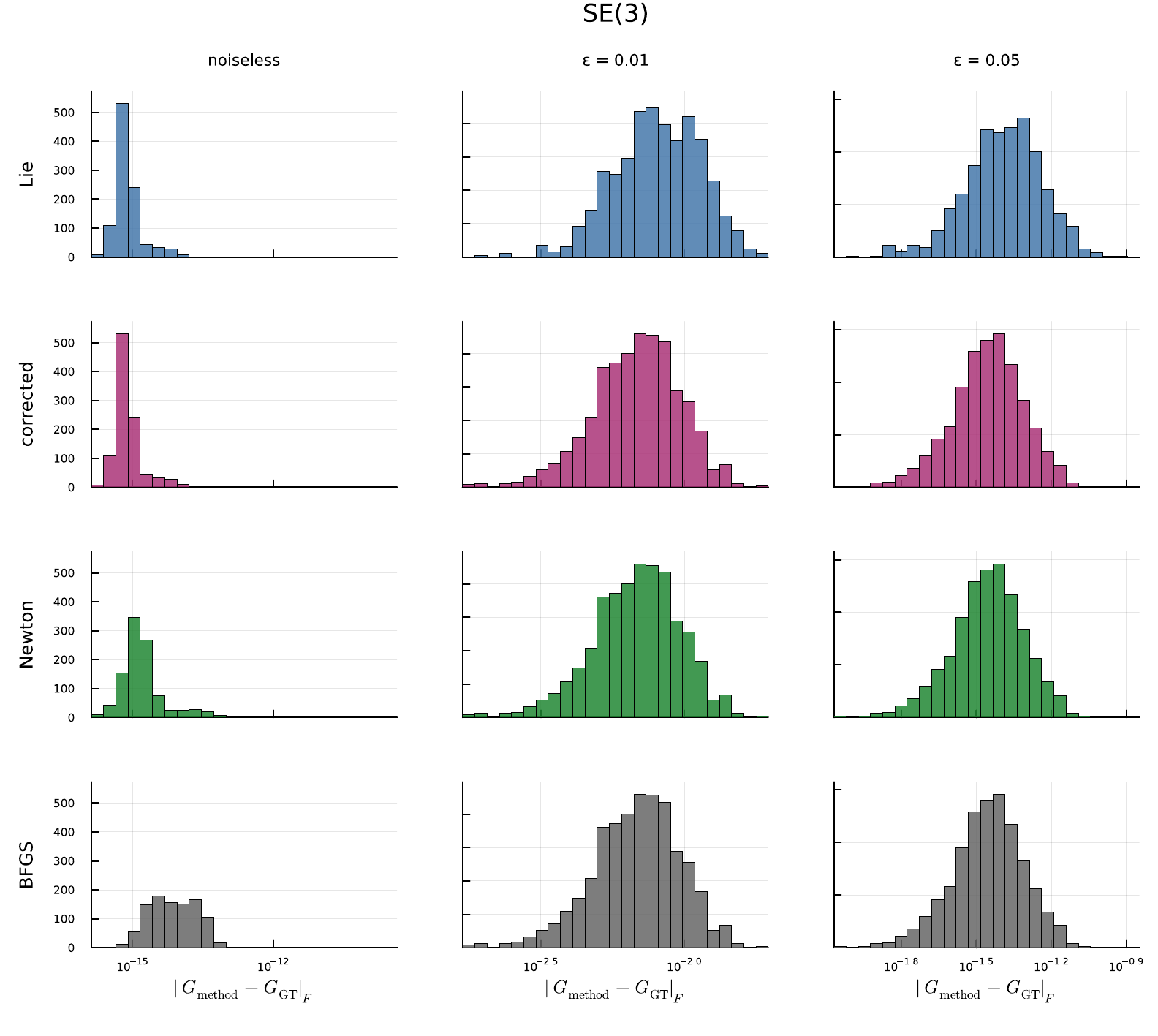}
\caption{Recovery error for $SE(3)$.}
\end{figure}

\begin{figure}[p]\centering
\includegraphics[width=\linewidth]{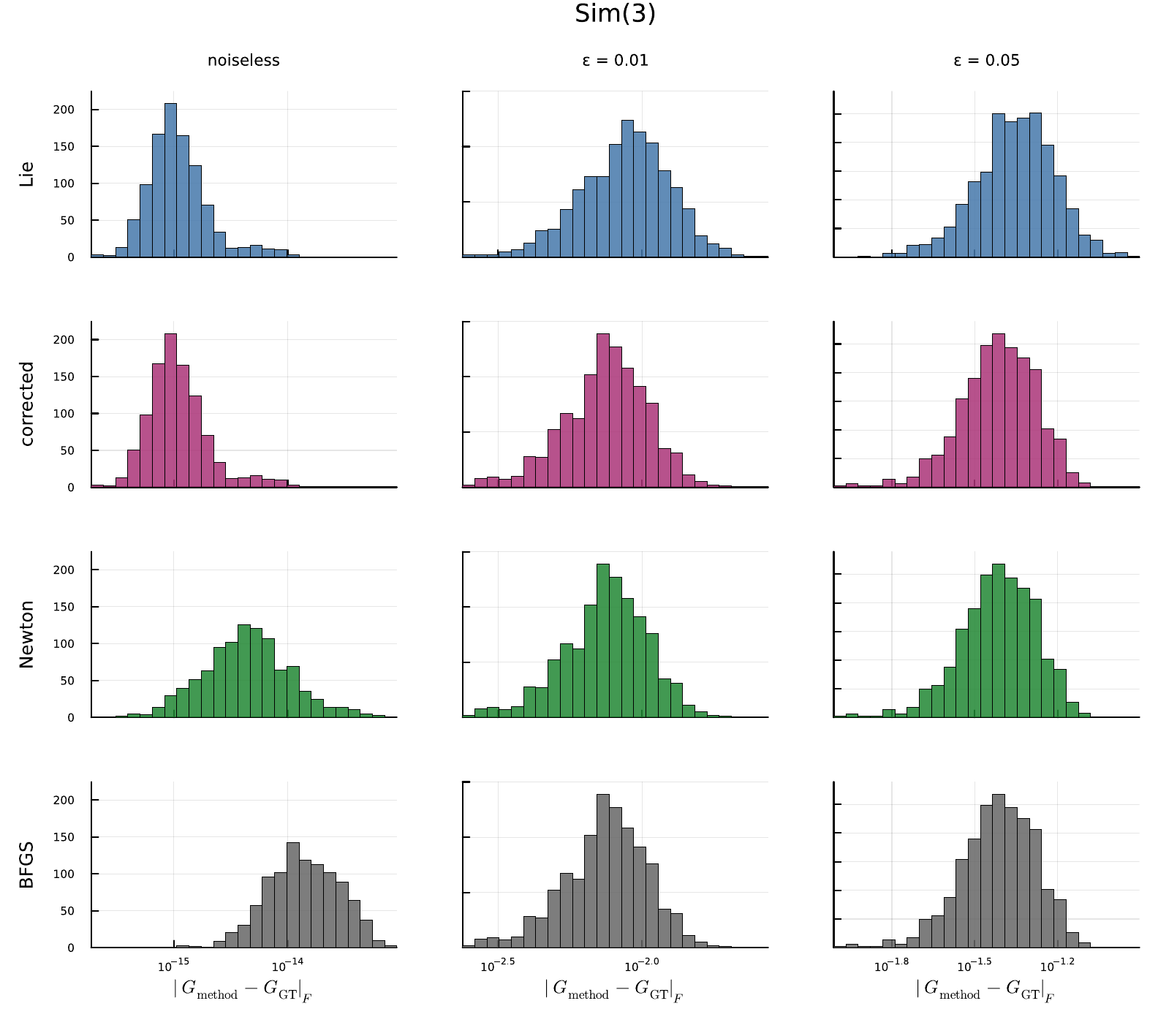}
\caption{Recovery error for $\mathrm{Sim}(3)$.}
\end{figure}

\begin{figure}[p]\centering
\includegraphics[width=\linewidth]{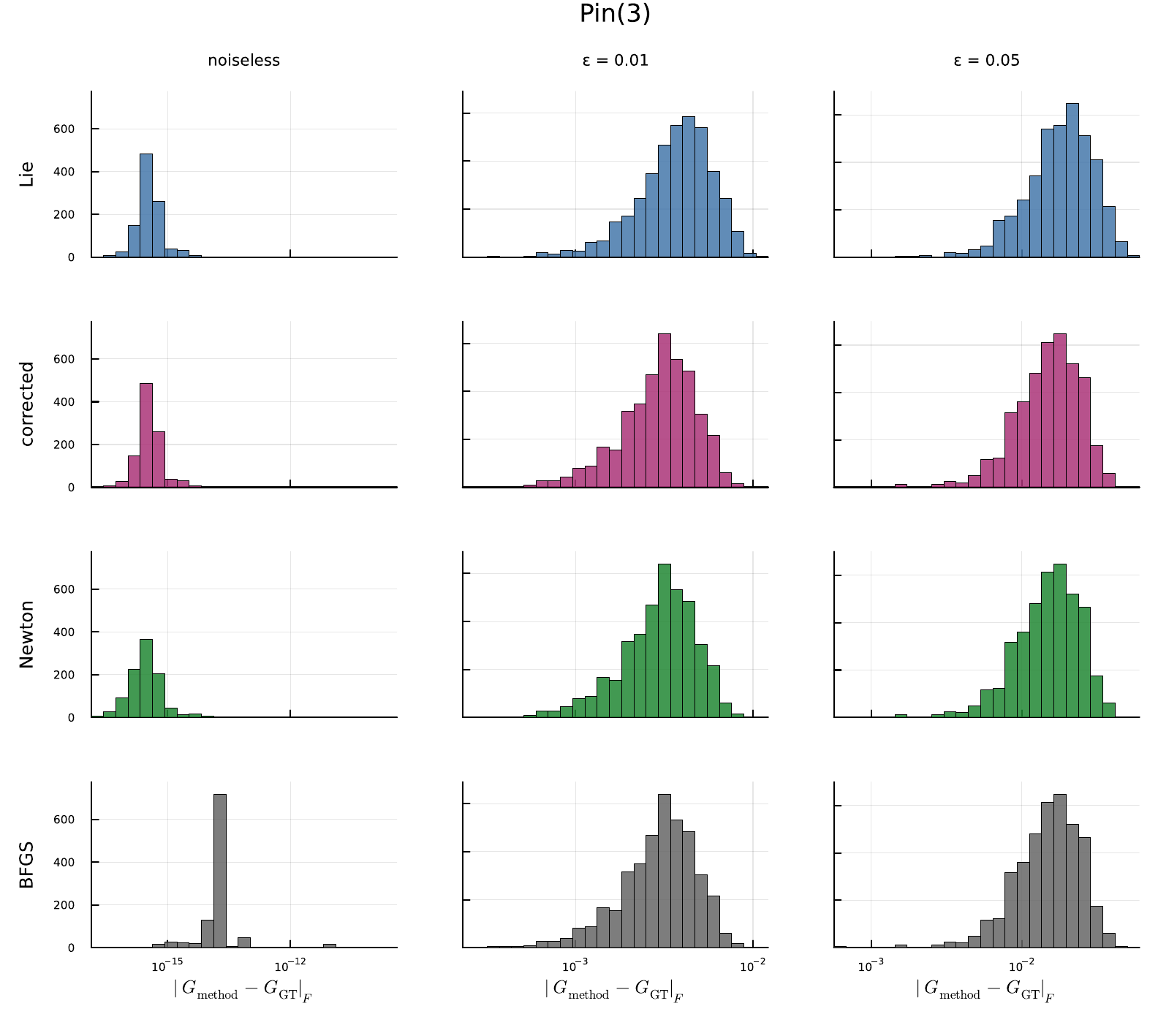}
\caption{Recovery error for $\mathrm{Pin}(3)$.}
\end{figure}

\end{document}